\documentclass[reqno]{amsart}

\usepackage{fullpage,amssymb,dsfont}

\usepackage[all]{xy}

\usepackage{hyperref} 

\usepackage[nobysame,alphabetic,initials,msc-links,lite]{amsrefs}

\DefineSimpleKey{bib}{how}
\renewcommand{\PrintDOI}[1]{%
  \href{http://dx.doi.org/#1}{{\tt DOI:#1}}%
}
\renewcommand{\eprint}[1]{#1}
\BibSpec{book}{%
    +{}  {\PrintPrimary}                {transition}
    +{.} { \PrintDate}                  {date}
    +{.} { \textit}                     {title}
    +{.} { }                            {part}
    +{:} { \textit}                     {subtitle}
    +{,} { \PrintEdition}               {edition}
    +{}  { \PrintEditorsB}              {editor}
    +{,} { \PrintTranslatorsC}          {translator}
    +{,} { \PrintContributions}         {contribution}
    +{,} { }                            {series}
    +{,} { \voltext}                    {volume}
    +{,} { }                            {publisher}
    +{,} { }                            {organization}
    +{,} { }                            {address}
    +{,} { }                            {status}
    +{,} { \PrintDOI}                   {doi}
    +{,} { \PrintISBNs}                 {isbn}
    +{}  { \parenthesize}               {language}
    +{}  { \PrintTranslation}           {translation}
    +{;} { \PrintReprint}               {reprint}
    +{.} { }                            {note}
    +{.} {}                             {transition}
    +{}  {\SentenceSpace \PrintReviews} {review}
}
\BibSpec{article}{%
    +{}  {\PrintAuthors}                {author}
    +{,} { \textit}                     {title}
    +{.} { }                            {part}
    +{:} { \textit}                     {subtitle}
    +{,} { \PrintContributions}         {contribution}
    +{.} { \PrintPartials}              {partial}
    +{,} { }                            {journal}
    +{}  { \textbf}                     {volume}
    +{}  { \PrintDatePV}                {date}
    +{,} { \issuetext}                  {number}
    +{,} { \eprintpages}                {pages}
    +{,} { }                            {status}
    +{,} { \PrintDOI}                   {doi}
    +{,} { \eprint}        {eprint}
    +{}  { \parenthesize}               {language}
    +{}  { \PrintTranslation}           {translation}
    +{;} { \PrintReprint}               {reprint}
    +{.} { }                            {note}
    +{.} {}                             {transition}
    +{}  {\SentenceSpace \PrintReviews} {review}
}
\BibSpec{collection.article}{%
    +{}  {\PrintAuthors}                {author}
    +{,} { \textit}                     {title}
    +{.} { }                            {part}
    +{:} { \textit}                     {subtitle}
    +{,} { \PrintContributions}         {contribution}
    +{,} { \PrintConference}            {conference}
    +{}  {\PrintBook}                   {book}
    +{,} { }                            {booktitle}
    +{,} { \PrintDateB}                 {date}
    +{,} { pp.~}                        {pages}
    +{,} { }                            {publisher}
    +{,} { }                            {organization}
    +{,} { }                            {address}
    +{,} { }                            {status}
    +{,} { \PrintDOI}                   {doi}
    +{,} { \eprint}        {eprint}
    +{}  { \parenthesize}               {language}
    +{}  { \PrintTranslation}           {translation}
    +{;} { \PrintReprint}               {reprint}
    +{.} { }                            {note}
    +{.} {}                             {transition}
    +{}  {\SentenceSpace \PrintReviews} {review}
}
\BibSpec{misc}{%
  +{}{\PrintAuthors}  {author}
  +{,}{ \textit}      {title}
  +{.}{ }             {how}
  +{}{ \parenthesize} {date}
  +{,} { available at \eprint}        {eprint}
  +{,}{ available at \url}{url}
  +{,}{ }             {note}
    +{,} { \PrintDOI}                   {doi}
  +{.}{}              {transition}
}

\numberwithin{equation}{section}

\newtheorem{theorem}{Theorem}[section]
\newtheorem{corollary}[theorem]{Corollary}
\newtheorem{lemma}[theorem]{Lemma}
\newtheorem{proposition}[theorem]{Proposition}
\theoremstyle{remark}
\newtheorem{remark}[theorem]{Remark}
\newtheorem{remarks}[theorem]{Remarks}
\newtheorem{example}[theorem]{Example}
\theoremstyle{definition}
\newtheorem{definition}[theorem]{Definition}

\newcommand{\bp}{\begin{proof}}
\newcommand{\ep}{\end{proof}}

\newcommand{\C}{\mathbb{C}}
\newcommand{\Z}{\mathbb{Z}}
\newcommand{\eps}{\varepsilon}

\newcommand{\CC}{\mathcal{C}}

\newcommand{\NN}{\mathcal{N}}
\newcommand{\PP}{\mathcal{P}}
\newcommand{\un}{\mathds{1}}

\mathchardef\mhyph="2D
\newcommand{\indcat}[1]{{\mathrm{ind}\mhyph#1}}
\newcommand{\indC}{\indcat{\CC}}
\newcommand{\Dcen}{\mathcal{Z}}
\newcommand{\ZC}{{\Dcen(\indC)}}
\newcommand{\ZsC}{{\Dcen_s(\indC)}}
\newcommand{\ZCC}{{\Dcen(\CC)}}
\newcommand\ZHC{{\mathcal Z}_{\mathcal C}}
\newcommand{\tB}{ B}
\newcommand{\Zreg}{Z_{\mathrm{reg}}}
\newcommand{\opos}{\mathrm{op}}
\newcommand{\norm}[1]{\left\|#1\right\|}

\newcommand{\Pol}{\mathcal{O}}
\newcommand{\coD}{\hat{D}}
\newcommand{\hlf}{{1/2}}
\renewcommand\Re{\operatorname{Re}}
\DeclareMathOperator{\Mat}{Mat}
\DeclareMathOperator{\Mor}{Mor}
\DeclareMathOperator{\End}{End}
\DeclareMathOperator{\Hom}{Hom}
\DeclareMathOperator{\Tr}{Tr}
\DeclareMathOperator{\tr}{tr}
\DeclareMathOperator{\Irr}{Irr}
\DeclareMathOperator{\Hilb}{Hilb}
\DeclareMathOperator{\Rep}{Rep}
\DeclareMathOperator*{\colim}{colim}
\DeclareMathOperator{\Id}{Id}

\newcommand{\Qmod}{Q\mhyph\mathrm{mod}}
\newcommand{\modQ}{\mathrm{mod}\mhyph Q}
\newcommand{\QmodQ}{Q\mhyph\mathrm{mod}\mhyph Q}
\newcommand{\indQmodQ}{\indcat{\QmodQ}}
\newcommand{\hatQmod}{\hat{Q}\mhyph\mathrm{mod}}
\newcommand{\hatQbimod}{\hat{Q}\mhyph\mathrm{mod}\mhyph\hat{Q}}

\begin{document}

\title[Representation theory for monoidal categories]{Drinfeld center and representation theory for monoidal categories}

\author[S. Neshveyev]{Sergey Neshveyev}

\email{sergeyn@math.uio.no}

\address{Department of Mathematics, University of Oslo, P.O. Box 1053
Blindern, NO-0316 Oslo, Norway}

\thanks{The research leading to these results has received funding
from the European Research Council under the European Union's Seventh
Framework Programme (FP/2007-2013)
/ ERC Grant Agreement no. 307663
}

\author[M. Yamashita]{Makoto Yamashita}

\email{yamashita.makoto@ocha.ac.jp}

\address{Department of Mathematics, Ochanomizu University, Otsuka
2-1-1, Bunkyo, 112-8610 Tokyo, Japan}

\thanks{Supported by JSPS KAKENHI Grant Number 25800058}

\date{January 27, 2015; new version: March 24, 2015; minor corrections: May 25, 2016}

\begin{abstract}
  Motivated by the relation between the Drinfeld double and
  central property (T) for quantum groups, given a rigid C$^*$-tensor category $\CC$ and a unitary half-braiding on
  an ind-object, we construct a $*$-representation of the fusion
  algebra of $\CC$.
  This allows us to present an alternative approach to recent results of Popa and Vaes, who defined C$^*$-algebras of monoidal categories and introduced property (T) for them.  As an
  example we analyze categories $\CC$ of Hilbert bimodules over a
  II$_1$-factor.  We show that in this case the Drinfeld center is
  monoidally equivalent to a category of Hilbert bimodules over
  another II$_1$-factor obtained by the Longo--Rehren
  construction.  As an application, we obtain an alternative proof of the result of Popa and Vaes stating that property (T) for the category defined by an extremal finite index subfactor $N\subset M$ is
  equivalent to Popa's property~(T) for the corresponding SE-inclusion of II$_1$-factors.

 In the last part of the paper we study M\"uger's notion of weakly monoidally Morita equivalent categories and analyze the behavior of our constructions under the equivalence of the corresponding Drinfeld centers established by Schauenburg. In particular, we prove that   property (T) is invariant under weak monoidal Morita equivalence.
\end{abstract}

\maketitle

\section{Introduction}
\label{sec:introduction}

In this paper, we explore the relation between positive definite
functions on rigid C$^*$-tensor categories and their Drinfeld
centers. Our inspiration for seeking such a relation comes from recent
developments in the so called central approximation properties of
discrete quantum groups. In particular, we give a categorification of
the correspondence between the completely positive central
functions on discrete quantum groups and the positive linear functionals on the character
algebra constructed from the spherical representations of the Drinfeld
double, as observed in~\cite{MR3238527}.
Such a categorification has been already obtained in the recent work of Popa and
Vaes~\cite{MR3406647}, and similarly to their work our main motivation is to understand approximation properties of monoidal categories. In this respect the paper can be seen as a natural continuation of our previous
work~\cite{arXiv:1405.6572}, where we studied amenability of monoidal
categories.

In fact, a connection between approximation properties of monoidal
categories and their Drinfeld centers has already appeared in
subfactor theory, although in a disguised form. Ocneanu introduced the
notion of asymptotic inclusion based on iterated basic extensions and
studied the associated $3$-dimensional topological quantum field
theory for finite depth
subfactors~\citelist{\cite{MR996454}}.  Subsequently its relation to
the Drinfeld center was clarified through the work of
Evans--Kawahigashi~\cite{MR1316301}, Longo--Rehren~\cite{MR1332979},
Izumi~\cite{MR1782145}, and
M\"{u}ger~\citelist{\cite{MR1966524}\cite{MR1966525}}, to name a
few.  In a related direction, Popa introduced the notion of symmetric
enveloping algebra $M \boxtimes_{e_N} M^\opos$ associated with a
subfactor $N \subset M$ as a byproduct of his celebrated
classification program~\cite{MR1302385}. This notion specializes to
asymptotic inclusion in the finite depth case but has a better
universality property for infinite depth and non-irreducible
subfactors.

What arises from Popa's work is the principle that
approximation properties of the combinatorial data encoding the
original subfactor correspond to approximation properties of the
SE-inclusion $M\bar \otimes M^\opos \subset M \boxtimes_{e_N} M^\opos$
formulated in the language of Hilbert bimodules (correspondences). In
particular, based on the general theory of correspondences and
rigidity developed in~\cite{popa-corr-preprint}, Popa introduced the
notion of property (T) for subfactors~\cite{MR1729488} as an
antithesis of amenability, which played a central role in the
classification. Since any finitely generated rigid C$^*$-tensor category can
be realized as a part of the standard invariant~\cite{MR1334479}, it
is natural to try to borrow from this theory to formulate various
notions for C$^*$-tensor categories. Popa and Vaes~\cite{MR3406647} achieved this by axiomatizing the notions
of completely positive multipliers and completely bounded multipliers on rigid
C$^*$-tensor categories, and by relating them to properties of
asymptotic inclusions. They defined C$^*$-completions of the fusion algebras by considering what they called admissible representations, which are characterized by the property that their matrix coefficients are completely positive multipliers, and using these completions defined analogues of various approximation properties for such categories.

Aiming for a more direct connection with the Drinfeld center, we start
our work by considering unitary half-braidings on ind-objects of a
rigid C$^*$-tensor category~$\CC$. For every such half-braiding, we
construct a $*$-representation of the fusion algebra. This
construction is an analogue of the restriction of spherical
representations of the Drinfeld double of a quantum group to the
subspace of spherical vectors. As an example if we consider the category
of $\Gamma$-graded Hilbert spaces for a discrete group $\Gamma$, then
we recover all unitary representations of $\Gamma$. In this scheme
there is a distinguished half-braiding, which we call the regular
half-braiding, giving rise to the regular representation of the fusion
algebra.  Morally, it corresponds to the algebra object representing
the forgetful functor on the Drinfeld center which appeared in the
work of
Brugui\'{e}res--Virelizier~\citelist{\cite{MR2355605}\cite{MR3079759}}
and Brugui\'{e}res--Natale~\cite{MR2863377} in the framework of fusion
categories.  The representations defined by the half-braidings lead to
a completion $C^*(\CC)$ of the fusion algebra to a C$^*$-algebra. As in~\cite{MR3406647}, the algebra $C^*(\CC)$ can then be
used to formulate analogues of various approximation properties of
groups for C$^*$-tensor categories. In particular, by considering the
isolation property of the representation defined by the unit object,
we get a natural definition of property (T).

In order to illustrate the general theory we consider categories $\CC$
of Hilbert bimodules over a II$_1$-factor~$M$. Expanding on Izumi's
work~\cite{MR1782145} on the Longo--Rehren inclusion, we show that in
this case the Drinfeld center is monoidally equivalent to a category
of Hilbert $\tB$-bimodules, where $A\subset\tB$ is the Longo--Rehren
inclusion associated with $\CC$. As a particular case, this gives an equivalence between the category of representations of $C^*(\CC)$ and the category of Hilbert $B$-bimodules generated by $A$-central vectors, which has been already established by Popa and Vaes in their approach for categories arising from extremal finite index subfactors.
This equivalence can be used to connect the categorical notion of property (T)
to Popa's property (T) of the Longo--Rehren inclusion. In the case when $\CC$ is the category associated with a finite index
extremal subfactor $N\subset M$, this means that property (T) for~$\CC$ is equivalent to Popa's property~(T) for $N\subset M$. As has been observed in~\cite{MR3406647}, combined
with the results of~\cite{arXiv:1410.6238} this, in turn, can be used to
construct subfactors with property~(T) that do not come from discrete
groups.

\smallskip

The results described above had been obtained when we received preprint~\cite{MR3406647} by Popa and Vaes. A natural task was then to compare the two approaches.
It turned out that they are equivalent, and even the classes of representations of the fusion algebras are the same, so that a representation is admissible in the sense of Popa and Vaes if and only if it is defined by a unitary half-braiding. In light of this, our definition of property (T)
is a reformulation of theirs. The approaches naturally complement each other and
have their own advantages. For example, in our setting it is almost
immediate that for representation categories of compact quantum groups
the completion of the fusion algebra coincides with the one obtained
by embedding it in the Drinfeld double~\cite{MR3238527}. On the other
hand, for general categories it is more difficult to see in our
approach that the regular representation of the fusion algebra extends
to its C$^*$-algebra completion. We also mention that soon after both papers were posted, yet another alternative approach to representation theory of monoidal categories was suggested by Ghosh and C.~Jones~\cite{MR3447719}.

\smallskip

In the last section we study weakly monoidally Morita equivalent categories. This notion was
introduced by M\"{u}ger~\cite{MR1966524}. A prototypical example is the categories of Hilbert bimodules
over a factor and its finite index subfactor.  As was pointed out by M\"{u}ger, a result of Schauenburg~\cite{MR1822847} implies that weakly Morita equivalent categories have monoidally equivalent Drinfeld
centers.  This does \emph{not} imply that the corresponding fusion algebras are Morita equivalent, and the precise relation between these algebras will be discussed elsewhere. What we prove in the present paper, is that the property of weak containment of the unit object is preserved under Schauenburg's equivalence, which allows us to compare approximation properties of the original categories. In particular, we show that property (T) is invariant under weak monoidal Morita equivalence.

\medskip\noindent
{\bf Acknowledgement.} We are grateful to Sorin Popa and Stefaan Vaes for fruitful correspondence and in particular for informing us about their work and for their interest in ours.

\bigskip

\section{Preliminaries}
\label{sec:preliminaries}

\subsection{C\texorpdfstring{$^*$}{*}-tensor categories}
\label{sec:c-tensor-categories}

The main object of our study is rigid C$^*$-tensor categories, and in general we keep the conventions of~\cite{arXiv:1405.6572}, see also~\cite{MR3204665} for the proofs. For the convenience of the reader let us summarize the basic terminology.

A \emph{C$^*$-category} $\CC$ is a linear category over the complex number field $\C$, endowed with Banach space norms on the morphism sets $\CC(X, Y)$ and a conjugate linear anti-multiplicative involution $\CC(X, Y) \to \CC(Y, X)$, $T \mapsto T^*$ satisfying the C$^*$-identity $\norm{T^* T} = \norm{T}^2 = \norm{T T^*}$. We always assume that a C$^*$-category is closed under taking subobjects, so that any projection in the C$^*$-algebra $\CC(X) = \CC(X, X)$ corresponds to a subobject of $X$. We also assume that $\CC$ is closed under finite direct sums. A C$^*$-category is called \emph{semisimple} if its morphism sets are finite dimensional. In such categories one can always take a decomposition of an object~$X$ into a direct sum of simple objects using minimal projections in the finite-dimensional C$^*$-algebra $\CC(X)$. A \emph{unitary functor}, or a \emph{C$^*$-functor}, between C$^*$-categories is a linear functor $F$ compatible with involutions: $F(T^*) = F(T)^*$.

A \emph{C$^*$-tensor category} is a C$^*$-category $\CC$ endowed with a bifunctor $\otimes \colon \CC \times \CC \to \CC$, a distinguished object~$\un$, and natural \emph{unitary} isomorphisms
\begin{align*}
  \un \otimes X &\to X \leftarrow X \otimes \un, &
  \Phi\colon (X \otimes Y) \otimes Z &\to X \otimes (Y \otimes Z)
\end{align*}
satisfying the standard set of axioms for monoidal categories. In this paper the unit $\un$ is always assumed to be simple, namely $\CC(\un) \cong \C$. A \emph{unitary tensor functor}, or a \emph{C$^*$-tensor functor}, between C$^*$-tensor categories is a C$^*$-functor $F$ together with a unitary isomorphism $F_0\colon \un \to F(\un)$ and natural unitary isomorphisms $F_2\colon F(X)\otimes F(Y) \to F(X \otimes Y)$ satisfying the standard compatibility conditions. If there is no fear of confusion we use $F$ instead of $(F, F_0, F_2)$ to denote C$^*$-tensor functors. The natural transformation of C$^*$-tensor functors $F \to G$ is defined in the same way as in the case of monoidal functors, but with the additional requirement that the structure morphisms are all unitary. If two C$^*$-tensor categories $\CC$, $\CC'$ are related by C$^*$-tensor functors $F\colon \CC \to \CC'$ and $G\colon \CC' \to \CC$ such that there exist natural isomorphisms of C$^*$-tensor functors $\Id_\CC \to G F$ and $\Id_{\CC'} \to F G$, we say that $\CC$ and $\CC'$ are \emph{unitarily monoidally equivalent}.

A C$^*$-tensor category $\CC$ is said to be \emph{rigid} if every object $X$ in $\CC$ has a dual. Assuming for simplicity that~$\CC$ is strict, this means that there is an object $\bar{X}$ in $\CC$ and morphisms $R \in \CC(\un, \bar{X} \otimes X)$ and $\bar{R} \in \CC(\un, X \otimes \bar{X})$ satisfying the \emph{conjugate equations}
\begin{align*}
  (\iota_{\bar{X}} \otimes \bar{R}^*) (R \otimes \iota_{\bar{X}}) &= \iota_{\bar{X}},&
  (\iota_X \otimes R^*) (\bar{R} \otimes \iota_X) &= \iota_X.&
\end{align*}
Rigid C$^*$-tensor categories (with simple units) are always semisimple.

A rigid C$^*$-tensor category has a good notion of dimension, defined by
$$
d^\CC(X) = \min_{(R, \bar{R})} \norm{R} \norm{\bar{R}},
$$
where $(R, \bar{R})$ runs over the solutions of the conjugate equations for $X$. If there is no fear of confusion we simply write $d(X)$ instead of $d^\CC(X)$. A solution $(R, \bar{R})$ satisfying $\norm{R} = \norm{\bar{R}} = d(X)^{\hlf}$ is called \emph{standard}, and such solutions are unique up to transformations of the form $(R,\bar R)\mapsto((T \otimes \iota)R, (\iota \otimes T)\bar{R})$ for unitary morphisms~$T$.  We often denote a choice of standard solution for the conjugate equations for $X$ as $(R_X, \bar{R}_X)$. As a convenient shorthand, when $(X_i)_{i \in I}$ is a parametrized family of objects in $\CC$, we write $(R_i, \bar{R}_i)$ instead of $(R_{X_i}, \bar{R}_{X_i})$. Similarly, for many other constructions we use index $i$ instead of $X_i$, so for example we write $d_i$ for $d(X_i)$. If the family is self-dual, we also write $\bar i$ for the index corresponding to the dual of $X_i$.

There are several constructions based on standard solutions. For example, if $X, Y \in \CC$, then $(R_X + R_Y, \bar{R}_X + \bar{R}_Y)$ is a standard solution for $X \oplus Y$. Similarly, $((\iota_{\bar{Y}} \otimes R_X \otimes \iota_Y) R_Y, (\iota_{\bar{Y}} \otimes \bar{R}_Y \otimes \iota_Y) \bar{R}_X)$ is a standard solution for $X \otimes Y$. The \emph{categorical trace} is the trace on $\CC(X)$ given by
$$
\Tr_X(T) = R^*_X(\iota \otimes T)R_X = \bar{R}^*_X(T \otimes \iota)\bar{R}_X,
$$
which is independent of the choice of standard solutions $(R_X, \bar{R}_X)$. The second equality above characterizes the standard solutions. The normalized categorical traces are defined by $\tr_X=d(X)^{-1}\Tr_X$.

More generally, we can define \emph{partial categorical traces}
$$
\Tr_X\otimes\iota\colon \CC(X\otimes Y,X\otimes Z)\to\CC(Y,Z)\ \ \text{by}\ \ (\Tr_X\otimes\iota)(T)=(R^*_X\otimes\iota_Z)(\iota_{\bar X} \otimes T)(R_X\otimes\iota_Y),
$$
and similarly define $\iota\otimes\Tr_X$.

For $X, Y \in \CC$ and a choice of standard solutions $(R_X, \bar{R}_X)$ and $(R_Y, \bar{R}_Y)$, we can define a linear anti-multiplicative map $\CC(X, Y) \to \CC(\bar{Y}, \bar{X})$, denoted by $T \mapsto T^\vee$, which is characterized by $(T \otimes \iota) \bar{R}_X = (\iota \otimes T^\vee) \bar{R}_Y$. This map can be also characterized by $(\iota \otimes T) R_X = (T^\vee \otimes \iota) R_Y$ and satisfies $T^{\vee *} = T^{* \vee}$ for the choice of standard solutions $(\bar{R}_X, R_X)$, $(\bar{R}_Y, R_Y)$ for $\bar{X}$ and $\bar{Y}$.

\subsection{Ind-objects in \texorpdfstring{C$^*$}{C*}-categories}
\label{subsec:ind-obj-C*-cat}

Let $\CC$ be a semisimple C$^*$-category. By an \emph{ind-object} of $\CC$ we will mean an inductive system $\{u_{ji}\colon X_i\to X_j\}_{i\prec j}$ in~$\CC$, where $u_{ji}$ are isometries. We define a morphism between two such objects $\{u_{ji}\colon X_i\to X_j\}_{i\prec j}$ and $\{v_{lk}\colon Y_k\to Y_l\}_{k\prec l}$ as a collection $T$ of morphisms $T_{ki}\colon X_i\to Y_k$ in~$\CC$ such that
$$
v^*_{lk}T_{li}=T_{ki}\ \ \text{if}\ k\prec l,\ \
T_{kj}u_{ji}=T_{ki}\ \ \text{if}\ i\prec j,
\ \ \text{and}\ \ \|T\|:=\sup_{k,i}\|T_{ki}\|<\infty.
$$
For ind-objects $X_* = \{u_{ji}\colon X_i\to X_j\}_{i\prec j}$, $Y_* = \{v_{lk}\colon Y_k\to Y_l\}_{k\prec l}$, and $Z_* = \{w_{nm}\colon Z_m\to Z_n\}_{m\prec n}$, the composition of morphisms $T\colon X_* \to Y_*$ and $S\colon Y_* \to Z_*$ is defined by
$$
(ST)_{ni}=\lim_k S_{nk}T_{ki}.
$$
In order to see that this is well-defined we need the following.

\begin{lemma}
For any morphism $T\colon\{u_{ji}\colon X_i\to X_j\}_{i\prec j}\to\{v_{lk}\colon Y_k\to Y_l\}_{k\prec l}$, index $k$ and $\eps>0$ there exists an index $i_0$ such that for all $j\succ i\succ i_0$ we have
$$
\|T_{kj}-T_{ki}u_{ji}^*\|<\eps.
$$
\end{lemma}

\bp Since $T_{kj}u_{ji}=T_{ki}$, the net $\{T_{ki}T^*_{ki}\}_i$ in the finite dimensional C$^*$-algebra $\End_\CC(Y_k)$ is increasing. Since it is also bounded, it converges in norm. Hence we can find $i_0$ such that for all $j\succ i\succ i_0$ we have
$$
\|T_{kj}T^*_{kj}-T_{ki}T^*_{ki}\|<\eps^2.
$$
It remains to observe that
$$
\|T_{kj}-T_{ki}u_{ji}^*\|^2=\|T_{kj}T^*_{kj}-T_{kj}u_{ji}T_{ki}^*-T_{ki}u_{ji}^*T_{kj}^*+T_{ki}u_{ji}^*u_{ji}T_{ki}^*\|= \|T_{kj}T^*_{kj}-T_{ki}T^*_{ki}\|,
$$
which proves the assertion.
\ep

\begin{lemma}
The composition of morphisms of ind-objects is well-defined and is associative.
\end{lemma}

\bp
With $X_*, Y_*, Z_*$ as above, consider morphisms $T\colon X_* \to Y_*$ and $S\colon Y_* \to Z_*$. By the previous lemma, for fixed $i$ and $n$ we can find $k_0$ such that for all $l\succ k\succ k_0$ the morphism $S_{nl}$ is close to $S_{nk}v_{lk}^*$. But then $S_{nl}T_{li}$
is close to
$$
S_{nk}v_{lk}^*T_{li}=S_{nk}T_{ki}.
$$
It follows that the net $\{S_{nk}T_{ki}\}_k$ is convergent. Therefore the composition $ST$ is well-defined.

Assume now we are given one more morphism $R\colon Z_* \to \{t_{qp}\colon W_p\to\ W_q\}_{p\prec q}$. By definition  we have
$$
[R(ST)]_{pi}=\lim_n\lim_k R_{pn}S_{nk}T_{ki}.
$$
As above, by the previous lemma we can find $n_0$ such that $R_{pm}$ is close to $R_{pn}w_{mn}^*$ for $m\succ n\succ n_0$. Similarly, applying the lemma to the morphism $T^*=(T^*_{ki})_{i,k}$, we can find $k_0$ such that $T_{li}$ is close to $v_{lk}T_{ki}$ for $l\succ k\succ k_0$. Then $R_{pm}S_{ml}T_{li}$ is close to $R_{pn}S_{nk}T_{ki}$. It follows that
$$
[R(ST)]_{pi}=\lim_{n,k}R_{pn}S_{nk}T_{ki}.
$$
In a similar way we get the same expression for $[(RS)T]_{pi}$.
\ep

We denote by $\indC$ the category of ind-objects of $\CC$. It is easy to see that this is a C$^*$-category. Moreover, the simple objects of $\CC$ remain simple in $\indC$. In particular, if $X \in \CC$ is irreducible and $Y_*$ is any ind-object, the morphism set $\Mor_\indC(X, Y_*)$ is a Hilbert space, with the inner product such that $(S, T)\iota_X = T^* S$.

\begin{remark}
The morphisms between ind-objects $\{u_{ji}\colon X_i\to X_j\}_{i\prec j}$ and $\{v_{lk}\colon Y_k\to Y_l\}_{k\prec l}$ can be described similarly to the purely algebraic case as
$$
\lim_i\colim_k\CC(X_i,Y_k),
$$
where limit and colimit are understood in the topological (Banach space theoretic) sense. One disadvantage of this picture is that  one has to check not only that the compositions but also that the adjoints are well-defined.
\end{remark}

In the following we assume that $\CC$ is essentially small. We will mainly be interested in ind-objects defined by inductive systems of objects of the form $\oplus_{i\in F}X_i$ for finite $F\subset I$ with obvious inclusion maps between them. We denote such ind-objects by $\oplus_{i\in I}X_i$. A morphisms between two such ind-objects $\oplus_{i\in I}X_i$ and $\oplus_{k\in K}Y_k$ is a collection of morphisms $T_{ki}\colon X_i\to Y_k$ such that the morphisms
$$
(T_{ki})_{k\in G, i\in F}\colon \oplus_{i\in F}X_i\to\oplus_{k\in G}Y_k
$$
are uniformly bounded when $F$ and $G$ run over all finite subsets of $I$ and $K$, respectively. In fact, there is no loss of generality in considering only such ind-objects.

\begin{proposition}
Any ind-object of $\CC$ is isomorphic to an object of the form $\oplus_{i\in I}X_i$.
\end{proposition}

\bp
Consider an ind-object $Y_* = \{v_{lk}\colon Y_k\to Y_l\}_{k\prec l}$. Fix a simple object $X$ and assume first for simplicity that every $Y_k$ is isotypic to $X$. Consider $H=\Mor_\indC(X,Y_*)$. As remarked before, this is a Hilbert space with inner product such that $T^*S=(S,T)\iota_U$. Choose an orthonormal basis $\{\xi_i\}_{i\in I}$ in $H$. By definition, every basis vector $\xi_i$ is a collection of morphisms $\xi_{ki}\colon X\to Y_k$. For every finite subset $F\subset I$ these morphisms define a morphism $u_{k,F}=(\xi_{ki})_{i\in F}\colon\oplus_{i\in F}U\to Y_k$. The morphisms $u_{k,F}$ define, in turn, a morphism of ind-objects $u\colon\oplus_{i\in I}X\to Y_*$. Orthonormality of the vectors $\xi_i$ implies that $u$ is well-defined and isometric. We claim that $u$ is unitary. This means that for every $k$ the morphisms $u_{k,F}u_{k,F}^*$ converge to the identity morphism of $Y_k$ as $F\to I$. In order to show this, it suffices to check that for every $T\colon U\to Y_k$ we have
$$
\lim_F u_{k,F}u_{k,F}^*T=T.
$$
The morphisms $v_{lk}T\colon X\to Y_l$ define a morphism $\zeta\colon X\to Y_*$. Then, as long as $F$ is large enough, this morphism is close to $\sum_{i\in F}(\zeta,\xi_i)\xi_i=\sum_{i\in F}\xi_i\xi_i^*\zeta$, so that $T$ is close to
$$
\sum_{i\in F}(\xi_i\xi_i^*\zeta)_k=\lim_{l}\sum_{i\in F}\xi_{ki}\xi_{li}^*v_{lk}T=u_{k,F}u_{k,F}^*T,
$$
and our claim is proved. Thus $Y\cong\oplus_{i\in I}X$.

\smallskip

In the general case we can decompose the objects $Y_k$ into isotypic components and repeat the above arguments.
\ep

Let us choose representatives $(U_s)_{s\in\Irr(\CC)}$ of the isomorphism classes of simple objects of $\CC$. Then the proposition and its proof show that any ind-object can be represented by a formal direct sum $\bigoplus_s U_s\otimes H_s$, where~$H_s$ are Hilbert spaces, cf.~\citelist{\cite{MR654325}\cite{MR3121622}}. The morphism space between two such direct sums $\bigoplus_s U_s\otimes H_s$ and $\bigoplus_s U_s\otimes H'_s$ is defined as
$$
\ell^\infty\text{-}\bigoplus_s B(H_s,H'_s).
$$
While this gives a very clear picture of $\indC$, it is not always convenient, as we will see soon, to decompose ind-objects into direct sums of simple objects.

\bigskip

\section{Drinfeld center}
\label{sec:half-braidings}

From now on we assume that $\CC$ is an essentially small strict rigid C$^*$-tensor category satisfying our standard assumptions: $\CC$ is closed under finite direct sums and subobjects, and the unit of $\CC$ is simple.

\subsection{Half-braidings in rigid \texorpdfstring{C$^*$}{C*}-tensor categories}
\label{subsec:half-br}

The category $\indC$ is itself a C$^*$-tensor category: the tensor product of ind-objects defined by inductive systems  $\{u_{ji}\colon X_i\to X_j\}_{i\prec j}$ and $\{v_{lk}\colon Y_k\to Y_l\}_{k\prec l}$ is represented by the inductive system
 $\{u_{ji}\otimes v_{lk}\colon X_i\otimes Y_k\to X_j\otimes Y_l\}_{i\prec j, k\prec l}$. The category $\indC$ is again closed under direct sums and subobjects, and the unit of $\indC$ is simple, but $\indC$ is no longer rigid. More precisely, the only ind-objects that have conjugates are the ones lying in $\CC$.

Consider now the Drinfeld center, or the Drinfeld double, $\ZC$ of $\indC$ in the C$^*$-algebraic sense, meaning that it is constructed using unitary half-braidings. More precisely, recall that given an ind-object~$Z$, a half-braiding on $Z$ is a collection of natural in $X\in\indC$ isomorphisms $c_X\colon X\otimes Z\to Z\otimes X$ such that for all objects $X$ and $Y$ in $\indC$ we have
\begin{equation} \label{eq:halfbr}
c_{X\otimes Y}=(c_X\otimes\iota_Y)(\iota_X\otimes c_Y).
\end{equation}
We will only consider unitary half-braidings.

\begin{remark}\label{rhb}
A unitary half-braiding is completely determined by its values on objects of $\CC$. In other words, having a unitary half-braiding on $Z$ is the same thing as having a collection of natural in $X\in\CC$ unitary isomorphisms $c_X\colon X\otimes Z\to Z\otimes X$ such that for all objects $X$ and $Y$ in $\CC$ identity~\eqref{eq:halfbr} holds.
\end{remark}

By definition, the objects of $\ZC$ are pairs $(Z,c)$, where $Z$ is an ind-object of $\CC$ and $c$ is a unitary half-braiding on $Z$. The morphisms are defined as the morphisms of $\indC$ respecting the half-braidings. Then $\ZC$ is a C$^*$-tensor category with the tensor product
$$
(Z,c)\otimes(Z',c')=(Z\otimes Z', (\iota_Z\otimes c')(c\otimes\iota_{Z'})).
$$
Furthermore, $\ZC$ is braided, with the unitary braiding defined by
$$
\sigma_{(Z,c),(Z',c')}=c'_Z.
$$

The Drinfeld center $\ZCC$ of the category $\CC$ is a full C$^*$-tensor subcategory of $\ZC$. It consists exactly of the objects that have duals: it is not difficult to see that as a dual of $(Z,c)$, with $Z\in\CC$, we can take $(\bar Z,\bar c)$, where $\bar c_X=(c_{\bar X})^\vee$.

\subsection{Regular half-braidings}
\label{sec:regul-half-braid}

Our goal now is to construct a particular element of $\ZC$ playing the role of the regular representation. Fix representatives $(U_s)_{s \in \Irr(\CC)}$ of isomorphism classes of simple objects in~$\CC$. Denote the index corresponding to the class of $\un$ by $e$ and assume for convenience that $U_e=\un$. Consider the ind-object
$$
\Zreg = \Zreg(\CC) =\bigoplus_{s \in \Irr(\CC)} U_s\otimes\bar U_s.
$$

Recall that once standard solutions are fixed, we have anti-multiplicative maps $\CC(X,Y)\to \CC(\bar Y,\bar X)$, $T\mapsto T^\vee$, defined by either of the following identities:
$$
(\iota\otimes T)R_X=(T^\vee\otimes\iota)R_Y,\ \ (T\otimes\iota)\bar R_X=(\iota\otimes T^\vee)\bar R_Y.
$$
Let us now fix  an object $X$ and choose a standard solution $(R_X,\bar R_X)$ of the conjugate equations. Let us also fix once for all standard solutions $(R_s,\bar R_s)$ for $U_s$. For every $s$ and $t$ choose isometries $u^\alpha_{st}\colon U_t\to X\otimes U_s$ such that $\sum_\alpha u^\alpha_{st}u^{\alpha*}_{st}$ is the projection onto the isotypic component of $X\otimes U_s$ corresponding to $U_t$. We then define
$$
c_{X,ts}\colon X\otimes U_s\otimes\bar U_s\to U_t\otimes\bar U_t\otimes X
$$
by
$$
c_{X,ts}=\left(\frac{d_s}{d_t}\right)^{1/2}\sum_\alpha (u^{\alpha*}_{st}\otimes u^{\alpha\vee}_{st}\otimes\iota_X)(\iota_X\otimes\iota_s\otimes\iota_{\bar s}\otimes R_X).
$$
Here $d_s$ and $d_t$ denote the quantum dimensions of $U_s$ and $U_t$, while to define $u^{\alpha\vee}_{st}$ we take as the dual of $X\otimes U_s$ the tensor product $\bar U_s\otimes\bar X$, with the standard solutions defined in the usual way from our fixed standard solutions for $X$ and $U_s$:
$$
R_{X\otimes U_s}=(\iota\otimes R_X\otimes\iota)R_s,\ \ \bar R_{X\otimes U_s}=(\iota\otimes \bar R_s\otimes\iota)\bar R_X.
$$

\begin{lemma}
The morphisms $c_{X,ts}$ depend neither on the choice of isometries $u^\alpha_{st}$ nor on the choice of standard solutions for $X$ (assuming that $R_s$ are fixed). Furthermore, these morphisms are natural in $X$.
\end{lemma}

\bp The claim that $c_{X,ts}$ does not depend on the choice of $u^\alpha_{st}$ is standard and easy to check. As for dependence on the standard solutions, recall that any other standard solution $(R_X',\bar R_X')$ of the conjugate equations for $X$ has the form $R_X'=(u\otimes\iota)R_X$ and $\bar R_X'=(\iota\otimes u)\bar R_X$ for a unitary $u$. This changes $u^{\alpha\vee}_{st}$ into $u^{\alpha\vee}_{st}(\iota_{\bar s}\otimes u^*)$. But then we see that $(u^{\alpha\vee}_{st}\otimes\iota_X)(\iota_{\bar s}\otimes R_X)$ remains unchanged. More explicitly, a direct computation shows that
$$
(u^{\alpha\vee}_{st}\otimes\iota)(\iota\otimes R_X)=(\iota_{\bar t}\otimes\iota_X\otimes \bar R_s^*)(\iota_{\bar t}\otimes u^\alpha_{st}\otimes\iota_{\bar s}) (R_t\otimes\iota_{\bar s}).
$$

Finally, the last statement of the lemma follows easily from the first two, since in order to prove it, it suffices to check that the morphisms $c_{X,ts}$ respect the embeddings $X\to X\oplus Y$ and projections $X\oplus Y\to X$.
\ep

Note for future reference that
\begin{equation} \label{eq:matrixte}
c_{s,te}=\delta_{st}d_s^{-1/2}(\iota_s\otimes R_s).
\end{equation}
Later, see identity~\eqref{eq:regbr}, we will also obtain the following expression for $c_{X,ts}$:
$$
c_{X,ts}=d_t^{1/2}d_s^{1/2}(\iota_t\otimes\iota_{\bar t}\otimes\iota_X\otimes \bar R^*_s)(\iota_t\otimes p^{\bar U_t\otimes X\otimes U_s}_e\otimes\iota_{\bar s})(\bar R_t\otimes\iota_X\otimes\iota_s\otimes\iota_{\bar s}),
$$
where $p^U_e$ is the projection onto the isotypic component of $U$ corresponding to the unit object.

Observe next that the matrix $(c_{X,ts})_{t,s}$ is row and column finite, so when taking compositions of such matrices we will not have to worry about convergence.

\begin{lemma}
\label{lem:reg-half-br-candid-unitary}
The morphisms $c_{X,ts}$ define a unitary $c_X\colon X\otimes \Zreg\to\Zreg\otimes X$.
\end{lemma}

\bp Let us first check that the morphisms $c_{X,ts}$ define an isometry $c_X\colon X\otimes\Zreg\to \Zreg\otimes X$. It suffices to check that for all $r$ and $s$, we have
$$
\sum_t c_{X,tr}^*c_{X,ts}=\delta_{rs}\iota_{X\otimes U_s\otimes\bar U_s}.
$$
By definition this means that we have to check that
$$
\frac{(d_rd_s)^{1/2}}{d_t}\sum_{t,\alpha,\beta}(\iota\otimes\Tr_{\bar X})(u^\alpha_{rt}u^{\beta*}_{st}\otimes u^{\alpha\vee*}_{rt}u^{\beta\vee}_{st})
=\delta_{rs}\iota.
$$
For this, in turn, it suffices to check that if $U_t\prec X\otimes U_s$, then
\begin{equation} \label{eorth1}
\frac{(d_rd_s)^{1/2}}{d_t}(\iota\otimes\Tr_{\bar X})(u^{\alpha\vee*}_{rt}u^{\beta\vee}_{st})=\delta_{rs}\delta_{\alpha\beta}\iota_{\bar s}.
\end{equation}
Since $\bar U_r$ and $\bar U_s$ are simple, the left hand side is zero if $r\ne s$. If $r=s$, the left hand side is a scalar multiple of the identity morphism. Therefore in this case in order to check the identity we can take categorical traces of both sides. Then the right hand side gives $\delta_{\alpha\beta}d_s$, while the left hand side gives
$$
\frac{d_s}{d_t}\Tr_{\bar U_s\otimes\bar X}( u^{\alpha\vee*}_{st}u^{\beta\vee}_{st} )
=\frac{d_s}{d_t}\Tr_{\bar t}(u^{\beta\vee}_{st}u^{\alpha\vee*}_{st})
=\frac{d_s}{d_t}\Tr_{\bar t}((u^{\alpha*}_{st}u^{\beta}_{st})^\vee)=\delta_{\alpha\beta}d_s,
$$
which is what we need.

\medskip

We next check that $c_X$ is unitary. We have to show that for all $t$ and $\tau$ we have
$$
\sum_s c_{X,ts}c_{X,\tau s}^*=\delta_{t\tau}\iota_{U_t\otimes\bar U_t\otimes X}.
$$
Since $u^{\alpha*}_{st}u^\beta_{s\tau}=\delta_{t\tau}\delta_{\alpha\beta}\iota_{U_t}$, the above identity is immediate for $t\ne\tau$, while for $t=\tau$ the left hand side equals
$$
\sum_{s,\alpha}\frac{d_s}{d_t}(\iota_{t}\otimes u^{\alpha\vee}_{st}\otimes\iota_X)(\iota_{t}\otimes\iota_{\bar s}\otimes R_X R_X^*)(\iota_{t}\otimes u^{\alpha\vee*}_{st}\otimes\iota_X).
$$
Therefore in order to finish the proof it suffices to show that for every $s$, the morphism
$$
\sum_\alpha\frac{d_s}{d_t}(u^{\alpha\vee}_{st}\otimes\iota_X)(\iota_{\bar s}\otimes R_XR_X^*)(u^{\alpha\vee*}_{st}\otimes\iota_X)
$$
is the projection onto the isotypic component of $\bar U_t\otimes X$ corresponding to $\bar U_s$. For this, observe that by Frobenius reciprocity the morphisms
$$
w^\alpha_{ts}=\left(\frac{d_s}{d_t}\right)^{1/2}(u^{\alpha\vee}_{st}\otimes\iota_X)(\iota_{\bar s}\otimes R_X)\colon \bar U_s\to \bar U_t\otimes X
$$
form a basis in $\CC(\bar U_s,\bar U_t\otimes X)$. We claim that we also have the orthogonality $w^{\alpha*}_{ts}w^\beta_{ts}=\delta_{\alpha\beta}\iota$ with respect to this relation. By definition of the categorical trace this is equivalent to
$$
\frac{d_s}{d_t}(\iota\otimes\Tr_{\bar X})(u^{\alpha\vee*}_{st}u^{\beta\vee}_{st})=\delta_{\alpha\beta}\iota_{\bar U_s}.
$$
But this follows from \eqref{eorth1}, so our claim is proved. We conclude that $\sum_\alpha w^{\alpha}_{ts}w^{\alpha*}_{ts}$ is the projection onto the isotypic component of $\bar U_t\otimes X$ corresponding to $\bar U_s$.
\ep

\begin{theorem} \label{them:reg-half-braiding}
The unitaries $c_X\colon X\otimes\Zreg\to\Zreg\otimes X$ form a half-braiding on $\Zreg$.
\end{theorem}

\bp
It remains only to check identity~\eqref{eq:halfbr}. In order to compute $c_{X\otimes Y}$, choose isometries $u^\alpha_{rt}\colon U_t\to X\otimes U_r$ as before, and similarly choose isometries $v^\beta_{sr}\colon U_r\to Y\otimes U_s$. Then using the isometries
$$
(\iota_X\otimes v^\beta_{sr})u^\alpha_{rt}\colon U_t\to X\otimes Y\otimes U_s
$$
in the definition of $c_{X\otimes Y,ts}$, we get
\begin{align*}
c_{X\otimes Y,ts}&=\left(\frac{d_s}{d_t}\right)^{1/2}\sum_{r,\alpha,\beta}(u^{\alpha*}_{rt}(\iota_X\otimes v^{\beta*}_{sr})\otimes u^{\alpha\vee}_{rt}(v^{\beta\vee}_{sr}\otimes\iota_{\bar X})\otimes\iota_{X\otimes Y})\\
&\qquad\qquad\qquad\qquad\qquad\qquad (\iota_{X\otimes Y\otimes U_s\otimes\bar U_s \otimes\bar Y}\otimes R_X\otimes\iota_Y)(\iota_{X\otimes Y\otimes U_s\otimes\bar U_s}\otimes R_Y)\\
&=\sum_r(c_{X,tr}\otimes\iota_Y)(\iota_X\otimes c_{Y,rs}).
\end{align*}
This means that $c_{X\otimes Y}=(c_X\otimes\iota_Y)(\iota_X\otimes c_Y)$.
\ep

We will often denote the object $(\Zreg, c)$ by just one symbol $\Zreg$ or $\Zreg(\CC)$.

\subsection{Unitary half-braidings and amenability}
We have shown that $\ZC$ is always rich. Expanding on ideas of Longo and Roberts~\cite{MR1444286}*{Section~5}, we will now show that generally this is not the case for~$\ZCC$, so we do need to consider ind-objects in order to construct nontrivial unitary half-braidings. These considerations are not going to be used in the subsequent sections, so we will be somewhat brief.

\smallskip

For every object $X$ in $\CC$ denote by $\Gamma_X=(a^X_{st})_{s,t}\in B(\ell^2(\Irr(\CC)))$ the matrix describing decompositions of $X\otimes Y$ into simple objects, so $a^X_{st}=\dim \CC(U_s,X\otimes U_t)$. Then $\norm{\Gamma_X}\le d(X)$, and the category $\CC$ is called amenable if $\|\Gamma_X\|=d(X)$ for all objects $X$ in $\CC$. Let us say that an object $X$ is amenable, if the full rigid C$^*$-tensor subcategory of $\CC$ generated by $X$ is amenable.
We remark that it is not difficult to show, see e.g.~the proof of~\cite{MR1644299}*{Proposition~4.8}, that the norm of the matrix $\Gamma_X$ remains the same if we replace $\CC$ by any full rigid C$^*$-tensor subcategory of $\CC$ containing $X$. Therefore $\CC$ is amenable if and only if every object of $\CC$ is amenable.

\begin{theorem}
Assume that for a rigid C$^*$-tensor category $\CC$ there exists a unitary half-braiding on an object $X\in\CC$. Then $X$ is amenable.
\end{theorem}

\bp
We may assume that $\CC$ is generated by $X$ as a rigid C$^*$-tensor category. Replacing, if necessary, $X$ by~$X\oplus \bar X$, we may also assume that every simple object embeds into $X^{\otimes n}$ for some $n\ge1$. Consider the Poisson boundary $\PP$ of $\CC$ with respect to the probability measure on $\Irr(\CC)$ defined by the normalized categorical trace on $X$~\cite{arXiv:1405.6572}. We will prove that the Poisson boundary is trivial, which by~\cite{arXiv:1405.6572}*{Theorem~5.7} implies amenability of $\CC$.

We view $\CC$ as a C$^*$-tensor subcategory of $\PP$. By definition, the elements of $\PP(Z)$ are bounded collections~$\xi=(\xi_Y)_Y$ of natural in $Y$ morphisms $Y\otimes Z\to Y\otimes Z$ that are harmonic, meaning that
$$
(\tr_X\otimes\iota)(\xi_{X\otimes Y})=\xi_Y\ \ \text{for all objects}\ \ Y\in\CC.
$$
They can be realized as follows~\cite{arXiv:1405.6572}*{Proposition~3.3}. The algebras $\NN_Z^{(n)}=\CC(X^{\otimes n}\otimes Z)$, equipped with the normalized categorical traces and the embeddings $T\mapsto\iota_X\otimes T$,  form an inductive system. In the limit we get a finite von Neumann algebra $\NN_Z$. For any $\xi\in\PP(Z)$, the elements $\xi^{[n]}=\xi_{X^{\otimes n}}\in\NN^{(n)}_Z$ converge in the strong$^*$ operator topology to an element $\xi^{[\infty]}\in\NN_Z$, and the map $\xi\mapsto\xi^{[\infty]}$ gives an algebra embedding of $\PP(Z)$ into~$\NN_Z$.

Take $\xi\in\PP(Z)$. Then $(\iota_X\otimes\xi)^{[n]}=\xi_{X^{\otimes(n+1)}}$. On the other hand, if $c$ is a unitary half-braiding on $X$, then
$$
(c_{Z}(\xi\otimes\iota_X)c_{Z}^*)^{[n]}=(\iota\otimes c_{Z})(\xi_{X^{\otimes n}}\otimes\iota)(\iota\otimes c_{Z}^*)
=(c^*_{X^{\otimes n}}\otimes\iota)
(\iota\otimes\xi_{X^{\otimes n}})(c_{X^{\otimes n}}\otimes\iota).
$$
Since $\iota_X\otimes\xi_{X^{\otimes n}}$ is the image of $\xi^{[n]}$ under the embedding $\NN^{(n)}_Z\hookrightarrow \NN^{(n+1)}_Z$, as $n$ grows, the last expression becomes close in the trace-norm to
$$
(c^*_{X^{\otimes n}}\otimes\iota)\xi^{[n+1]}(c_{X^{\otimes n}}\otimes\iota)
=(c^*_{X^{\otimes n}}\otimes\iota)\xi_{X^{\otimes(n+1)}}(c_{X^{\otimes n}}\otimes\iota)=\xi_{X^{\otimes(n+1)}}=(\iota_X\otimes\xi)^{[n]}.
$$
It follows that $c_{Z}(\xi\otimes\iota_X)c_{Z}^*=\iota_X\otimes\xi$, that is,
$$
(\iota_Y\otimes c_{Z})(\xi_Y\otimes\iota_X)(\iota_Y\otimes c_{Z}^*)=\xi_{Y\otimes X}\ \ \text{for all}\ \ Y.
$$
The left hand side equals $(c^*_{Y}\otimes\iota_Z)(\iota_X\otimes\xi_Y)(c_{Y}\otimes\iota_Z)$. Therefore by conjugating by $c_{Y}\otimes\iota_Z$ we get
$$
\iota_X\otimes\xi_Y=\xi_{X\otimes Y}\ \ \text{for all}\ \ Y.
$$
A simple induction shows then that the same identity holds with $X$ replaced by $X^{\otimes n}$, hence it holds for any simple object $U$ in place of $X$. Letting $Y=\un$ we then get $\iota_U\otimes\xi_\un=\xi_U$. Thus, under our embedding of~$\CC(Z)$ into~$\PP(Z)$, we have $\xi=\xi_\un\in\CC(Z)$.
\ep

In particular, if $\CC$ admits a unitary braiding, or even weaker, if $\CC$ is generated as a rigid C$^*$-tensor category by objects admitting unitary half-braidings, then $\CC$ is amenable. This is a categorical analogue of the fact that abelian groups are amenable.

\begin{example}
If $\CC=\Rep G$ is the representation category of a compact quantum group $G$, then a necessary condition for amenability of $U\in\Rep G$ is the equality $\dim U=\dim_q U$. Therefore if $\dim U<\dim_q U$, there exists no unitary half-braiding on $U$.
\end{example}

\bigskip

\section{Representations of the character algebra}
\label{sec:character-algebra}

We continue to assume that $\CC$ is a rigid C$^*$-tensor category as in the previous section.

\subsection{From half-braidings to representations}

Recall that there is a semiring structure on the semigroup $\Z_+[\Irr(\CC)]$, with the product defined by
$$
[U] \cdot [V] = \sum_{s\in\Irr(\CC)} \dim \CC(U_s, U \otimes V) [U_s].
$$
The operation $[U] \mapsto [\bar{U}]$ extends to an anti-multiplicative involution of this semiring. We embed the involutive semiring $\Z_+[\Irr(\CC)]$  into the involutive $\C$-algebra $\C[\Irr(\CC)]$.

Suppose that $(c_X\colon X \otimes Z \to Z \otimes X)_{X \in \CC}$ is a unitary half-braiding on an ind-object $Z$. We want to define a $*$-representation of $\C[\Irr(\CC)]$ on the Hilbert space $\Mor_{\indC}(\un, Z)$ with scalar product defined by $(\xi,\zeta)\iota=\zeta^*\xi$. Let $X$ be an object in $\CC$ and $(R_X, \bar{R}_X)$ be a standard solution of the conjugate equations for~$X$. If $\xi \in \Mor_\indC(\un, Z)$, we obtain a new element in the same morphism set by
$$
\pi_{(Z,c)}([X])\xi = (\iota_Z \otimes \bar{R}_X^*)(c_X \otimes \iota_{\bar{X}})(\iota_X \otimes \xi \otimes \iota_{\bar{X}})\bar{R}_X\colon \un \to X \otimes \bar{X} \to X \otimes Z \otimes \bar{X} \to Z \otimes X \otimes \bar{X} \to Z.
$$
Since any other choice of $(R_X, \bar{R}_X)$ is of the form $((T \otimes \iota) R_X, (\iota \otimes T) \bar{R}_X)$ for some unitary $T$, the above definition does not depend on the choice of a standard solution. In order to simplify the notation we write~$\pi_Z$ instead of $\pi_{(Z,c)}$ when there is no danger of confusion.

It is clear that $\|\pi_Z([X])\|\le \|\bar R_X\|^2=d(X)$. It is also easy to see that $\pi_Z([X])$ is additive in $X$. The half-braiding axiom~\eqref{eq:halfbr} implies that $\pi_Z([X])$ is multiplicative in $X$. Thus we obtain a representation $\pi_Z$ of the algebra $\C[\Irr(\CC)]$ on $\Mor_\indC(\un, Z)$.

Next we want to check the compatibility with the involution. For this we need the following lemma, which we will also repeatedly use later.

\begin{lemma} \label{lem:braidinv}
We have $(\iota_{\bar X}\otimes c_X)(R_X\otimes\iota_Z)=(c^*_{\bar X}\otimes\iota_X)(\iota_Z\otimes R_X)$.
\end{lemma}

\bp Since $c_\un=\iota$, we have $c_{\bar X\otimes X}(R_X\otimes\iota_Z)=\iota_Z\otimes R_X$. Using then that $c_{\bar X\otimes X}=(c_{\bar X}\otimes\iota_X)(\iota_{\bar X}\otimes c_X)$, we get the result.
\ep

\begin{lemma}
\label{lem:char-alg-act-compat-adj}
  For any $\xi, \eta \in \Mor_\indC(\un, Z)$ and $X \in \CC$, we have $(\pi_Z([X]) \xi, \eta) = (\xi, \pi_Z([\bar{X}]) \eta)$.
\end{lemma}

\begin{proof}
We have to show the equality
\begin{equation} \label{eq:starrep}
\eta^*(\iota_Z \otimes \bar{R}_X^*)(c_X \otimes \iota_{\bar{X}}) (\iota_X \otimes \xi \otimes \iota_{\bar{X}}) \bar{R}_X = R_X^* (\iota_{\bar{X}} \otimes \eta^* \otimes \iota_X) (c_{\bar{X}}^* \otimes \iota_X) (\iota_Z \otimes R_X) \xi.
\end{equation}
The left hand side can be written as
$$
\bar{R}_X^*(\eta^*\otimes\iota_X\otimes\iota_{\bar X})(c_X \otimes \iota_{\bar{X}}) (\iota_X \otimes \xi \otimes \iota_{\bar{X}}) \bar{R}_X.
$$
Consider the morphism $T=(\eta^*\otimes\iota_X)c_X (\iota_X \otimes \xi)\in\CC(X)$. Then the left hand side of \eqref{eq:starrep} equals
$$
\bar{R}_X^*(T\otimes\iota_{\bar X})\bar R_X=\Tr_X(T).
$$
On the other hand, since $(c^*_{\bar X}\otimes\iota_X)(\iota_Z\otimes R_X)=(\iota_{\bar X}\otimes c_X)(R_X\otimes\iota_Z)$ by Lemma~\ref{lem:braidinv}, the right hand side of~\eqref{eq:starrep} equals
$$
R_X^*(\iota\otimes T)R_X=\Tr_X(T),
$$
so we get the desired equality.
\end{proof}

From now on by a representation of $\C[\Irr(\CC)]$ we mean a $*$-representation.

\begin{definition}
  We define the \emph{C$^*$-character algebra} $C^*(\CC)$ of $\CC$ to be the C$^*$-completion of the $*$-algebra $\C[\Irr(\CC)]$ with respect to the representations $\pi_Z$ for all objects $(Z, c)\in\ZC$.
\end{definition}

As we already observed, $\|[X]\|\le d(X)$ in $C^*(\CC)$. The next example shows that this is actually equality.

\begin{example}
  Consider the trivial half-braiding $(X \otimes \un \to \un \otimes X)_X$ for $\un$. Then we obtain a representation of $\C[\Irr(\CC)]$ on $\C$, that is, a character. Expanding the relevant definitions we see that $\pi_\un([X]) = d(X)$. We call $\pi_\un$ the \emph{trivial representation} of $\C[\Irr(\CC)]$.
\end{example}

\begin{example} \label{ex:regrep}
  Consider the half-braiding $(X \otimes \Zreg \to \Zreg \otimes X)_X$ constructed in Section~\ref{subsec:half-br}.  The Hilbert space $\Mor_\indC(\un, \Zreg)$
  has an orthonormal basis consisting of the vectors $\xi_s=d_s^{-1/2}\bar{R}_s$, $s\in\Irr(\CC)$. It follows from \eqref{eq:matrixte} that $\pi_{\Zreg}([U_s])\xi_e=\xi_s$. Therefore $\pi_{\Zreg}$ can be identified with the regular representation of $\C[\Irr(\CC)]$ on $\ell^2(\Irr(\CC))$.
\end{example}

\begin{remark}\label{rmk:q-grp-char-corr}
If $\CC = \Rep G$ for some compact quantum group $G$, the half-braidings correspond to the $*$-representations of the Drinfeld double $\Pol_c(\coD(G)) = \Pol(G) \bowtie c_c(\hat{G})$ via the standard argument (cf.~\cite{MR1321145}*{Section~IX.5}). The C$^*$-algebra $C^*(\CC)$ coincides with the C$^*$-completion of the character algebra of $G$ with respect to the embedding $\chi_U \mapsto \sigma_{-i/2}(\chi_U) h$ and the norm on $\Pol_c(\coD(G))$ induced by the ``spherical unitary representations'', where $h$ is the Haar state and $\sigma_z$ is its modular automorphism group, see~\cite{MR3238527}*{Remark~31}.
\end{remark}

\subsection{Positive definite functions}\label{sec:positivedefinite}

Given an object $(Z,c)\in\ZC$ and a vector $\xi\in\Mor_{\indC}(\un,Z)$, the cyclic representation of $\C[\Irr\CC]$ on $\overline{\pi_Z(\C[\Irr\CC])\xi}$ is completely determined by the function $\phi(s)=d_s^{-1}(\pi([U_s])\xi,\xi)$ on $\Irr(\CC)$. It is natural to call such functions positive definite. While this definition would be sufficient for the theory we develop in the subsequent sections, it is clearly unsatisfactory. A correct intrinsic definition has been given by Popa and Vaes~\cite{MR3406647}. We will present it in a way convenient for our applications.

For a function $\phi$ on $\Irr(\CC)$ denote by $M^\phi$ the endomorphism of the identity functor on $\CC$ such that $M^\phi_s\colon U_s\to U_s$ is the scalar morphism $\phi(s)$ for every $s\in\Irr(\CC)$. For $s,t\in\Irr(\CC)$ define a morphism
$$
A^\phi_{st}=d_s^{1/2}d_t^{1/2}(\iota_s\otimes\iota_{\bar s}\otimes \bar R^*_t)(\iota_s\otimes M^\phi_{\bar U_s\otimes U_t}\otimes\iota_{\bar t})(\bar R_s\otimes\iota_t\otimes\iota_{\bar t})\colon U_t\otimes\bar U_t\to U_s\otimes\bar U_s.
$$

\begin{definition}
A function $\phi$ on $\Irr(\CC)$ is called \emph{positive definite}, or a \emph{cp-multiplier}, if for any $s_1,\dots,s_n\in\Irr(\CC)$ the morphism
$$
(A^\phi_{s_i,s_j})^n_{i,j=1}\colon\bigoplus^n_{k=1}U_{s_k}\otimes\bar U_{s_k}\to \bigoplus^n_{k=1}U_{s_k}\otimes\bar U_{s_k}
$$
is positive.
\end{definition}

In the original definition of Popa and Vaes a cp-multiplier is defined by requiring certain maps $\theta^\phi_{U,V}$ on $\CC(U\otimes V)$ to be completely positive for all $U,V\in\CC$. But it is shown in~\cite{MR3406647}*{Lemma~3.7} that it suffices to check positivity of $\theta^\phi_{U,\bar U}(\bar R_U\bar R_U^*)$ for all $U$. Expanding the definitions one can check that
$$
\theta^\phi_{U,\bar U}(\bar R_U\bar R^*_U)=( \iota_U\otimes\iota_{\bar U}\otimes \bar R_U^*)(\iota_U\otimes M^\phi_{\bar U\otimes U} \otimes \iota_{\bar U})(\bar R_U\otimes\iota_U\otimes\iota_{\bar U}).
$$
For $U=\oplus^n_{i=1}U_{s_i}$ positivity of the above expression means exactly positivity of $(A^\phi_{s_i,s_j})^n_{i,j=1}$. Thus the above definition is equivalent to
 the one in~\cite{MR3406647}.

\begin{example}
Let $\Gamma$ be a discrete group and $\CC=\Hilb_{\Gamma,f}$ be the category of $\Gamma$-graded finite dimensional Hilbert spaces, or in other words, the representation category of the dual compact quantum group $\hat\Gamma$. Thus $\Irr(\CC)=\Gamma$ and we can choose representatives $U_s$, $s\in\Gamma$, of isomorphism classes of simple objects such that $U_s\otimes U_t=U_{st}$. Then $A^\phi_{st}$ is the scalar endomorphism $\phi(s^{-1}t)$ of $U_e$. Therefore a function $\phi$ on $\Gamma$ is positive definite in the above sense if and only if the matrix $(\phi(s_i^{-1}s_j))^n_{i,j=1}$ is positive for any $s_1,\dots,s_n\in\Gamma$, which is the standard definition of positive definite functions on groups.
\end{example}

\begin{example}\label{ex:deltae}
Consider an arbitrary rigid C$^*$-tensor category $\CC$ and the function $\phi=\delta_e$. In this case $A^\phi_{st}=\delta_{st}\iota$, since $M^\phi_{\bar U_s\otimes U_t}=\delta_{st}d_s^{-1}R_sR_s^*$. Therefore the function $\phi=\delta_e$ is positive definite.
\end{example}

\begin{theorem} \label{thm:positivedef}
For any function $\phi$ on $\Irr(\CC)$ the following conditions are equivalent:
\begin{itemize}
\item[(i)] $\phi$ is positive definite;
\item[(ii)] $\phi(s)=d_s^{-1}(\pi_Z([U_s])\xi,\xi)$ for some $(Z,c)\in\ZC$ and $\xi\in\Mor_{\indC}(\un,Z)$;
\item[(iii)] $\phi(s)=d_s^{-1}\omega([U_s])$ for a positive linear functional $\omega$ on $C^*(\CC)$.
\end{itemize}
\end{theorem}

Popa and Vaes defined a C$^*$-algebra $C_u(\CC)$ for a rigid C$^*$-tensor category $\CC$ as the C$^*$-envelope of $\C[\Irr(\CC)]$ with respect to the representations $\pi\colon \C[\Irr(\CC)]\to B(H)$ such that $s \mapsto d_s^{-1}(\pi([U_s]) \xi, \xi)$ is a cp-multiplier for any $\xi \in H$. As an immediate consequence of the above theorem we get the following.

\begin{corollary} \label{cor:twodef}
The identity map on $\C[\Irr(\CC)]$ extends to an isomorphism of $C^*(\CC)$ onto $C_u(\CC)$.
\end{corollary}

Turning to the proof of Theorem~\ref{thm:positivedef}, (ii) obviously implies (iii). Let us prove that (iii) implies (i). Take a representation $\pi\colon C^*(\CC)\to B(H)$ and a vector~$\xi$. We want to show that the function $\phi(s)=d_s^{-1}(\pi([U_s])\xi,\xi)$ is positive definite. Since any representation of $C^*(\CC)$ is weakly contained in a direct sum of representations defined by objects of $\ZC$ and the set of positive definite functions is closed under convex combinations and pointwise limits, without loss of generality we may assume that $\pi$ is defined by an object $(Z,c)\in\ZC$. Then $\xi \in \Mor_\indC(\un, Z)$. In other words, it suffices to show that (ii) implies (i).

\begin{lemma} \label{lem:Mphi}
For every object $U$ of $\CC$ the endomorphism $M^\phi_U$ is defined by the composition
$$
U\xrightarrow{\iota\otimes\xi}U\otimes Z\xrightarrow{c_U}Z\otimes U\xrightarrow{\xi^*\otimes\iota}U.
$$
\end{lemma}

\bp
It is clear that the composition in the formulation is natural in $U$. Therefore it suffices to consider $U=U_s$. Then the above composition is a scalar endomorphism $\alpha_s$. It follows that
$$
d_s\phi(s)=(\pi([U_s]) \xi, \xi)= \xi^*(\iota_Z \otimes \bar{R}_s^*)(c_s \otimes \iota_{\bar{s}})(\iota_s \otimes \xi \otimes \iota_{\bar{s}})\bar{R}_s=\alpha_s\bar{R}_s^*\bar{R}_s,
$$
so $\phi(s)=\alpha_s$.
\ep

\bp[Proof of the implication {\rm(ii)}$\Rightarrow${\rm(i)} in Theorem~\ref{thm:positivedef}]
Consider the function $\phi(s)=d_s^{-1}(\pi([U_s])\xi,\xi)$ as above. By the previous lemma we have
\begin{align*}
A^\phi_{st}&=d_s^{1/2}d_t^{1/2}(\iota_s\otimes\iota_{\bar s}\otimes \bar R^*_t)(\iota_s\otimes\xi^*\otimes\iota_{\bar s}\otimes \iota_t\otimes\iota_{\bar t})
(\iota_s\otimes c_{\bar U_s\otimes U_t}\otimes\iota_{\bar t})(\iota_s\otimes\iota_{\bar s}\otimes\iota_t\otimes\xi\otimes\iota_{\bar t})
(\bar R_s\otimes\iota_t\otimes\iota_{\bar t})\\
&=d_s^{1/2}d_t^{1/2}(\iota_s\otimes\xi^*\otimes\iota_{\bar s})(\iota_s\otimes\iota_Z\otimes\iota_{\bar s}\otimes \bar R_t^*)
(\iota_s\otimes c_{\bar U_s\otimes U_t}\otimes\iota_{\bar t})(\bar R_s\otimes\iota_t\otimes\iota_Z\otimes\iota_{\bar t})
(\iota_t\otimes\xi\otimes\iota_{\bar t}).
\end{align*}
Using that $c_{\bar U_s\otimes U_t}=(c_{\bar s}\otimes\iota_t)(\iota_{\bar s}\otimes c_t)$ we get
$$
A^\phi_{st}=d_s^{1/2}d_t^{1/2}(\iota_s\otimes\xi^*\otimes\iota_{\bar s})(\iota_s\otimes c_{\bar s})(\bar R_s\otimes\iota_Z)(\iota_Z\otimes\bar R_t^*)(c_t\otimes\iota_{\bar t})(\iota_t\otimes\xi\otimes\iota_{\bar t}).
$$
Since $(\iota_s\otimes c_{\bar s})(\bar R_s\otimes\iota_Z)=(c^*_s\otimes\iota_Z)(\iota_Z\otimes\bar R_s)$ by Lemma~\ref{lem:braidinv}, we therefore see that $A^\phi_{st}=T_s^*T_t$, where
$$
T_t=d_t^{1/2}(\iota_Z\otimes\bar R_t^*)(c_t\otimes\iota_{\bar t})(\iota_t\otimes\xi\otimes\iota_{\bar t})\colon U_t\otimes\bar U_t\to Z.
$$
This obviously implies positive definiteness of $\phi$.
\ep

Next, starting from a positive definite function we want to construct a unitary half-braiding. The construction will be a modification of our construction of $\Zreg$. Let us first describe the framework within which we will define such a modification.

Consider an ind-object  $\{u_{ji}\colon X_i\to X_j\}_{i\prec j}$ and assume that for every $i$ we are given a  positive morphism $A_i\colon X_i\to X_i$ such that
\begin{equation} \label{eq:poscompatibl}
A_i=u_{ji}^*A_ju_{ji}\ \ \text{for}\ \ i\prec j.
\end{equation}
From this data we can construct a new ind-object as follows. For every $i$ choose an object $Y_i$ and a surjective morphism $v_i\colon X_i\to Y_i$ such that $v_i^*v_i=A_i$. It is easy to see that such a pair $(Y_i,v_i)$ exists and is unique up to a unitary isomorphism. For example, we can take $Y_i$ to be the subobject of $X_i$ corresponding to the complement of the kernel of $A_i$ and then take $v_i=A_i^{1/2}$. But it is more instructive to think of $Y_i$ as a quotient of $X_i$ with a new inner product on morphisms into $Y_i$: given morphisms $S,T\colon X\to X_i$ we have
\begin{equation} \label{eq:newscalarpr}
(v_iT)^*v_iS=T^*A_iS.
\end{equation}

The following lemma is immediate by definition.

\begin{lemma} \label{lem:indiso}
Assume $T\colon X_i\to U$ is a morphism such that $T^*T=A_i$. Then there exists a unique isometry $\tilde T\colon Y_i\to U$ such that $T=\tilde T v_i$.
\end{lemma}

In particular, applying this to $T=v_ju_{ji}$ we conclude that for $i\prec j$ there exists a unique isometric morphism $w_{ji}\colon Y_i\to Y_j$ such that $v_ju_{ji}=w_{ji}v_i$. We thus get a new ind-object $\{w_{ji}\colon Y_i\to Y_j\}_{i\prec j}$.

We will use this construction for ind-objects of the form $\oplus_{i\in I}X_i$. In this case to be given positive endomorphisms $A_F$ of $X_F=\oplus_{i\in F}X_i$ for all finite sets $F\subset I$ satisfying \eqref{eq:poscompatibl} is the same thing as to have morphisms $A_{ji}\colon X_i\to X_j$ such that $(A_{ij})_{i,j\in F}$ is positive for any finite $F$. In this case, by slightly abusing the terminology, we simply say that $A=(A_{ij})_{i,j\in I}$ is positive. Therefore, starting from an ind-object $\oplus_{i\in I}X_i$ and a positive matrix of morphisms $A=(A_{ij})_{i,j\in I}$ we get a new ind-object, which we denote by
$$
A\text{-}\oplus_{i\in I}X_i.
$$
Note that by definition for any finite set $F\subset I$ we have a canonical morphism $\oplus_{i\in F}X_i\to A\text{-}\oplus_{i\in I}X_i$ obtained by composing $v_F\colon \oplus_{i\in F}X_i\to A\text{-}\oplus_{i\in F}X_i$ with the canonical isometry $A\text{-}\oplus_{i\in F}X_i\to A\text{-}\oplus_{i\in I}X_i$. But in general these morphisms do not define a bounded morphism $\oplus_{i\in I} X_i\to A\text{-}\oplus_{i\in I}X_i$.

In some cases an endomorphism of the original ind-object defines an endomorphism of the new one. The following will be sufficient for our purposes.

\begin{lemma} \label{lem:indunitary}
Let $\oplus_{i\in I}X_i$ and $\oplus_{k\in K}X'_k$ be ind-objects, $A=(A_{ij})_{i,j\in I}$ and $B=(B_{kl})_{k,l\in K}$ be positive matrices of morphisms $A_{ij}\colon X_j\to X_i$ and $B_{kl}\colon X'_l\to X'_k$, and $U=(U_{ki})_{k\in K,i\in I}\colon\oplus_iX_i\to\oplus_kX'_k$ be a unitary such that the matrix $(U_{ki})_{k,i}$ is row and column finite and $UA=BU$, that is,
$$
\sum_j U_{kj}A_{ji}=\sum_l B_{kl}U_{li}
$$
for all $i\in I$ and $k\in K$. Then $U$ defines a unitary $V\colon A\text{-}\oplus_i X_i\to B\text{-}\oplus_k X'_k$, meaning that for any finite set $F\subset I$ and all sufficiently large finite sets $G\subset K$ the diagram
$$
\begin{xymatrix}{
\oplus_{i\in F}X_i\ar[rr]^{(U_{ki})_{k\in G,i\in F}}\ar[d] & & \ar[d]\oplus_{k\in G}X'_k\\
A\text{-}\oplus_{i\in I} X_i\ar[rr]_V& & B\text{-}\oplus_{k\in K} X'_k
}\end{xymatrix}
$$
commutes.
\end{lemma}

\bp Take a finite set $F\subset I$ and let $G\subset K$ be any finite set such that $U_{ki}=0$ if $i\in F$ and $k\notin G$. Consider the morphisms $A_F=(A_{ij})_{i,j\in F}$, $B_G=(B_{kl})_{k,l\in G}$ and $U_{G,F}=(U_{ki})_{k\in G,i\in F}$. By the choice of $G$ and the assumptions of the lemma, for any $i,j\in F$ we have
$$
A_{ij}=\sum_{k,l\in K}U^*_{ki}B_{kl}U_{lj}=\sum_{k,l\in G}U^*_{ki}B_{kl}U_{lj},
$$
so $A_F=U^*_{G,F}B_GU_{G,F}$. By Lemma~\ref{lem:indiso} this implies that $U_{G,F}\colon\oplus_{i\in F}X_i\to \oplus_{k\in G}X'_k$ induces an isometry $V_{G,F}\colon A\text{-}\oplus_{i\in F}X_i\to B\text{-}\oplus_{k\in G}X'_k$. It is easy to see that the family of isometries $V_{G,F}$ is consistent and hence defines an isometry $V\colon A\text{-}\oplus_i X_i\to B\text{-}\oplus_k X'_k$ satisfying the statement of the lemma. Using $U^*$ instead of~$U$ we can similarly construct an isometry $V'\colon B\text{-}\oplus_k X'_k\to A\text{-}\oplus_i X_i$. It is straightforward to check that the isometries $V$ and $V'$ are inverse to each other.
\ep

Note that, under the assumptions of the previous lemma, if we denote by $\pi_F\colon\oplus_{i\in F}X_i\to A\text{-}\oplus_{i\in I}X_i$ and $\pi'_G\colon\oplus_{k\in G}X'_k\to B\text{-}\oplus_{k\in K}X'_k$ the canonical morphisms, then this lemma together with identity~\eqref{eq:newscalarpr} imply that for any morphisms $S=(S_i)_{i\in F}\colon X\to \oplus_{i\in F}X_i$ and $T=(T_k)_{k\in G}\colon X\to \oplus_{k\in G}X'_k$ we have
\begin{equation} \label{eq:newscalarpr1}
(\pi'_G T)^*V\pi_FS=\sum_{k,l,i}T_k^*B_{kl}U_{li}S_i.
\end{equation}

\bp[Proof of the implication {\rm(i)}$\Rightarrow${\rm(ii)} in Theorem~\ref{thm:positivedef}]
Let $\phi$ be a positive definite function on $\Irr(\CC)$, so that we have a positive matrix $A^\phi$ of morphisms $A^\phi_{st}\colon U_t\otimes\bar U_t\to U_s\otimes\bar U_s$. We can then define an ind-object
$$
Z_\phi=A^\phi\text{-}\bigoplus_{s\in\Irr(\CC)} U_s\otimes\bar U_s.
$$
We claim that the half-braiding $c$ for $\Zreg=\oplus_s U_s\otimes\bar U_s$ constructed in Section~\ref{subsec:half-br} defines a unitary half-braiding $c_\phi$ for $Z_\phi$. For objects $X\in\CC$, we have natural unitary isomorphisms
$$
X\otimes Z_\phi\cong(\iota_X\otimes A^\phi)\text{-}\oplus_s X\otimes U_s\otimes\bar U_s,\ \
Z\otimes X_\phi\cong(A^\phi\otimes\iota_X)\text{-}\oplus_s U_s\otimes\bar U_s\otimes X.
$$
Therefore in order to show that $c_X$ defines a unitary $c_{\phi,X}\colon X\otimes Z_\phi\to Z_\phi\otimes X$, by Lemma~\ref{lem:indunitary} it suffices to show that
\begin{equation} \label{eq:intertwin}
\sum_s c_{X,ps}(\iota_X\otimes A^\phi_{st})=\sum_q (A^\phi_{pq}\otimes\iota_X)c_{X,qt}.
\end{equation}
As in Section~\ref{subsec:half-br}, for all $q$ and $t$ choose a maximal family of isometries $u^\alpha_{tq}\colon U_q\to X\otimes U_t$ with mutually orthogonal ranges. Then the right hand side of \eqref{eq:intertwin} equals
\begin{align*}
&\sum_{\alpha,q} d_p^{1/2}d_t^{1/2}(\iota_p\otimes\iota_{\bar p}\otimes \bar R^*_q\otimes\iota_X)(\iota_p\otimes M^\phi_{\bar U_p\otimes U_q}\otimes\iota_{\bar q}\otimes\iota_X)(\bar R_p\otimes\iota_q\otimes\iota_{\bar q}\otimes\iota_X) \\
&\qquad\qquad\qquad\qquad\qquad\qquad\qquad\qquad\qquad\qquad\qquad\qquad (u^{\alpha*}_{tq}\otimes u^{\alpha\vee}_{tq}\otimes\iota_X)(\iota_X\otimes\iota_t\otimes\iota_{\bar t}\otimes R_X)\\
&=\sum_{\alpha,q} d_p^{1/2}d_t^{1/2}(\iota_p\otimes\iota_{\bar p}\otimes \bar R^*_q\otimes\iota_X)
(\iota_p\otimes\iota_{\bar p}\otimes u^{\alpha*}_{tq}\otimes u^{\alpha\vee}_{tq}\otimes\iota_X)(\iota_p\otimes M^\phi_{\bar U_p\otimes X\otimes U_t}\otimes\iota_{\bar t}\otimes\iota_{\bar X}\otimes\iota_X)\\
&\qquad\qquad\qquad\qquad\qquad\qquad\qquad\qquad\qquad\qquad\qquad\qquad (\bar R_p\otimes\iota_X\otimes\iota_t\otimes\iota_{\bar t}\otimes R_X).
\end{align*}
Now observe that by definition of $u^{\alpha\vee}_{tq}$ we have
$$
\sum_\alpha \bar R^*_q(u^{\alpha*}_{tq}\otimes u^{\alpha\vee}_{tq})=\sum_\alpha\bar R^*_{X\otimes U_t}(u^{\alpha}_{tq}u^{\alpha*}_{tq}\otimes\iota_{\bar t}\otimes\iota_{\bar X})
=\bar R^*_{X\otimes U_t}(p^{X\otimes U_t}_q\otimes\iota_{\bar t}\otimes\iota_{\bar X}),
$$
where $p^U_q$ denotes the projection onto the isotypic component of $U$ corresponding to $U_q$. Taking the summation over $q$ we conclude that the right hand side of \eqref{eq:intertwin} equals
$$
d_p^{1/2}d_t^{1/2}(\iota_p\otimes\iota_{\bar p}\otimes \bar R^*_{X\otimes U_t}\otimes\iota_X)(\iota_p\otimes M^\phi_{\bar U_p\otimes X\otimes U_t}\otimes\iota_{\bar t}\otimes\iota_{\bar X}\otimes\iota_X)(\bar R_p\otimes\iota_X\otimes\iota_t\otimes\iota_{\bar t}\otimes R_X).
$$
Recalling that $\bar R_{X\otimes U_t}=(\iota_X\otimes \bar R_t\otimes\iota_{\bar X})\bar R_X$, we see that this expression equals
\begin{equation}\label{eq:intertwin1}
d_p^{1/2}d_t^{1/2}(\iota_p\otimes\iota_{\bar p}\otimes\iota_X\otimes \bar R^*_t)(\iota_p\otimes M^\phi_{\bar U_p\otimes X\otimes U_t}\otimes\iota_{\bar t})(\bar R_p\otimes\iota_X\otimes\iota_t\otimes\iota_{\bar t}).
\end{equation}

Note that in the particular case of $\phi=\delta_e$, when $A_{st}=\delta_{st}\iota$ by Example~\ref{ex:deltae}, the equality of \eqref{eq:intertwin1} to the right hand side of \eqref{eq:intertwin} gives the identity
\begin{equation}\label{eq:regbr}
c_{X,pt}=d_p^{1/2}d_t^{1/2}(\iota_p\otimes\iota_{\bar p}\otimes\iota_X\otimes \bar R^*_t)(\iota_p\otimes p^{\bar U_p\otimes X\otimes U_t}_e\otimes\iota_{\bar t})(\bar R_p\otimes\iota_X\otimes\iota_t\otimes\iota_{\bar t}).
\end{equation}

Now, using this identity we see that the left hand side of \eqref{eq:intertwin} equals
\begin{align*}
&\sum_{s} d_p^{1/2}d_sd_t^{1/2}(\iota_p\otimes\iota_{\bar p}\otimes\iota_X\otimes \bar R^*_s)(\iota_p\otimes p^{\bar U_p\otimes X\otimes U_s}_e\otimes\iota_{\bar s})(\bar R_p\otimes\iota_X\otimes\iota_s\otimes\iota_{\bar s})\\
&\qquad\qquad\qquad\qquad
(\iota_X\otimes\iota_s\otimes\iota_{\bar s}\otimes \bar R^*_t)(\iota_X\otimes\iota_s\otimes M^\phi_{\bar U_s\otimes U_t}\otimes\iota_{\bar t})(\iota_X\otimes\bar R_s\otimes\iota_t\otimes\iota_{\bar t})\\
&=\sum_s d_p^{1/2}d_sd_t^{1/2}(\iota_p\otimes\iota_{\bar p}\otimes\iota_X\otimes \bar R^*_t)
(\iota_p\otimes\iota_{\bar p}\otimes\iota_X\otimes\bar R^*_s\otimes\iota_t\otimes\iota_{\bar t})
(\iota_p\otimes p^{\bar U_p\otimes X\otimes U_s}_e\otimes\iota_{\bar s}\otimes\iota_t\otimes\iota_{\bar t})\\
&\qquad\qquad\qquad\qquad
(\iota_p\otimes\iota_{\bar p}\otimes\iota_X\otimes\iota_s\otimes M^\phi_{\bar U_s\otimes U_t}\otimes\iota_{\bar t})
(\iota_p\otimes\iota_{\bar p}\otimes\iota_X\otimes\bar R_s\otimes\iota_t\otimes\iota_{\bar t})
(\bar R_p\otimes\iota_X\otimes\iota_t\otimes\iota_{\bar t}).
\end{align*}
In order to show that this expression equals \eqref{eq:intertwin1} it suffices to show that
$$
M^\phi_{\bar U_p\otimes X\otimes U_t}=\sum_s d_s(\iota_{\bar p}\otimes\iota_X\otimes\bar R^*_s\otimes\iota_t)
(p^{\bar U_p\otimes X\otimes U_s}_e\otimes\iota_{\bar s}\otimes\iota_t)(\iota_{\bar p}\otimes\iota_X\otimes\iota_s\otimes M^\phi_{\bar U_s\otimes U_t})
(\iota_{\bar p}\otimes\iota_X\otimes\bar R_s\otimes\iota_t).
$$
By naturality of $M^\phi$ we can rewrite the right hand side as
\begin{multline*}
\sum_s d_s(\iota_{\bar p}\otimes\iota_X\otimes\bar R^*_s\otimes\iota_t)
(p^{\bar U_p\otimes X\otimes U_s}_e\otimes\iota_{\bar s}\otimes\iota_t)M^\phi_{\bar U_p\otimes X\otimes U_s\otimes\bar U_s\otimes U_t}
(\iota_{\bar p}\otimes\iota_X\otimes\bar R_s\otimes\iota_t)\\
=M^\phi_{\bar U_p\otimes X\otimes U_t}\sum_s d_s(\iota_{\bar p}\otimes\iota_X\otimes\bar R^*_s\otimes\iota_t)
(p^{\bar U_p\otimes X\otimes U_s}_e\otimes\iota_{\bar s}\otimes\iota_t)
(\iota_{\bar p}\otimes\iota_X\otimes\bar R_s\otimes\iota_t).
\end{multline*}
Therefore it remains to check that
$$
\sum_s d_s(\iota_{\bar p}\otimes\iota_X\otimes\Tr_s)(p^{\bar U_p\otimes X\otimes U_s}_e)=\iota_{\bar p}\otimes\iota_X.
$$
But this is clearly true, since for any $q$ we have
$$
\sum_s d_s(\iota_q\otimes\Tr_s)(p^{U_q\otimes U_s}_e)=d_q(\iota_q\otimes\Tr_{\bar q})(p^{U_q\otimes\bar U_q}_e)=\iota_q.
$$
Thus we have proved that both sides of \eqref{eq:intertwin} are equal to \eqref{eq:intertwin1}. This completes the construction of the unitary~$c_{\phi,X}$. Naturality of this construction and the half-braiding condition easily follow from the corresponding properties of $c$.

Consider the representation $\pi_\phi$ of $C^*(\CC)$ defined by $(Z_\phi,c_\phi)$. Denote by $\xi_\phi$ the canonical morphism $\un=U_e\otimes\bar U_e\to Z_\phi=A^\phi\text{-}\oplus_s U_s\otimes\bar U_s$. Then
$$
\phi(s)=d_s^{-1}(\pi_\phi([U_s])\xi_\phi,\xi_\phi).
$$
Indeed, by Lemma~\ref{lem:Mphi} we have
$$
d_s^{-1}(\pi_\phi([U_s])\xi_\phi,\xi_\phi)\iota_s=(\xi_\phi^*\otimes\iota_s)c_{\phi,s}(\iota_s\otimes\xi_\phi).
$$
By \eqref{eq:newscalarpr1} the last expression equals
$
\sum_t(A^\phi_{et}\otimes\iota_s)c_{s,te}.
$
Recalling that $c_{s,te}=\delta_{st}d_s^{-1/2}(\iota_s\otimes R_s)$ by \eqref{eq:matrixte} and using that $A^\phi_{et}=d_t^{1/2}\phi(t)\bar R^*_t$, we get
$$
\sum_t(A^\phi_{et}\otimes\iota_s)c_{s,te}=\phi(s)\iota_s,
$$
and the proof of the theorem is complete.
\ep

The triples $(Z_\phi,c_\phi,\xi_\phi)$ constructed in the proof of Theorem~\ref{thm:positivedef} have the following universal property.

\begin{proposition}\label{prop:Zphi-universal}
Let $(Z,c)$ be an object in $\ZC$ and $\xi\in\Mor_\indC(\un,Z)$. Consider the positive definite function $\phi(s)=d_s^{-1}(\pi_Z([U_s])\xi,\xi)$. Then there exists a unique isometric morphism $T\colon (Z_\phi,c_\phi)\to (Z,c)$ in $\ZC$ such that $T\xi_\phi=\xi$.
\end{proposition}

\bp Let us first prove the uniqueness. Assume $T\colon (Z_\phi,c_\phi)\to (Z,c)$ is a morphism. Denote by $T_s$ the composition of $T$ with the canonical morphism $U_s\otimes\bar U_s\to Z_\phi$, so in particular we have $T_e=T\xi_\phi$. Clearly, the morphism $T$ is completely determined by the morphisms $T_s$, so we just have to check that $T_s$ is determined by $T_e$. By Lemma~\ref{lem:indunitary} and formula~\eqref{eq:matrixte} for the half-braiding on $\Zreg$, we have the commutative diagram
$$
\begin{xymatrix}{
U_s\ar[d]_{\iota_s\otimes\xi_\phi}\ar[rr]^{d_s^{-1/2}(\iota_s\otimes R_s)\qquad\ \ } & & (U_s\otimes\bar U_s)\otimes U_s\ar[d]\\
U_s\otimes Z_\phi\ar[rr]_{c_{\phi,s}}& & Z_\phi\otimes U_s.
}\end{xymatrix}
$$
Applying $T$ and using that $c_s(\iota_s\otimes T)=(T\otimes\iota_s)c_{\phi,s}$, we get $c_s(\iota_s\otimes T_e)=d_s^{-1/2}(T_s\otimes\iota_s)(\iota_s\otimes R_s)$, that is,
\begin{equation} \label{eq:intertbr}
T_s = d_s^{1/2} (\iota_Z \otimes \bar{R}_s^*) ( c_s \otimes \iota_{\bar s}) (\iota_s \otimes T_e \otimes \iota_{\bar s}).
\end{equation}
Thus $T_s$ is indeed determined by $T_e$.

\smallskip

For the existence, we let $T_e=\xi$ and define $T_s\colon U_s\otimes\bar U_s\to Z$ by \eqref{eq:intertbr}. In order to show that the morphisms~$T_s$ define an isometry $T\colon Z_\phi\to Z$, by Lemma~\ref{lem:indiso} it suffices to check that for any finite set $F\subset\Irr(\CC)$ for the morphism $T_F=(T_s)_{s\in F}\colon\oplus_{s\in F}U_s\otimes\bar U_s\to Z$ we have $T_F^*T_F=(A^\phi_{s,t})_{s,t\in F}$, that is, $T_s^*T_t=A^\phi_{st}$. But this is exactly the computation we made in the proof of the implication (ii)$\Rightarrow$(i) in Theorem~\ref{thm:positivedef}.

\smallskip

It remains to check that $T$ intertwines the half-braidings. By the construction of $T$ we already have
$$
c_X(\iota_X\otimes T)(\iota_X\otimes\xi_\phi)=(T\otimes\iota_X)c_{\phi,X}(\iota_X\otimes\xi_\phi)
$$
for $X=U_s$, hence for all $X\in\CC$. Applying this to $X\otimes U_s$ in place of $X$ and using the multiplicativity property of half-braidings, we get
\begin{equation} \label{eq:intertwin2}
(c_X\otimes\iota_s)(\iota_X\otimes T\otimes\iota_s)(\iota_X\otimes B_s)
=(T\otimes\iota_X\otimes\iota_s)(c_{\phi,X}\otimes\iota_s)(\iota_X\otimes B_s),
\end{equation}
where $B_s=c_{\phi,s}(\iota_s\otimes\xi_\phi)\colon U_s\to Z_\phi\otimes U_s$. But by equation~\eqref{eq:intertbr} for the identity map on $Z_\phi$ we know that
$$
d_s^{1/2} (\iota_{Z_\phi} \otimes \bar{R}_s^*) ( B_s \otimes \iota_{\bar s})
$$
is the canonical morphism $U_s\otimes\bar U_s\to Z_\phi$. Since \eqref{eq:intertwin2} holds for all $s$, we can therefore conclude that $c_X(\iota_X\otimes T)=(T\otimes\iota_X)c_{\phi,X}$.
\ep

From the proof we also get the following.

\begin{corollary} \label{cor:cyclic}
The vector $\xi_\phi$ is cyclic for the representation $\pi_\phi$ of $C^*(\CC)$ defined by $(Z_\phi,c_\phi)$.
\end{corollary}

\bp Let $\xi_{\phi,s}$ be the morphism $\un\to Z_\phi$ obtained by composing $d_s^{-1/2}\bar R_s\colon \un\to U_s\otimes\bar U_s$ with the canonical morphism $U_s\otimes\bar U_s\to Z_\phi$. From the construction of $Z_\phi$ one can see that the vectors $\xi_{\phi,s}$, $s\in\Irr(\CC)$, span a dense subspace of $\Mor_{\indC}(\un,Z_\phi)$. On the other hand, from equality~\eqref{eq:intertbr} for the identity morphism $T$ on~$Z_\phi$ we have
$$
\xi_{\phi,s}=d_s^{-1/2}T_s\bar R_s=(\iota_Z \otimes \bar{R}_s^*) ( c_s \otimes \iota_{\bar s}) (\iota_s \otimes \xi_\phi \otimes \iota_{\bar s})\bar R_s=\pi_\phi([U_s])\xi_\phi.
$$
Hence the vector $\xi_\phi$ is indeed cyclic.
\ep

By decomposing representations of $C^*(\CC)$ into direct sums of cyclic representations, we now obtain the following result.

\begin{corollary} \label{cor:tworepclasses}
Any representation of $C^*(\CC)$ is unitarily equivalent to the representation $\pi_Z$ defined by an object $(Z,c)\in\ZC$.
\end{corollary}

This result can also be formulated as follows. Let us say that an object $(Z,c)\in\ZC$ is \emph{spherical} if for any $T\in\End_\ZC((Z,c))$ such that $T\xi=0$ for all $\xi\in\Mor_\indC(\un,Z)$ we have $T=0$. Such objects form a full C$^*$-subcategory $\ZsC$ of $\ZC$ closed under direct sums and subobjects. Note that in general this is not a tensor category.

\begin{proposition}
The category $\ZsC$ is unitarily equivalent to the representation category of $C^*(\CC)$.
\end{proposition}

\bp
Consider the unitary functor $F\colon\ZsC\to\Rep C^*(\CC)$ mapping an object $(Z,c)$ into the corresponding representation $\pi_{(Z,c)}$ of $C^*(\CC)$. By the definition of spherical objects this functor is faithful. Since the objects $(Z_\phi,c_\phi)$ are spherical by Proposition~\ref{prop:Zphi-universal}, this functor is also essentially surjective. It remains to show that it is full. By Proposition~\ref{prop:Zphi-universal}, any spherical object decomposes into a direct sum of objects $(Z_\phi,c_\phi)$. Therefore it suffices to show that any bounded operator $T\colon\Mor_\indC(\un,Z_\phi)\to \Mor_\indC(\un,Z)$ intertwining the representations $\pi_\phi$ and $\pi_{(Z,c)}$ is defined by a morphism $(Z_\phi,c_\phi)\to (Z,c)$ in $\ZC$. If~$T$~is isometric, this is true, again by Proposition~\ref{prop:Zphi-universal}. The general case follows from this, since any contraction $T\colon H\to K$ between Hilbert spaces can be dilated to an isometry $H\to K\oplus H$, $\xi\mapsto (T\xi,(1-T^*T)^{1/2}\xi)$.
\ep

\begin{example}
Consider the regular representation of $C^*(\CC)$ on $\ell^2(\Irr(\CC))$. Denote by $W^*(\CC)$ the von Neumann algebra generated by $C^*(\CC)$ in this representation. Since the vector $\delta_e$ is a cyclic trace vector for $W^*(\CC)$, the commutant $W^*(\CC)'$ is anti-isomorphic to $W^*(\CC)$. On the other hand, as we essentially observed in  Example~\ref{ex:regrep}, the regular representation corresponds to $\Zreg(\CC)$ under the equivalence $\Rep C^*(\CC)\cong\ZsC$. Hence $\End_{\ZC}(\Zreg(\CC))\cong W^*(\CC)^\opos$.
\end{example}

\subsection{Property (T)}
\label{sec:property-t}

With the algebra $C^*(\CC)$ at our disposal, the following definition is very natural.

\begin{definition}
\label{defn:cat-prop-T}
We say that \emph{$\CC$ has property (T)} if any representation of $C^*(\CC)$ which weakly contains the trivial representation $\pi_\un$, contains $\pi_\un$ as a subrepresentation.
\end{definition}

Given a representation $\pi$ of $C^*(\CC)$ on a Hilbert space $H$, let us say that a vector $\xi\in H$ is \emph{invariant} if $\pi([X])\xi=d(X)\xi$ for all $X\in\CC$. Any nonzero invariant vector gives an embedding of $\pi_\un$ into~$\pi$. More generally, let us say that unit vectors $\xi_i\in H$, indexed by a directed set $I$, are \emph{almost invariant} if $\lim_i(\pi([X])\xi_i,\xi_i)=d(X)$ for all $X\in\CC$.
Since $\|\pi([X])\|\le d(X)$, in this terminology the above definition of property (T) means that if a representation $\pi\colon C^*(\CC)\to B(H)$ has almost invariant vectors, then it has nonzero invariant vectors.

Almost invariance can be phrased in different ways.

\begin{lemma}\label{lem:aa}
Let $\pi=\pi_{(Z,c)}$ be the representation defined by an object $(Z,c)\in\ZC$ and $\{\xi_i\}_i$ be a net of unit vectors in $\Mor_\indC(\un,Z)$. Then the following conditions are equivalent:
\begin{itemize}
\item[(i)] the vectors $\xi_i$ are almost invariant;
\item[(ii)] we have $\lim_i\|\pi([X])\xi_i-d(X)\xi_i\|=0$ for all objects $X\in\CC$;
\item[(iii)] we have $\lim_i\norm{c_X (\iota_X \otimes {\xi}_i) - \xi_i \otimes \iota_X} = 0$ for all objects $X\in\CC$.
\end{itemize}
\end{lemma}

\bp The equivalence of (i) and (ii) is immediate from $\norm{\pi([X])}\le d(X)$. By faithfulness of $\Tr_X$ on $\CC(X)$, condition (iii) is equivalent~to
$$
\Tr_X((c_X (\iota_X \otimes {\xi}_i) - \xi_i \otimes \iota_X)^* (c_X (\iota_X \otimes {\xi}_i) - \xi_i \otimes \iota_X)) \to 0.
$$
Expanding the product inside $\Tr_X$ and using that $(\pi([X])\xi_i,\xi_i)=\Tr_X((\xi^*_i \otimes \iota_X) c_X (\iota_X \otimes \xi_i))$ by the proof of Lemma~\ref{lem:char-alg-act-compat-adj}, we have
$$
\Tr_X((c_X (\iota_X \otimes {\xi}_i) - \xi_i \otimes \iota_X)^* (c_X (\iota_X \otimes {\xi}_i) - \xi_i \otimes \iota_X)) = 2 d(X) - 2 \Re (\pi([X]) \xi_i, \xi_i).
$$
Using once again that $\norm{\pi([X])}\le d(X)$ we conclude that (iii) holds if and only if $(\pi([X]) \xi_i, \xi_i)\to d(X)$, so
(iii) is equivalent to (i).
\ep

As in the group case, there are many equivalent ways of formulating property (T). Let us list some of them.

\begin{proposition} \label{prop:propertyT}
The following conditions are equivalent:
\begin{itemize}
\item[(i)] the category $\CC$ has property (T);
\item[(ii)] there exist a finite set $F\subset \Irr(\CC)$ and $\eps>0$ such that if $\pi\colon C^*(\CC)\to B(H)$ is a representation and $\xi\in H$ is a unit vector such that $|(\pi([U_s])\xi,\xi)- d_s|<\eps$ for all $s\in F$, then there exists an invariant unit vector in $H$;
\item[(iii)] there exists a nonzero projection $p\in C^*(\CC)$ such that $[X]p=d(X)p$ for all $X\in\CC$.
\end{itemize}
\end{proposition}

\bp
(i)$\Rightarrow$(iii) This can be proved in the same way as the existence of Kazhdan projections, see e.g.~\cite{MR2391387}*{Section~17.2}. Consider any family of representations $\{\pi_\lambda\}_\lambda$ of $C^*(\CC)$ without nonzero invariant vectors such that the representation $\pi_\un\oplus(\oplus_\lambda\pi_\lambda)$ is faithful. For every $\lambda$ take countably many copies of $\pi_\lambda$ and denote by~$\sigma$ the direct sum of all these representations for all $\lambda$. We claim that $\sigma$ is not faithful. Assume this is not the case. Then for any separable C$^*$-subalgebra $A\subset C^*(\CC)$ the representation $\sigma|_A$ is faithful and essential, so by Voiculescu's theorem it weakly contains any other representation. Since this is true for any~$A$, we conclude that $\pi_\un$ is weakly contained in $\sigma$. But then $\sigma$ must have nonzero invariant vectors, which is a contradiction. Therefore $J=\ker\sigma\ne0$. Since $\pi_\un\oplus\sigma$ is faithful, it follows that $J=\C p$ for a projection $p$. Clearly, $p$ has the required property.

\smallskip

(iii)$\Rightarrow$(ii) Since $[X]p=d(X)p$ for all $X\in\CC$, we have $ap=\pi_\un(a)p$ for all $a\in C^*(\CC)$, and letting $a=p$ we get $\pi_\un(p)=1$. Now choose an element $x=\sum_{s\in F}\alpha_s[U_s]$, with $F\subset\Irr(\CC)$ finite, such that $\|p-x\|<1/2$. Then $|1-\sum_{s\in F}\alpha_s d_s|<1/2$. We can find $\eps>0$ such that whenever $|\beta_s-d_s|<\eps$ for all $s\in F$, we still have $|1-\sum_{s\in F}\alpha_s \beta_s|<1/2$. Then, assuming that $|(\pi([U_s])\xi,\xi)- d_s|<\eps$ for all $s\in F$, we get a nonzero invariant vector $\pi(p)\xi$. Thus the pair $(F,\eps)$ has the required property.

\smallskip

The implication (ii)$\Rightarrow$(i) is obvious.
\ep

Property (T) for rigid  C$^*$-tensor categories has been also introduced by Popa and Vaes~\cite{MR3406647}.

\begin{corollary} \label{cor:twodefT}
Definition \ref{defn:cat-prop-T} of property (T) is equivalent to the definition of Popa and Vaes.
\end{corollary}

\bp
This follows from Corollary~\ref{cor:twodef}, Proposition~\ref{prop:propertyT} and~\cite{MR3406647}*{Proposition~5.5}.
\ep

\begin{remarks}\mbox{\ }

\noindent
(i) If $G$ is a compact quantum group, then by Remark~\ref{rmk:q-grp-char-corr} the category $\Rep G$ has property (T) if and only if the dual discrete quantum group $\hat G$ has central property (T) in the sense of~\cite{arXiv:1410.6238}. Together with Corollary~\ref{cor:twodefT} this gives an alternative proof of~\cite{MR3406647}*{Proposition~6.3}.

\smallskip\noindent
(ii) Recall that by Corollary~\ref{cor:tworepclasses} any representation of $C^*(\CC)$ is equivalent to a representation of the form $\pi_Z$. If we did not know this, we would have the dilemma of defining property (T) using either all representations or only representations $\pi_Z$. It is, however, not difficult to see that these two approaches are equivalent independently of the results of Section~\ref{sec:positivedefinite}. The key point is that the implication (i)$\Rightarrow$(iii) in Proposition~\ref{prop:propertyT} remains true if we define property (T) using only representations~$\pi_Z$. In order to see this, we have to use representations $\pi_Z$ in the proof, and for this we have to be able to split any representation~$\pi_Z$ into a direct sum of copies of $\pi_\un$ and a representation of the same form without invariant vectors. For this, in turn, we have to show that if $(Z, c)\in\ZC$ and $\xi$ is an invariant unit vector in $\Mor_\indC(\un, Z)$, then $\xi$ is a morphism in $\ZC$, that is, $c_X(\iota\otimes\xi_X)=\xi\otimes\iota_X$ for all $X$. But this is true by Lemma~\ref{lem:aa}. This was our initial approach before the appearance of~\cite{MR3406647}.
\end{remarks}

\bigskip

\section{Categories of Hilbert bimodules}

In this section we give an interpretation of our results and constructions in terms of Hilbert bimodules. Throughout the whole section $M$ denotes a fixed II$_1$-factor.

\subsection{Duality for Hilbert bimodules} \label{sec:Hilbdual}

Let us briefly review a few basic facts from the theory of Hilbert modules, see e.g.~\citelist{\cite{MR1245354}\cite{MR1642584}} for more details.

Denote by $\tau$ the unique tracial state on $M$. Let $X$ be a Hilbert $M$-bimodule, that is, a Hilbert space together with two commuting normal unital representations of $M$ and $M^\opos$ on $X$, where $M^\opos$ is the factor~$M$ with the opposite product. Denote by $\dim(X_M)$ the Murray--von Neumann dimension of $X$ considered as a right $M$-module, so if $X_M\cong pL^2(M)^n$ for a projection $p\in\Mat_n(M)$, then $\dim(X_M)=(\Tr\otimes\tau)(p)$, and if no such $p$ and $n\in\mathbb N$ exist, then $\dim(X_M)=\infty$ . We can similarly define $\dim({}_MX)$.

For a Hilbert $M$-bimodule $X$, consider the subspace $X^0\subset X$ of left bounded vectors, that is, vectors $\xi \in X$ satisfying $(\xi x, \xi) \le c \tau(x)$ for some $c \ge 0$ and all $x\in M_+$. Then $X^0$ is a sub-bimodule, and it admits a unique $M$-valued inner product $\langle\xi, \eta\rangle_M$ (antilinear in $\xi$, linear in $\eta$) satisfying $\langle\xi, \eta x\rangle_M = \langle\xi, \eta\rangle_M x$ and $\tau(\langle\xi, \eta\rangle_M)=(\eta,\xi)$. Then, given another Hilbert $M$-bimodule $Y$, the tensor product $X\otimes_MY$ is defined as the tensor product $X^0\otimes_M Y$ in the sense of Hilbert C$^*$-modules. This way the category $\Hilb_M$ of Hilbert $M$-bimodules becomes a C$^*$-tensor category.

Let us now describe the duality in $\Hilb_M$. Assume $X$ is a Hilbert $M$-bimodule such that both $\dim(X_M)$ and $\dim({_M}X)$ are finite. In this case the spaces of left and right bounded vectors in $X$ coincide and therefore~$X^0$ carries also the structure of a left Hilbert C$^*$-module over $M$, so it is equipped with an $M$-valued inner product $_M\langle\xi, \eta\rangle$ which  is linear in $\xi$ and antilinear in $\eta$. The complex conjugate space $\bar{X}$ is also an $M$-bimodule such that $x \bar{\xi} y = \overline{y^* \xi x^*}$, and the inner products are related by $\langle\bar{\xi}, \bar{\eta}\rangle_M = {}_M\langle\xi, \eta\rangle$. We can choose a basis $\{\rho_i\}^n_{i=1} \subset X^0$ of $X$ as a right $M$-module, in the sense that $\xi = \sum_i \rho_i \langle\rho_i, \xi\rangle_M$ for all $\xi \in X^0$. The formula
$$
\Tr^r_X(T)=\sum_i\tau(\langle \rho_i,T\rho_i\rangle_M)=\sum_i(T\rho_i,\rho_i)
$$
defines a trace on $\End(X_M^0)$, and we have $\Tr^r_X(1)=\dim(X_M)$. Similarly, we have a trace $\Tr^l_X$ on $\End({}_MX^0)$ such that
$$
\Tr^l_X(T)=\sum_j(T\lambda_j,\lambda_j)
$$
for any basis $\{\lambda_j\}^m_{j=1}$ of the left $M$-module $X$.

There exists a unique Hilbert $M$-bimodule map $\bar R_X\colon L^2(M) \to X \otimes_M \bar{X}$ such that $\bar R_X(1)=\sum_i \rho_i \otimes \bar{\rho}_i $. The map $\bar R_X$ does not depend on the choice of $\{\rho_i\}^n_{i=1}$ and the adjoint map is given by $\bar R_X^*(\xi\otimes\bar\eta)={}_M\langle \xi,\eta\rangle$. It follows that
$$
\|\bar R_X\|^2=\bar R^*_X\bar R_X(1)=\sum_i(\rho_i,\rho_i)=\dim(X_M).
$$
Consider now $R_X=\bar R_{\bar X}\colon L^2(M)\to \bar X\otimes_M X$, so $R_X(1)=\sum_j \bar\lambda_j\otimes\lambda_j$. We have
$$
R^*_X(\bar\xi\otimes\eta)=\langle\xi,\eta\rangle_M\ \ \text{and}\ \ \|R_X\|^2=\dim(\bar X_M)=\dim({}_MX).
$$
From this we see that $(\iota\otimes R^*_X)(\bar R_X\otimes\iota)=\iota$ and conclude that $(R_X,\bar R_X)$ is a solution of the conjugate equations for $X$. Thus $\bar X$ is dual to $X$ and
$$
d(X)\le \sqrt{\dim({}_M X) \dim(X_M)}.
$$

In general this is a strict inequality and the solution $(R_X,\bar R_X)$ is not standard. The general criterion of standardness $\bar R^*_X(T\otimes\iota)\bar R_X=R_X^*(\iota\otimes T)R_X$ becomes
\begin{equation} \label{eq:eqtrace}
\Tr^r_X=\Tr^l_X\ \ \text{on}\ \ \End_{M\mhyph M}(X),
\end{equation}
in which case we also have $d(X)=\dim({}_M X)=\dim(X_M)$.

\subsection{Drinfeld center and Longo--Rehren construction}\label{subsec:LR}

Let $\CC$ be a rigid full C$^*$-tensor subcategory of~$\Hilb_M$ such that condition~\eqref{eq:eqtrace} holds for all $X\in\CC$. As usual, we also assume that $\CC$ is closed under finite direct sums and subobjects. We will use the solutions $(R_X,\bar R_X)$ of conjugate equations defined in Section~\ref{sec:Hilbdual}. Note that since these solutions, as well as the left and right bases of Hilbert modules~\cite{MR1642584}*{Proposition~9.62}, behave well with respect to direct sums and tensor products, it suffices to check~\eqref{eq:eqtrace} on a set of bimodules generating the C$^*$-tensor category $\CC$ by taking direct sums, tensor products and subobjects.

\begin{example} \label{ex:subfactorcat}
Assume $N\subset M$ is an extremal finite index subfactor. Consider the corresponding Jones tower $N\subset M\subset M_1\subset\dots$. Let $\CC=\CC_{N\subset M}$ be the full C$^*$-tensor subcategory of $\Hilb_M$ generated by the module~$L^2(M_1)$. Then~$\CC$ satisfies the above assumptions. Indeed, since the Hilbert $M$-bimodule $L^2(M_1)$ is self-dual, the category $\CC$ is rigid. Next, we have a canonical isomorphism $\End({}_ML^2(M_1)_M)\cong M'\cap M_{2}$. Under this isomorphism the traces $\Tr^r$ and $\Tr^l$ on $\End({}_ML^2(M_1)_M)$ introduced in Section~\ref{sec:Hilbdual} coincide, up to the factor $[M:N]$, with the restriction of the tracial states on $M_{2}$ and $M'\subset B(L^2(M_1))$, respectively, to $M'\cap M_{2}$. The extremality assumption implies that these traces are equal, so the condition~\eqref{eq:eqtrace} holds for all $X\in\CC$.

In fact, any finitely generated rigid C$^*$-tensor category $\CC$ (satisfying our standard assumptions) is unitarily monoidally equivalent to a category of the form $\CC_{N\subset M}$ for some $M$ and $N$. More precisely, taking an object $X\in\CC$ such that $X\otimes\bar X$ is a generating object, we have a standard $\lambda$-lattice consisting of the algebras $\CC(X \otimes \bar{X} \otimes \cdots)$ ($n$ factors) and $\CC(\bar{X} \otimes X \otimes \cdots)$ ($n-1$ factors), with the Jones projections given by copies of $\frac{1}{d(X)} R_X R_X^*$ and  $\frac{1}{d(X)} \bar{R}_X \bar{R}_X^*$. Then a result of Popa~\cite{MR1334479} gives an extremal finite index subfactor $N \subset M$ such that $\CC_{N\subset M}$ is unitarily monoidally equivalent to $\CC$, and such that under this equivalence the module $L^2(M_1)$ corresponds to $X \otimes \bar{X}$.
\end{example}

With every category $\CC$ as above one can associate an inclusion $A\subset \tB$ of II$_1$-factors, called the Longo-Rehren inclusion~\citelist{\cite{MR1332979} \cite{MR1742858}}. Namely,
put $A=M\bar\otimes M^\opos$. Recall that there is a functorial construction of a Hilbert $M^\opos$-bimodule $X^\natural$ from a Hilbert $M$-bimodule $M$:  as a linear space we put $X^\natural = \bar{X}$, and then define the bimodule structure by $x\bar\xi=\overline{x^*\xi}$ and $\bar\xi x=\overline{\xi x^*}$. With this definition we have a natural identification of $(X \otimes_M Y)^\natural$ with $ X^\natural \otimes_{M^\opos} Y^\natural$. Thus, $X \mapsto X^\natural$ is a monoidal functor which is antilinear on morphisms.
Now, choose a complete system of representatives of isomorphism classes of simple modules  $\{X_s\}_{s \in I}$ in $\CC$. Then $\tB$ is generated by the spaces $X_s^0 \otimes X_s^{0\natural}$, with the product
\begin{equation}\label{eq:prod-in-tB}
(\xi_1 \otimes \bar \xi_2) \cdot (\eta_1 \otimes \bar \eta_2) = \sum_{\alpha} \left(\frac{d_s d_t}{d_{r_\alpha}}\right)^{\hlf} w^\alpha_{st}(\xi_1 \otimes \eta_1) \otimes \overline{w^{\alpha}_{st}(\xi_2 \otimes \eta_2)}
\end{equation}
for $\xi_1 \otimes \bar\xi_2 \in X_s^0 \otimes X_s^{0\natural}$ and $\eta_1 \otimes \bar\eta_2 \in X_t^0 \otimes X_t^{0\natural}$,
where $\{w^\alpha_{st}\}_\alpha$ is any family of coisometries $X_s\otimes_MX_t\to X_{r_\alpha}$ defining a decomposition of $X_s\otimes_M X_t$ into simple bimodules. The $*$-structure is defined by
$$
(\xi_1 \otimes \bar\xi_2)^*=\bar\xi_1\otimes \xi_2,
$$
where we identify $\bar X_s^0\otimes\bar X_s^{0\natural}$ with a subspace of $\tB$ using the map $\bar\xi_1\otimes \xi_2\mapsto J_s\bar\xi_1\otimes \overline{J_s\bar \xi_2}$, where $J_s$ is any unitary isomorphism of Hilbert $M$-bimodules $\bar X_s\to X_{\bar s}$. If $e\in I$ corresponds to $L^2(M)$, then $A$ is identified with the subalgebra formed by the bounded vectors in $X_e\otimes X_e^\natural$. The projection onto this subalgebra composed with the trace on $A$ defines a tracial state on $\tB$; it is worth noting that this is the point where condition~\eqref{eq:eqtrace} is used. By construction $L^2(\tB)$ decomposes into the direct sum of the simple Hilbert $A$-bimodules $X_s\otimes X_s^\natural$. In particular, $A'\cap \tB=\C1$ and so $\tB$ is a II$_1$-factor.

\begin{remark}
For $\CC=\CC_{N\subset M}$ as in Example~\ref{ex:subfactorcat}, Masuda~\cite{MR1742858}  proved that $\tB$ is isomorphic to the symmetric enveloping algebra $M \boxtimes_{e_N} M^\opos$ of Popa~\cite{MR1302385}.
\end{remark}

Our goal is to prove the following result.

\begin{theorem} \label{thm:mon-eqv-D-center-B-bimod}
Let $\CC\subset\Hilb_M$ be a C$^*$-tensor category as above and $A\subset\tB$ be the associated Longo--Rehren inclusion. Then $\ZC$ is unitarily monoidally equivalent to the full subcategory $\ZHC$ of $\Hilb_\tB$ consisting of the Hilbert $\tB$-bimodules $X$ such that as a Hilbert $A$-bimodule $X$ decomposes into a direct sum of copies of~$X_s\otimes X_t^\natural$.
\end{theorem}

Note that it is not immediately obvious, but will become clear from the proof, that $\ZHC$ is a tensor category. Let us also remark that the objects of $\ZHC$ can equivalently be characterized as Hilbert bimodules that are generated, as $\tB$-bimodules, by $A$-sub-bimodules isomorphic to $X_s\otimes X_t^\natural$. In order to see this, it suffices to show that given a Hilbert $\tB$-bimodule $H$ and a copy of $X_s\otimes X_t^\natural$ in ${}_AH_A$, the $\tB$-bimodule structure defines bounded maps $L^2(\tB)\otimes_A(X_s\otimes X_t^\natural)\to H$ and
$(X_s\otimes X_t^\natural)\otimes_AL^2(\tB)\to H$. This, in turn, follows from the following general result (compare with \cite{MR3406647}*{Lemma~2.8}).

\begin{lemma}\label{lem:modstructure}
Assume that $P\subset Q$ is an irreducible inclusion of II$_1$-factors, $H$ is a Hilbert $Q$-$P$-module, and $X\subset H$ is a Hilbert $P$-sub-bimodule such that $\dim(X_P)$ and $\dim({}_PX)$ are finite. Then the map $Q\otimes X\ni a\otimes\xi\mapsto a\xi$ extends to a bounded map $L^2(Q)\otimes_PX\to H$.
\end{lemma}

\bp Choose a basis $\{\rho_i\}^n_{i=1}$ of $X_P$ and a basis $\{\lambda_j\}^m_{j=1}$ of ${}_PX$. Define a normal positive linear functional~$\psi$ on~$Q$ by
$$
\psi(a)=\sum_i(a\rho_i,\rho_i).
$$
The standard argument shows that $\psi$ is independent of the choice of $\{\rho_i\}_i$: if $\{\rho'_k\}_k$ is another basis, then
$$
\sum_k(a\rho'_k,\rho'_k)=\sum_{i,k}(a\rho_i\langle\rho_i,\rho'_k\rangle_P,\rho'_k)
=\sum_{i,k}(a\rho_i,\rho'_k\langle\rho'_k,\rho_i\rangle_P)=\sum_i(a\rho_i,\rho_i).
$$
In particular, since for any unitary $u\in P$ we can take the basis $\{u\rho_i\}^n_{i=1}$, we see that $u$ is contained in the centralizer of $\psi$. Since $P'\cap Q=\C1$, we conclude that $\psi$ coincides, up to a scalar factor, with the tracial state~$\tau$ on~$Q$ (e.g.~ because Connes' Radon--Nikodym cocycle $[D\psi : D\tau]_t$ lies in $P'\cap Q$). Thus $\psi=\dim(X_P)\tau$. It follows that for every $i=1,\dots,n$ the map $Q\ni a\mapsto a\rho_i$ extends to a bounded map $T_i\colon L^2(Q)\to H$ satisfying $\norm{T_i} \le \sqrt{\dim(X_P)}$.

Next, for every $j=1,\dots,m$ consider the map $L_j\colon L^2(Q)\to L^2(Q)\otimes_P X$ defined by $L_j\zeta=\zeta\otimes\lambda_j$. This map is bounded, and the adjoint map is given by
$$
L_j^*(\zeta\otimes\xi)=\zeta{}_P\langle\xi,\lambda_j\rangle\ \ \text{for}\ \ \zeta\in L^2(Q),\ \xi\in X^0.
$$
For $a\in Q$ and $\xi\in X^0$ we then have
$$
a\xi=\sum_{i,j}T_iL_j^*(a\otimes\xi)\langle\rho_i,\lambda_j\rangle_P,
$$
which shows that the map $a\otimes\xi\mapsto a\xi$ is bounded on $L^2(Q)\otimes_PX$.
\ep

We will often use the following identities, which are immediate from the definition of $\tB$.

\begin{lemma} \label{lem:prodrelations}
If $\bar R_s(1)=\sum_\alpha\rho_{s\alpha}\otimes\bar\rho_{s\alpha}$ and $R_s(1)=\sum_\beta\bar\lambda_{s\beta}\otimes\lambda_{s\beta}$, then for any $\xi,\eta\in X_s^0$ we have the following identities in $\tB$:
$$
d_s^{-1}\sum_\alpha (\rho_{s\alpha}\otimes\bar\xi)(\bar\rho_{s\alpha}\otimes\eta)=1\otimes{}_M\langle\eta,\xi\rangle,\ \
d_s^{-1}\sum_\beta (\bar\lambda_{s\beta}\otimes\xi)(\lambda_{s\beta}\otimes\bar\eta)=1\otimes\langle\eta,\xi\rangle_M,
$$
$$
d_s^{-1}\sum_\alpha (\xi\otimes\bar\rho_{s\alpha})(\bar\eta\otimes\rho_{s\alpha})={}_M\langle\xi,\eta\rangle\otimes1,\ \
d_s^{-1}\sum_\beta (\bar\xi\otimes\lambda_{s\beta})(\eta\otimes\bar\lambda_{s\beta})=\langle\xi,\eta\rangle_M\otimes1.
$$
\end{lemma}

Turning to the proof of the theorem, we start by constructing a half-braiding from a Hilbert $\tB$-bimodule.
The following observation will play a crucial role.

\begin{proposition}\label{prop:move-Xi-thru-bimod}
For any Hilbert $\tB$-bimodule $H$ and any $s \in I$, there are unitary isomorphisms of Hilbert $A$-bimodules
 $$
 (X_s \otimes L^2(M^\opos)) \otimes_A H \cong (L^2(M) \otimes \bar{X}_s^\natural) \otimes_A H
 \ \ \text{and}\ \ H \otimes_A (X_s \otimes L^2(M^\opos)) \cong H \otimes_A (L^2(M) \otimes \bar{X}_s^\natural).
 $$
\end{proposition}

\bp
Using the isomorphism $L^2(M^\opos)\cong L^2(M)^\natural$ and the isometry
 $d_s^{-1/2}R_s^\natural\colon L^2(M)^\natural \to \bar{X}_s^\natural \otimes_{M^\opos} X_s^\natural$ , we get a map
$$
(X_s \otimes L^2(M^\opos)) \otimes_A H \to (X_s \otimes (\bar{X}_s^\natural \otimes_{M^\opos} X_s^\natural)) \otimes_A H
\cong (L^2(M)\otimes \bar X_s^\natural)\otimes_A(X_s\otimes X_s^\natural)\otimes_A H.
$$
Now, the $\tB$-module structure gives us a map $(X^0_s \otimes X_s^{0\natural})\otimes H \to H$, so in combination with the above we get a map
\begin{equation}\label{eq:isomodule1}
(X_s \otimes L^2(M^\opos)) \otimes_A H\to (L^2(M) \otimes \bar{X}_s^\natural) \otimes_A H,\quad
(\xi\otimes1)\otimes\zeta\mapsto d_s^{-1/2}\sum_\beta(1\otimes\lambda_{s\beta})\otimes (\xi\otimes\bar\lambda_{s\beta})\zeta.
\end{equation}
Since generally the $\tB$-module structure does not define a bounded map $L^2(B)\otimes_A H\to H$, we still have to check that the above map is well-defined and isometric. For $\xi,\xi'\in X_s^0$ and $\zeta,\zeta'\in H$ we compute:

\smallskip
$\displaystyle
d_s^{-1}\sum_{\beta,\beta'}\big((1\otimes\lambda_{s\beta})\otimes (\xi\otimes\bar\lambda_{s\beta})\zeta,(1\otimes\lambda_{s\beta'})\otimes (\xi'\otimes\bar\lambda_{s\beta'})\zeta'\big)
$
\begin{align*}
&=d_s^{-1}\sum_{\beta,\beta'}\big((1\otimes{}_M\langle\lambda_{s\beta},\lambda_{s\beta'}\rangle)(\xi\otimes\bar\lambda_{s\beta})\zeta,(\xi'\otimes\bar\lambda_{s\beta'})\zeta'\big)\\
&=d_s^{-1}\sum_{\beta}\big((\xi\otimes\bar\lambda_{s\beta'})\zeta,(\xi'\otimes\bar\lambda_{s\beta'})\zeta'\big)\\
&=d_s^{-1}\sum_{\beta}\big((\bar\xi'\otimes\lambda_{s\beta'})(\xi\otimes\bar\lambda_{s\beta'})\zeta,\zeta'\big)\\
&=\big((\langle\xi',\xi\rangle_M\otimes1)\zeta,\zeta'\big),
\end{align*}
where in the last step we used Lemma~\ref{lem:prodrelations}. Thus the map \eqref{eq:isomodule1} is indeed isometric.

Similarly we get a map
$$
(L^2(M) \otimes \bar{X}_s^\natural) \otimes_A H\to (X_s \otimes L^2(M^\opos)) \otimes_A H
$$
such that
\begin{equation}\label{eq:isomodule2}
(1\otimes\xi)\otimes\zeta\mapsto d_s^{-1/2}\sum_\alpha(\rho_{s\alpha}\otimes 1)\otimes (\bar\rho_{s\alpha}\otimes\xi)\zeta.
\end{equation}
Using again Lemma~\ref{lem:prodrelations} it is easy to check that the maps \eqref{eq:isomodule1} and \eqref{eq:isomodule2} are inverse to each other. This proves the first isomorphism in the formulation of the lemma.

The second isomorphism is proved similarly. Namely, the map
$$
H \otimes_A (X_s \otimes L^2(M^\opos)) \to H \otimes_A (L^2(M) \otimes \bar{X}_s^\natural)
$$
is defined by
\begin{equation*}\label{eq:isomodule3}
\zeta\otimes(\xi\otimes1)\mapsto d_s^{-1/2}\sum_\alpha\zeta(\xi\otimes\bar\rho_{s\alpha})\otimes (1\otimes\rho_{s\alpha}),
\end{equation*}
and the inverse map is given by
\begin{equation*}\label{eq:isomodule4}
\zeta\otimes(1\otimes\xi)\mapsto d_s^{-1/2}\sum_\beta\zeta(\bar\lambda_{s\beta}\otimes\xi)\otimes (\lambda_{s\beta}\otimes1).
\end{equation*}
This proves the assertion.
\ep

Let $H$ be a Hilbert $\tB$-bimodule. Denote by $X_{H,s}$ the space of $M^\opos$-bimodule homomorphisms from $X_s^\natural$ to $H$. It has a natural inner product $(T, S) = S^* T \in \Hom_{M^\opos\text{-}M^\opos}(X_s^\natural, X_s^\natural) = \C$ and inherits the structure of a $M$-bimodule from $H$. We also put $X_H = X_{H,e}$ and identify $X_H$ with the space of $M^\opos$-central vectors in $H$. Define maps
$$
l_s\colon X^0_s\otimes X_H\to X_{H,s}\ \ \text{and}\ \ r_s\colon X_H\otimes X_s^0\to X_{H,s}
$$
by
$$
l_s(\xi\otimes\zeta)\bar\eta=(\xi\otimes\bar\eta)\zeta\ \ \text{and}\ \ r_s(\zeta\otimes\xi)\bar\eta=\zeta(\xi\otimes\bar\eta)\ \ \text{for}\ \ \eta\in X^0_s.
$$

\begin{lemma} \label{lem:lr-iso}
The maps $l_s$ and $r_s$ define unitary isomorphisms of Hilbert $M$-bimodules
$$
X_s\otimes_M X_H\cong X_{H,s}\ \ \text{and}\ \  X_H\otimes_M X_s\cong X_{H,s},
$$
which we denote by the same symbols $l_s$ and $r_s$.
\end{lemma}

\bp
We have unitary isomorphisms
$$
X_{H,s}\cong\Hom_{M^\opos\text{-}M^\opos}(L^2(M)^\natural, \bar X_s^\natural\otimes_{M^\opos}H)\cong \Hom_{M^\opos\text{-}M^\opos}(L^2(M)^\natural, X_s\otimes_MH),
$$
where the first isomorphism is the normalized Frobenius isomorphism $T\mapsto d_s^{-1/2}(\iota\otimes T)R^\natural_s$ and the second comes from Proposition~\ref{prop:move-Xi-thru-bimod}. Note that the $M^\opos$-bimodule structure on $X_s\otimes_MH$ is defined by that on $H$. It follows that the space of $M^\opos$-central vectors in $X_s\otimes_MH$ coincides with $X_s\otimes_M X_H$. We thus get a unitary isomorphism $X_{H,s}\cong X_s\otimes_M X_H$. Explicitly, using the formula for $R_s$ and \eqref{eq:isomodule2}, this isomorphism is given by
\begin{equation*}\label{eq:li-inverse}
T\mapsto d_s^{-1}\sum_{\alpha,\beta}\rho_{s\alpha}\otimes (\bar\rho_{s\alpha}\otimes\lambda_{s\beta})T\bar\lambda_{s\beta}.
\end{equation*}
Using Lemma~\ref{lem:prodrelations} it is straightforward to check that $l_s$ defines a right inverse of this map.

Similarly it is proved that $r_s$ defines a unitary isomorphism $X_H\otimes_M X_s\cong X_{H,s}$, with the inverse given by
\begin{equation}\label{eq:ri-inverse}
T\mapsto d_s^{-1}\sum_{\alpha,\beta}(T\bar\rho_{s\alpha})(\bar\lambda_{s\beta}\otimes\rho_{s\alpha})\otimes\lambda_{s\beta}.
\end{equation}
This proves the assertion.
\ep

By the previous lemma we obtain a unitary isomorphism of Hilbert $M$-bimodules $c^H_s=r_s^*l_s\colon X_s \otimes_M X_H \to X_H \otimes_M X_s$. Explicitly, using formula~\eqref{eq:ri-inverse} for $r_s^*$, we have
\begin{equation} \label{eq:halfbr-bimod}
c^H_s(\xi\otimes\zeta)=d_s^{-1}\sum_{\alpha,\beta}(\xi\otimes\bar\rho_{s\alpha})\zeta(\bar\lambda_{s\beta} \otimes\rho_{s\alpha})\otimes\lambda_{s\beta}.
\end{equation}
In a more invariant form we can say that $c^H_s$ is characterized by the identity
\begin{equation*} \label{eq:halfbr-bimod1}
(\xi\otimes\bar\eta)\zeta=c^H_s(\xi\otimes\zeta)_1(c^H_s(\xi\otimes\zeta)_2\otimes\bar\eta)\ \ \text{for}\ \ \xi,\eta\in X^0_s,\ \zeta\in X_H,
\end{equation*}
where we use Sweedler's sumless notation $c^H_s(\xi\otimes\zeta)_1\otimes c^H_s(\xi\otimes\zeta)_2$ for $c^H_s(\xi\otimes\zeta)$. Equivalently, $c^{H*}_s$ is characterized by
\begin{equation} \label{eq:halfbr-bimod2}
\zeta(\xi\otimes\bar\eta)=(c^{H*}_s(\zeta\otimes\xi)_1\otimes\bar\eta)c^{H*}_s(\zeta\otimes\xi)_2\ \ \text{for}\ \ \xi,\eta\in X^0_s,\ \zeta\in X_H.
\end{equation}

The family of isomorphisms $\{c^H_s\}_{s\in I}$ defines a natural family of unitary isomorphisms $(c^H_X\colon X \otimes X_H \to X_H \otimes X)_{X \in \CC}$ (where we write $\otimes$ for the tensor product $\otimes_M$ in $\Hilb_M$).

\begin{lemma}\label{lem:XH-tensor-central}
The unitaries $c^H_X$ satisfy the half-braiding condition $c^H_{X\otimes Y}=(c^H_X\otimes\iota_Y)(\iota_X\otimes c^H_Y)$.
\end{lemma}

\bp
It suffices to consider $X=X_s$ and $Y=X_t$. But in this case the result follows immediately from the explicit formula~\eqref{eq:halfbr-bimod} using that left and right bases in $X_s\otimes_M X_t$ can be obtained either by taking tensor products of bases in $X_s$ and $X_t$, or by decomposing $X_s\otimes_MX_t$ into direct sums of the simple modules $X_k$ and choosing bases in $X_k$.
\ep

So far we have used only that $H\in\Hilb_{\tB}$, in which case we cannot say much about $X_H\in\Hilb_M$. But if we assume that $H\in\ZHC$, then $X_H$ decomposes into a direct sum of copies of $X_s$, $s\in I$. Consider the full subcategory of $\Hilb_M$ consisting of bimodules allowing such a decomposition. We have an obvious functor from $\indC$ into this category, which is a unitary monoidal equivalence. In order to not introduce yet another notation, in the remaining part of the proof we do not distinguish between these two equivalent categories. Thus, if $H\in\ZHC$, then $X_H\in\indC$ and therefore $(X_H,c^H)\in\ZC$.

Observe also that if $H\in\ZHC$, then we can reconstruct $H$ from $(X_H,c^H)$. Indeed, first of all we have a unitary isomorphism $\oplus_s X_{H,s}\otimes X_s^\natural\cong H$ of Hilbert $A$-bimodules, mapping $T\otimes\bar\eta$ into $T\bar\eta$. So by Lemma~\ref{lem:lr-iso} we get a unitary isomorphism of Hilbert $A$-bimodules
$$
\bigoplus_s(X_s\otimes X_s^\natural)\otimes_A(X_H\otimes L^2(M^\opos))\cong H,\ \ (\xi\otimes\bar\eta)\otimes(\zeta\otimes1)\mapsto(\xi\otimes\bar\eta)\zeta.
$$
Identifying the $A$-bimodule on the left with $L^2(\tB)\otimes_A (X_H\otimes L^2(M^\opos))$, we see that we actually get an isomorphism of $\tB$-$A$-modules. Note in passing that this isomorphism easily implies that $\ZHC$ is closed under tensor products.
Next, the unitaries $c^H_s$ define a unitary isomorphism $U$ of Hilbert $A$-bimodules
$$
\bigoplus_s(X_s\otimes X_s^\natural)\otimes_A(X_H\otimes L^2(M^\opos))\cong \bigoplus_s(X_H\otimes L^2(M^\opos))\otimes_A(X_s\otimes X_s^\natural),
$$
\begin{equation} \label{eq:unitaryU}
U\big((\xi\otimes\bar\eta)\otimes(\zeta\otimes1)\big)=(c^H(\xi\otimes\zeta)_1\otimes1)\otimes(c^H(\xi\otimes\zeta)_2\otimes\bar\eta).
\end{equation}
Composing $U^*$ with the isomorphism $L^2(\tB)\otimes_A (X_H\otimes L^2(M^\opos))\cong H$, we get an isomorphism
\begin{equation*}\label{eq:isoright}
(X_H\otimes L^2(M^\opos))\otimes_A L^2(\tB)\cong H
\end{equation*}
of Hilbert $A$-bimodules. From \eqref{eq:halfbr-bimod2} we see that this isomorphism maps $(\zeta\otimes1)\otimes(\xi\otimes\bar\eta)$ into $\zeta(\xi\otimes\bar\eta)$ and hence respects the right actions of $\tB$. Thus up to an isomorphism the Hilbert $\tB$-bimodule $H$ can be reconstructed from the Hilbert $M$-module $X_H$ and the unitary $U$ defined by the half-braiding $c^H$.

\smallskip

Now take an object $(X,c)\in\ZC$. We want to construct a Hilbert $\tB$-bimodule $H=H_{(X,c)}\in\ZHC$ out of~$(X,c)$. For this we basically have to repeat the procedure we used above to reconstruct a Hilbert $\tB$-module from the corresponding object of $\ZC$. Thus, we put $H=L^2(\tB)\otimes_A (X\otimes L^2(M^\opos))$ and consider $H$ as a Hilbert $\tB$-$A$-module. Next, using the braiding $c$ instead of $c^H$ define a unitary $U$ by~\eqref{eq:unitaryU}. We can use the right $\tB$-module structure on $(X\otimes L^2(M^\opos))\otimes_A L^2(\tB)$ to define such a structure on~$H$, so for $\xi\in H$ and $b\in\tB$ we let $\xi b=U^*((U\xi)b)$.

\begin{lemma}\label{lem:half-br-and-prod-of-tB}
The left and right actions of $\tB$ on $H$ commute, so $H$ is a Hilbert $\tB$-module.
\end{lemma}

\bp For convenience write $\hat X$ for $X\otimes L^2(M^\opos)$. Let us also denote by  $m$ the product on $\tB$. We claim that
\begin{equation} \label{eq:Izu}
U (m \otimes \iota_{\hat{X}}) = (\iota_{\hat{X}} \otimes m) (U \otimes \iota_\tB) (\iota_\tB \otimes U)
\end{equation}
on a dense subspace of $L^2(\tB)\otimes_AL^2(\tB)\otimes_A\hat X$, which is an analogue of~\cite{MR1782145}*{(4.1)}. More precisely, we claim that the above identity holds on the $M\otimes_{\mathrm{alg}}M^\opos$-sub-bimodule spanned by vectors of the form $(\xi_s\otimes\bar\eta_s)\otimes(\zeta\otimes 1)\otimes(\xi_t\otimes\bar\eta_t)$, with $\xi_s,\eta_s\in X^0_s$, $\xi_t,\eta_t\in X^0_t$ and $\zeta\in X$. On this sub-bimodule both sides of~\eqref{eq:Izu} make perfect sense, since the spaces $X_s\otimes_M X$ and $X\otimes_M X_s$ coincide with the algebraic tensor products $X_s^0\otimes_M X$ and $X\otimes_M X_s^0$ as both $(X_s)_M$ and ${}_MX_s$ have finite Murray--von Neumann dimensions.

Choose coisometries $w_\alpha\colon X_s\otimes_M X_t\to X_{k_\alpha}$ defining a decomposition of $X_s\otimes_M X_t$ into simple modules, and put $\xi_\alpha=w_\alpha(\xi_s\otimes \xi_t)$ and $\eta_\alpha=w_\alpha(\eta_s\otimes\eta_t)$, so that in $\tB$ we have
$$
(\xi_s\otimes\bar\eta_s)(\xi_t\otimes\bar\eta_t)=\sum_\alpha\left(\frac{d_sd_t}{d_{r_\alpha}}\right)^\hlf\xi_\alpha\otimes\bar\eta_\alpha.
$$
Applying both sides of \eqref{eq:Izu} to $(\xi_s\otimes\bar\eta_s)\otimes(\zeta\otimes 1)\otimes(\xi_t\otimes\bar\eta_t)$ we have to check that
\begin{multline} \label{eq:Izu2}
\sum_\alpha\left(\frac{d_sd_t}{d_{r_\alpha}}\right)^\hlf c_{k_\alpha}(\xi_\alpha\otimes\zeta)\otimes\bar\eta_\alpha
\\=c_s\big(\xi_s\otimes c_t(\xi_t\otimes\zeta)_1\big)_1\otimes\big(c_s\big(\xi_s\otimes c_t(\xi_t\otimes\zeta)_1\big)_2\otimes\bar\eta_s\big)
\big(c_t(\xi_t\otimes\zeta)_2\otimes\bar\eta_t\big).
\end{multline}
By the half-braiding condition we have
$$
(c_s\otimes\iota)(\iota\otimes c_t)(\xi_s\otimes\xi_t\otimes\zeta)
=\sum_\alpha(\iota\otimes w_\alpha^*)c_{k_\alpha}(\xi_\alpha\otimes\zeta).
$$
This implies that the right hand side of \eqref{eq:Izu2} equals
$$
\sum_\alpha c_{k_\alpha}(\xi_\alpha\otimes\zeta)_1\otimes \big((w^*_\alpha c_{k_\alpha}(\xi_\alpha\otimes\zeta)_2)_1
\otimes\bar\eta_s\big)\big((w^*_\alpha c_{k_\alpha}(\xi_\alpha\otimes\zeta)_2)_2\otimes\bar\eta_t\big).
$$
We remark that the above expression is still meaningful, since the algebraic tensor product $X^0_s\otimes_M X^0_t$ coincides with $(X_s\otimes_M X_t)^0$ and hence the vector $(\iota\otimes w_\alpha^*)c_{k_\alpha}(\xi_\alpha\otimes\zeta)$ lies in the algebraic tensor product $X\otimes_M X^0_s\otimes_M X^0_t$. By definition of the product in $\tB$ we then see that the above expression equals
$$
\sum_{\alpha,\beta}\left(\frac{d_sd_t}{d_{r_\beta}}\right)^\hlf c_{k_\alpha}(\xi_\alpha\otimes\zeta)_1\otimes w_\beta w^*_\alpha c_{k_\alpha}(\xi_\alpha\otimes\zeta)_2\otimes \bar\eta_\beta,
$$
which is the left hand side \eqref{eq:Izu2} as the coisometries $w_\alpha$ have mutually orthogonal domains. Thus \eqref{eq:Izu} is proved.

Now, for $v=\xi\otimes1\in\hat X$ and $b,c\in\tB$ lying in the linear span of $X^0_s\otimes X^{0\natural}_s$, we get
\begin{equation}\label{eq:unitaryU1}
(b \otimes v) c = U^* (\iota \otimes m) (U \otimes \iota)(b \otimes v \otimes c) = (m \otimes \iota) (\iota \otimes U^*) (b \otimes v \otimes c)
=(b\otimes1)U^*(v\otimes c),
\end{equation}
from which it becomes obvious that the left and right actions of $\tB$ on $L^2(\tB)\otimes_A\hat X$ commute.
\ep

\begin{proof}[Proof of Theorem~\ref{thm:mon-eqv-D-center-B-bimod}]
The construction of the object $(X_H,c^H)$ out of $H\in\ZHC$ defines a unitary functor $G\colon\ZHC\to\ZC$, and the construction of the module $H_{(X,c)}$ out of $(X,c)\in\ZC$ defines a unitary functor $F\colon\ZC\to\ZHC$. The discussion following Lemma~\ref{lem:XH-tensor-central} implies that $FG$ is naturally unitarily isomorphic to the identity functor. In the opposite direction, we also have obvious natural isomorphisms $GF(X,c)\cong X$ in $\indC$. It is slightly less obvious that these isomorphisms respect the half-braidings. In fact, this is automatically the case. Indeed, since $FGF\cong F$, it suffices to show that if two half-braidings~$c$ and~$c'$ for some $X\in\indC$ define the same right $\tB$-module structure on $L^2(\tB)\otimes_A(X\otimes L^2(M^\opos))$, then $c=c'$. But this follows from formula~\eqref{eq:unitaryU1}, which shows that the unitary $U$ is completely determined by the $\tB$-bimodule structure. Thus $GF$ is naturally unitarily isomorphic to the identity functor on $\ZC$.

Therefore $F$ and $G$ are unitary equivalences between the categories $\ZC$ and $\ZHC$. In order to get a unitary monoidal equivalence, it remains to turn either of these two functors into a unitary tensor functor. Let us do this for the functor $G\colon\ZHC\to\ZC$. Given $H,K\in\ZC$, define $G_2\colon X_H\otimes_M X_K\to X_{H\otimes_\tB K}$ by $\xi\otimes\zeta\mapsto\xi\otimes\zeta$. This is easily seen to be a well-defined unitary isomorphism of Hilbert $M$-bimodules, since the embeddings $X_H\hookrightarrow H$ and $X_K\hookrightarrow K$ extend to unitary isomorphisms $(X_H\otimes L^2(M^\opos))\otimes_AL^2(\tB)\cong H$ and $L^2(\tB)\otimes_A (X_K\otimes L^2(M^\opos))\cong K$ of Hilbert $A$-$\tB$- and $\tB$-$A$-modules, respectively. It remains to check that $G_2$ respects the half-braidings. By \eqref{eq:halfbr-bimod}, for $\xi\in X_s^0$, $\zeta\in X_H$ and $\eta\in X_K$, we have
$$
(\iota\otimes c^H_s)(c^K_s\otimes\iota)(\xi\otimes\zeta\otimes\eta) = d_s^{-2}\sum_{\alpha,\beta,\alpha',\beta'}(\xi\otimes\bar\rho_{s\alpha})\zeta(\bar\lambda_{s\beta}\otimes\rho_{s\alpha}) \otimes
(\lambda_{s\beta}\otimes\bar\rho_{s\alpha'})\eta(\bar\lambda_{s\beta'}\otimes\rho_{s\alpha'})\otimes\lambda_{s\beta'}.
$$
Using that $d_s^{-1}\sum_\beta(\bar\lambda_{s\beta}\otimes\rho_{s\alpha}) (\lambda_{s\beta}\otimes\bar\rho_{s\alpha'}) =1\otimes\langle\rho_{s\alpha'},\rho_{s\alpha}\rangle_M$ by Lemma~\ref{lem:prodrelations}, then the $M^\opos$-centrality of~$\zeta$, and finally that $\sum_\alpha\rho_{s\alpha}\langle\rho_{s\alpha},\rho_{s\alpha'}\rangle_M=\rho_{s\alpha'}$, we see that the above expression equals
$$
d_s^{-1}\sum_{\alpha',\beta'}(\xi\otimes\bar\rho_{s\alpha'})\zeta\otimes
\eta(\bar\lambda_{s\beta'}\otimes\rho_{s\alpha'})\otimes\lambda_{s\beta'}.
$$
This is $c^{H\otimes_\tB K}_s(\xi\otimes\zeta\otimes\eta)$, as we need.
\end{proof}

We finish this section by noting that it is very plausible that the results we have obtained remain true without assumption~\eqref{eq:eqtrace} on $\CC$. In this case, however, the factor $B$ is no longer of type II$_1$ (see \cite{MR3406647}*{Remark~2.7}) and more care is needed in dealing with Hilbert bimodules.

\subsection{Pointed bimodules and representations of the character algebra}
\label{sec:rep-char-alg-on-bimod}

We continue to consider a category $\CC\subset\Hilb_M$ as in Section~\ref{subsec:LR}, with the associated Longo--Rehren inclusion $A\subset\tB$. For categories of the form $\CC_{N\subset M}$ as in Example~\ref{ex:subfactorcat}, the results that follow have been obtained by Popa and Vaes~\cite{MR3406647} by different methods.

Given a Hilbert $\tB$-module $H\in\ZHC$, we have the corresponding object $(X_H,c^H)\in\ZC$, and hence a representation $\pi_H$ of $C^*(\CC)$ on $\Mor_{\indC}(\un,X_H)=\Hom_{M\text{-}M}(L^2(M),X_H)$. Recall that $X_H\subset H$ is the subspace of $M^\opos$-central vectors. Then $\Hom_{M\text{-}M}(L^2(M),X_H)$ can be identified with the subspace of $X_H$ of $M$-central vectors, that is, with the subspace $H_0\subset H$ of $A$-central vectors. Recalling formula~\eqref{eq:halfbr-bimod} for the half-braiding~$c^H$ and how the representation associated with a half-braiding is defined, we get the following formula for~$\pi_H$:
$$
\pi_H([X_s])\xi=d_s^{-1}\sum_{\alpha,\alpha'}(\rho_{s\alpha}\otimes\bar\rho_{s\alpha'})\xi(\bar\rho_{s\alpha}\otimes\rho_{s\alpha'})\ \ \text{for}\ \ \xi\in H_0,
$$
where $\{\rho_{s\alpha}\}_\alpha$ is a basis of the right Hilbert $M$-module $X_s$.

Every $A$-central unit vector $\xi\in H_0$ defines an $A$-bimodular normal ucp map $\Phi_\xi\colon\tB\to \tB$, see e.g.~\cite{MR2215135}*{Section~1.1}. Namely, since $A'\cap\tB=\C1$, we have $(x\xi,\xi)=(\xi x,\xi)=\tau(x)$ for all $x\in\tB$. Therefore the map $\tB\ni x\mapsto \xi x$ extends to an isometry $T_\xi\colon L^2(\tB)\to H$. Viewing then elements of $\tB$ as operators acting on the left on $L^2(\tB)$ and $H$, define $\Phi_\xi(x) = T_\xi^*x T_\xi$. In other words, $\Phi_\xi(x)$ is characterized by the identity
\begin{equation} \label{eq:Phi}
\tau(\Phi_\xi(x)b)=(x\xi b,\xi)\ \ \text{for}\ \ b\in\tB.
\end{equation}
Conversely, given an $A$-bimodular normal ucp map $\Phi\colon \tB\to \tB$, one can construct a Hilbert $\tB$-bimodule $H$ together with a distinguished $A$-central unit vector $\xi\in H$: we obtain $H$ from $\tB$ using the pre-inner product $(x,y)=\tau(\Phi(x)\Phi(y)^*)$ and then let $\xi$ be the image of $1\in \tB$ in $H$. This gives a one-to-one correspondence between $A$-bimodular normal ucp maps $\Phi\colon \tB\to \tB$ and isomorphism classes of pointed Hilbert bimodules over $A\subset\tB$, by which one means pairs $(H,\xi)$ consisting of a Hilbert $\tB$-bimodule $H$ and an $A$-central unit vector $\xi\in H$ such that $\tB\xi\tB$ is dense in $H$. Note that if $(H,\xi)$ is a pointed Hilbert bimodule, then $H\in\ZHC$, since $H$ is generated as a $\tB$-bimodule by a copy of $L^2(A)$.

The ucp map $\Phi_\xi$ extends to a contraction on $L^2(\tB)$. Since it is $A$-bimodular, it follows that on $X_s\otimes X^\natural_s$ this extension acts as a scalar $\alpha_{\xi,s}$. Taking $x=\rho_{s\alpha}\otimes\bar\rho_{s\alpha'}$ and $b=\bar\rho_{s\alpha}\otimes\rho_{s\alpha'}=x^*$ in~\eqref{eq:Phi} and then summing up over $\alpha,\alpha'$, we get $\dim({}_A(X_s\otimes X^\natural_s))=d_s^2\alpha_{\xi,s}$ on the left and $d_s(\pi_H([X_s])\xi,\xi)$ on the right, that~is,
\begin{equation} \label{eq:bimodchar}
d_s\alpha_{\xi,s}=(\pi_H([X_s])\xi,\xi).
\end{equation}
Therefore $(\alpha_{\xi,s})_s$ is exactly the positive definite function on $I=\Irr(\CC)$ associated with $(X_H,c^H)\in\ZC$ and $\xi\in \Mor_{\indC}(\un,X_H)$. By Theorems~\ref{thm:positivedef} and~\ref{thm:mon-eqv-D-center-B-bimod} every positive definite function $\phi$ with $\phi(e)=1$ arises this way: the corresponding pointed Hilbert bimodule is $(F(Z_\phi),F(\xi_\phi))$. Note that the fact that~$F(Z_\phi)$ is generated by $F(\xi_\phi)$ as a $\tB$-module follows from the universality property of $(Z_\phi,\xi_\phi)$ established in Proposition~\ref{prop:Zphi-universal}.

To summarize the above discussion, we have proved the following result.

\begin{theorem}
We have a one-to-one correspondence between $A$-bimodular normal ucp maps $\Phi\colon \tB\to \tB$ and positive definite functions $\phi$ on $I=\Irr(\CC)$ such that $\phi(e)=1$. Namely, the map $\Phi$ corresponding to $\phi$ is defined by $\Phi(x)=\phi(s)x$ for $x\in X^0_s\otimes X^{0\natural}_s\subset\tB$ and $s\in I$.
\end{theorem}

\begin{example}
The function $\phi\equiv1$ defines the identity map. The corresponding pointed bimodule is $(L^2(\tB),1)$. The function $\phi=\delta_e$ defines the conditional expectation $\tB\to A$. The corresponding pointed bimodule is $(L^2(\tB)\otimes_A L^2(\tB),1\otimes1)$. In particular, the bimodule $L^2(\tB)\otimes_A L^2(\tB)$ corresponds to $\Zreg(\CC)$ under the equivalence between $\ZHC$ and $\ZC$.
\end{example}

Consider now the following rigidity condition. One says~\citelist{\cite{popa-corr-preprint}*{Chapter~4}\cite{MR914742}} that $\tB$ has property~(T) relative to $A$, or that $A$ is corigid in $\tB$, if there are a finite set $F \subset \tB$ and $\eps > 0$ such that if $(H, \xi)$ is a pointed Hilbert bimodule over $A \subset \tB$ with $\norm{x \xi - \xi x} < \eps$ for $x \in F$, then there is a $\tB$-central unit vector in~$H$. (This is not to be confused with rigidity of the inclusion $A \subset \tB$~\cite{MR2215135}.) For some other equivalent conditions see~\cite{popa-corr-preprint}*{4.1.5} and~\cite{MR1729488}*{9.1}.

\begin{proposition}
\label{prop:cat-T-vs-subfactor-T}
The C$^*$-tensor category $\CC$ has property (T) if and only if $\tB$ has property (T) relative to~$A$.
\end{proposition}

\bp Let $H\in\ZHC$ be a Hilbert $\tB$-module and $\xi\in H$ be an $A$-central unit vector. Then by~\eqref{eq:Phi}, for any $x\in\tB$ we have
$$
\|x\xi-\xi x\|^2=2\tau(x^*x)-\tau(\Phi_\xi(x)x^*)-\tau(\Phi_\xi(x^*)x).
$$
Hence for any $x\in X^0_s\otimes X^{0\natural}_s$ we have
$$
\|x\xi-\xi x\|^2=2(1-\Re \alpha_{\xi,s})\|x\|^2_2.
$$
As $d_s\alpha_{\xi,s}=(\pi_H([X_s])\xi,\xi)$ by \eqref{eq:bimodchar}, we see that an $A$-central vector is $\tB$-central if and only if it is invariant with respect to the representation $\pi_H\colon C^*(\CC)\to B(H_0)$. It also follows that $\tB$ has property~(T) relative to~$A$ if and only if there are a finite set $F \subset I$ and $\eps > 0$ such that if $H\in\ZHC$ and $\xi$ is an $A$-central unit vector satisfying $|(\pi_H([X_s])\xi,\xi)-d_s| < \eps$ for all $s \in F$, then there is an invariant unit vector in $H_0$. But in view of the equivalence between the categories $\ZC$ and $\ZHC$, this is exactly condition (ii) in Proposition~\ref{prop:propertyT}.
\ep

\bigskip

\section{Weakly Morita equivalent categories}
\label{sec:induct-thro-weak-Morita}

In this section we study Drinfeld centers and representation theory for weakly monoidally Morita equivalent C$^*$-tensor categories. As in Sections~\ref{sec:half-braidings} and~\ref{sec:character-algebra}, let $\CC$ be an essentially small strict rigid C$^*$-tensor category satisfying our standard assumptions.

\subsection{\texorpdfstring{$Q$}{Q}-systems}
\label{sec:prelim-q-system}

A \emph{$Q$-system}~\cite{MR1257245} in $\CC$ is a triple $(Q, v, w)$, where $Q$ is an object in $\CC$, $v$ is an isometry in $\CC(\un, Q)$, and $w$ is a scalar multiple of an isometry in $\CC(Q, Q \otimes Q)$, satisfying
\begin{itemize}
\item[]\emph{unit constraint}: $(v^* \otimes \iota) w = \iota = (\iota \otimes v^*) w$,
\item[]\emph{associativity}: $(w \otimes \iota) w = (\iota \otimes w) w$,
\item[]\emph{Frobenius condition}: $(w^* \otimes \iota) (\iota \otimes w) = w w^* = (\iota \otimes w^*) (w \otimes \iota)$.
\end{itemize}
(Of course, other normalizations than $v^*v=\iota$ are possible and used in the literature, and the last equality is redundant; in fact, it can be shown that the entire Frobenius condition is redundant~\cite{MR1444286}.)
In other words, a $Q$-system is a Frobenius object in $\CC$ such that the algebra and coalgebra structures on it are obtained from each other by taking adjoints, and the coproduct is a scalar multiple of an isometry. Depending on the context, we will sometimes write $m_Q$ for the product $w^*$.

The object~$Q$ is self-dual, with $(w v, w v)$ being a solution of the conjugate equations.
In the following we also assume that
$$
\text{the}\ Q\text{-system} \ Q \ \text{is \emph{standard} and \emph{simple}},
$$
see~\cite{MR3308880}. The first assumption means that $w^*w=d(Q)\iota$, so that $(w v, w v)$ is a standard solution of the conjugate equations. The second assumption means that $Q$ is simple as a $Q$-bimodule. In the last section instead of simplicity we will require $Q$ to be \emph{irreducible}, meaning that $Q$ is simple as a left and/or right $Q$-module.

A left $Q$-module is an object $M \in \CC$ together with a morphism $m_M = m^l_M \in \CC(Q \otimes M, M)$ satisfying the associativity condition $m_M (m_Q \otimes \iota) = m_M (\iota \otimes m_M)$, the unit constraint $m_M (v \otimes \iota) = \iota$, and the $*$-compatibility $m_M^* = (\iota \otimes m_M) (w v \otimes \iota)$. By the Frobenius condition, $Q$ itself is a left $Q$-module in this sense. Furthermore, for general $M$ the map $m_M$ has the following properties similar to $m_Q=w^*$.

\begin{lemma}
\label{lem:comput-m-m-star}
  For any left $Q$-module $M$, we have
$$
m_M m_M^{*} = d(Q) \iota\ \ \text{and}\ \  m^*_Mm_M=(\iota_Q\otimes m_M)(w\otimes\iota_M).
$$
\end{lemma}

\bp
We have
$$
m_M m_M^{*} = m_M (\iota \otimes m_M) (w v \otimes \iota)=m_M(w^* \otimes\iota)(w v \otimes \iota) = d(Q) \iota,
$$
since $w^*w=d(Q)\iota$ by our assumptions.

For the second identity, we get an equivalent one if we tensor it on the left by $\iota_Q$ and then multiply by $R_Q\otimes\iota_M=(wv)^*\otimes\iota_M$. Then the left hand side gives
$$
((wv)^*\otimes\iota_M)(\iota_Q\otimes\iota_Q \otimes m_M) (\iota_Q\otimes w v \otimes \iota_M)(\iota_Q\otimes m_M)=m_M(\iota_Q\otimes m_M),
$$
while the right hand side gives
$$
((wv)^*\otimes\iota_M)(\iota_Q\otimes\iota_Q\otimes m_M)(\iota_Q\otimes w\otimes\iota_M)=m_M(m_Q\otimes\iota_M),
$$
since $((wv)^*\otimes\iota)(\iota\otimes w)=w^*=m_Q$. Thus the lemma is proved.
\ep

A $Q$-morphism between two left $Q$-modules $M$ and $N$ is a morphism $T \in \CC(M, N)$ satisfying $T m_M = m_N (\iota \otimes T)$.

\begin{remark}
Let us call a left $Q$-module without the $*$-compatibility an algebraic left $Q$-module. Any algebraic left $Q$-module $M$ is isomorphic to a left $Q$-module as follows. Put $A = m_M m^*_M$. This is a positive morphism majorizing $\iota = m_M (vv^*\otimes\iota) m_M^*$, so in particular $A$ is invertible. Then it can be checked that $A^{-\hlf} m_M (\iota \otimes A^{\hlf})$ defines the structure of a left $Q$-module on $M$, and $A^{\hlf}$ gives an isomorphism of this module with $(M, m_M)$.
\end{remark}

We denote the category of left $Q$-modules by $\Qmod_\CC$ or, if it is clear from the context, just by $\Qmod$.
By the $*$-compatibility condition, $\Qmod$ is closed under the involution $T \mapsto T^*$, so it is a C$^*$-category. Since $\CC$ is closed under subobjects and has finite dimensional morphism spaces, it follows that $\Qmod$ is semisimple. We also note that $\Qmod$ is a right $\CC$-module category.

For any $X\in\CC$ and any left $Q$-module $M$ we have
an isomorphism
$$
\CC(X,M)\cong\Mor_{\Qmod}(Q\otimes X,M), \ \ T\mapsto m_M(\iota\otimes T),
$$
with the inverse $T'\mapsto T'(v\otimes\iota)$. In particular, $\CC(\un,Q)\cong\End_{\Qmod}(Q)$. Therefore irreducibility of $Q$ is equivalent to $\dim\CC(\un,Q)=1$. We also remark that the above isomorphism implies that either $\Irr(\CC)$ and $\Irr(\Qmod)$ are finite or they are both infinite of the same cardinality.

The notions of right $Q$-modules and $Q$-bimodules are defined similarly, and the corresponding categories are denoted by $\modQ$ and $\QmodQ$.

\begin{example}
  Let $X$ be a simple object in $\CC$. Then $Q_X = \bar{X} \otimes {X}$ has the canonical structure of an irreducible standard $Q$-system, with $v = d(X)^{-\hlf} R_X$ and $w = d(X)^\hlf \iota_{\bar{X}} \otimes \bar{R}_X \otimes \iota_X$. If $(M, m_M \colon Q_X \otimes M \to M)$ is a left $Q_X$-module, then $p = d(X)^{-3/2} (\iota_X \otimes m_M)(\bar{R}_X \otimes \iota_X\otimes\iota_M) \in \CC(X \otimes M)$ is a projection. Let $N$ be the subobject of $X \otimes M$ specified by $p$. Then $\bar X\otimes N$ has a $Q_X$-module structure given by $d(X)^{1/2}\iota \otimes \bar{R}_X^* \otimes \iota$, and it is not difficult to show that the map $d(X)^{1/2}(\iota_{\bar X}\otimes p)(R_X\otimes\iota_M)=d(X)^{-1}m^*_M$ is an isometric $Q_X$-module isomorphism $M\cong \bar{X} \otimes N$. This way we obtain an equivalence $Q_X\mhyph\operatorname{mod} \cong \CC$.
\end{example}

\begin{example}
Let $\CC=\CC_{N\subset M}$ be the category of Hilbert $M$-bimodules defined by an extremal finite index subfactor as in Example~\ref{ex:subfactorcat}. Then $Q = L^2(M_1)$ has the structure of a standard simple $Q$-system in~$\CC$, where $v\colon L^2(M) \to L^2(M_1)$ is the inclusion map and $w\colon L^2(M_1) \to L^2(M_1)\otimes_M L^2(M_1)=L^2(M_2)$ is the inclusion map multiplied by $[M : N]^{\hlf}$. This $Q$-system is irreducible if and only if $N\subset M$ is irreducible.

Up to monoidal equivalence, any standard simple $Q$-system $(Q, v, w)$ such that~$Q$ is a generating object is obtained this way.
Indeed, we get a $\lambda$-lattice by taking the algebras $\CC(Q^{\otimes k})$, $\End_{\Qmod}(Q^{\otimes k})$, $\End_{\modQ}(Q^{\otimes k})$, and $\End_{\QmodQ}(Q^{\otimes k})$ with the natural inclusions maps and the normalized categorical traces. In order to verify the axioms of $\lambda$-lattices it suffices to show that for any $Q$-bimodule $M$ we get a commuting square
$$
\begin{matrix}
 \End_{\modQ}(M) &\subset & \CC(M)\\
 \cup& &\cup\\
 \End_{\QmodQ}(M) &\subset & \End_{\Qmod}(M).
\end{matrix}
$$
We claim that the trace-preserving conditional expectation $E^l_M\colon \CC(M)\to\End_{\Qmod}(M)$ is given by $$E^l_M(T)=d(Q)^{-1}m_M^l(\iota_Q\otimes T)m_M^{l*}.$$ It follows easily from Lemma~\ref{lem:comput-m-m-star} that this formula defines a conditional expectation onto $\End_{\Qmod}(M)$, so we only need to show that the trace is preserved. This is equivalent to showing that $(\tr_Q\otimes\iota)(m^{l*}_Mm^l_M)=\iota_M.$
But this follows from the second identity in Lemma~\ref{lem:comput-m-m-star}, since
$$
(\tr_Q\otimes\iota)(w)=d(Q)^{-1}((wv)^*\otimes\iota)(\iota\otimes w)wv=d(Q)^{-1}w^*wv=v.
$$
The explicit formula shows that the conditional expectation $E^l_M$ maps $ \End_{\modQ}(M)$ into $ \End_{\QmodQ}(M)$, so we indeed get a commuting square.
\end{example}

The semisimplicity implies that the finite colimits exist in $\CC$, as can be seen by taking isotypic decompositions and reducing the statement to the category of finite dimensional Hilbert spaces.\footnote{However, infinite colimits do not make sense in general even in $\indC$, because we require uniform boundedness of morphisms.} In particular, we have a natural relative tensor product operation
$$
\modQ \times \Qmod \to \CC, \quad (M, N) \mapsto M \otimes_Q N = \text{coequalizer of } M \otimes Q \otimes N \rightrightarrows M \otimes N.
$$
We denote by $P_{M,N}$ the structure morphism $M\otimes N\to M\otimes_QN$ .

\begin{remark}
Although we will never need this, the tensor products over $Q$ can be explicitly constructed as follows. Consider the morphism
$$
p=d(Q)^{-1}(\iota\otimes v^*w^*\otimes\iota)(m^{r*}_M\otimes m^{l*}_N)\in\CC(M\otimes N).
$$
Then $p$ is a projection, so there exists an object $X$ and an isometry $u\colon X\to M\otimes N$ such that $uu^*=p$, and as $P_{M,N}\colon M\otimes N\to M\otimes_Q N$ we can take $u^*\colon M\otimes N\to X$.
\end{remark}

By taking the polar decomposition of the adjoint of the structure morphism $P_{M,N}\colon M\otimes N\to M\otimes_Q N$ we may first assume that~$P_{M,N}$ is a coisometry, and then in view of Lemma~\ref{lem:comput-m-m-star} we rescale it so that $$P_{M,N}P^*_{M,N}=d(Q)\iota.$$Therefore for a left $Q$-module $M$ as a model of $Q\otimes_Q M$ we take $M$ and $P_{Q,M}=m_M$, and similarly for the right $Q$-modules. The common normalization of the structure maps ensures that the natural map $\Mor_{\modQ}(M, M') \times \Mor_{\Qmod}(N, N') \to \CC(M \otimes_Q N, M' \otimes_Q N')$ is compatible with the involution. The category $\QmodQ$ then becomes a C$^*$-tensor category. This category is not strict on the nose: the associativity morphisms $\Phi^{Q}_{X,Y,Z}\colon (X\otimes_QY)\otimes_Q Z\to X\otimes_Q(Y\otimes_Q Z)$ are characterized by the identities $$\Phi^{Q}_{X,Y,Z}P_{X\otimes_Q Y,Z}(P_{X,Y}\otimes\iota_Z)=P_{X, Y\otimes_Q Z}(\iota_X\otimes P_{Y,Z}).$$
Note that our normalization of the structure maps $P$ implies that $\Phi^Q$ is unitary, as needed in C$^*$-tensor categories. As is common, from now on we will ignore the associativity morphisms and work with the category $\QmodQ$ as if it was strict. Since by assumption $Q$ is simple as a $Q$-bimodule, the tensor unit in $\QmodQ$ is simple.

Following M\"{u}ger~\cite{MR1966524}, we say that $\CC$ and $\QmodQ$ are \emph{weakly monoidally Morita equivalent}.

The \emph{dual $Q$-system} is given by the object $\hat{Q} = Q \otimes Q$ in $\QmodQ$, together with the morphisms $$
\hat{v} = d(Q)^{-\hlf} w\ \ \text{and}\ \ \hat{w} = d(Q)^{\hlf} \iota \otimes v \otimes \iota
$$
under the identification  of $\hat{Q} \otimes_Q \hat{Q}$ with $Q^{\otimes 3}$. Using that the unit of $\CC$ is simple it is easy to check that this $Q$-system is simple (even without the simplicity assumption on $Q$; see also Proposition~\ref{prop:hat-Q-bimod-eqv-C} below). As we will see shortly, the dimension of $\hat Q=Q\otimes Q$ in $\QmodQ$ equals the dimension of $Q$ in $\CC$. Since $\hat w^*\hat w=d(Q)\iota$, it follows that the $Q$-system $\hat Q$ is standard.

If $Q$ is irreducible, then $\hat Q$ is also irreducible, as follows from the isomorphism
$$
\End_{\Qmod}(Q)\cong\Mor_{\QmodQ}(Q,Q\otimes Q), \ \ T\mapsto (T\otimes \iota)w.
$$

\subsection{Duality for \texorpdfstring{$Q$}{Q}-modules}
\label{sec:mod-cat-Frob-recip}

We continue to assume that $Q$ is a standard simple $Q$-system. Given any left $Q$-module $M$, its conjugate $\bar{M}$ has the natural structure of a right $Q$-module defined by
$m_{\bar M}=m_M^{\vee*}\in \CC(\bar M\otimes Q,\bar M)$. Here, as usual, we fix a standard solution $(R_M,\bar R_M)$ of the conjugate equations for $M$ and then use it, together with the solution $(wv,wv)$ for $Q$, to define solutions of the conjugate equations for tensor products. A direct computation shows that $w^\vee=w^*$ and $v^\vee=v^*$, which then immediately implies by applying $\vee$ to the identities involving $m_M$ that $m_{\bar M}$ indeed defines the structure of a right $Q$-module on~$\bar M$.
Note that by definition we have
$$
(m^\vee_M\otimes\iota_{\bar M}) R_M=(\iota_{\bar M}\otimes\iota_Q\otimes m_M)(\iota_{\bar M}\otimes wv\otimes\iota_M)R_M=(\iota_{\bar M}\otimes m_M^*)R_M.
$$
Therefore $m_{\bar M}$ is characterized by the identity
\begin{equation} \label{eq:dualmod}
R_M^*(m_{\bar M}\otimes\iota)=R_M^*(\iota\otimes m_M)\ \ \text{on}\ \ \bar M\otimes Q\otimes M.
\end{equation}

If $M$ is a right $Q$-module, we can define the structure of a left $Q$-module on $\bar{M}$ in the same way by $m_{\bar{M}}^l = (m_M^r)^{\vee *}$. If $M$ is a $Q$-bimodule, then $\bar M$ becomes a $Q$-bimodule.

\smallskip

Now, assuming again that $M$ is a left $Q$-module, $M \otimes \bar{M}$ has the natural structure of a $Q$-bimodule. Since $m_M\colon Q\otimes M\to M$ is a left $Q$-module map, the morphism
$$
\bar{S}_M = (m_M \otimes \iota) (\iota \otimes \bar{R}_M)\in \CC(Q, M \otimes \bar{M})
$$
is a left $Q$-module map. Note that by \eqref{eq:dualmod} we have $m_M=(\iota\otimes R^*_M)(\iota\otimes m_{\bar M}\otimes\iota)(\bar R_M\otimes\iota\otimes\iota)$, which implies that $\bar S_M=(\iota \otimes m_{\bar{M}}) (\bar{R}_M \otimes \iota)$. Hence $\bar S_M$ is also a morphism of right $Q$-modules.

Next, by \eqref{eq:dualmod} the morphism $R_M^*$ descends to $\bar{M} \otimes_Q M$. Let us denote the induced morphism $\bar{M} \otimes_Q M \to \un$ by $[R_M^*]$.
We then have the following duality between left and right $Q$-modules.

\begin{lemma}
\label{lem:Q-mod-duality-1}
We have $(\iota \otimes [R_M^*]) (\bar{S}_M \otimes_Q \iota) = \iota_M$ and  $([R_M^*] \otimes \iota) (\iota \otimes_Q \bar{S}_M) = \iota_{\bar{M}}$.
\end{lemma}

\bp
Since the situation is symmetric, we just give a proof of the first identity. The left hand side of this identity is the morphism $Q\otimes_Q M\to M$ defined by the morphism
$$
(\iota_M \otimes R_M^*) (\bar{S}_M \otimes\iota_M)\colon Q\otimes M\to M.
$$
As $\bar{S}_M = (m_M \otimes \iota) (\iota \otimes \bar{R}_M)$, the latter morphism is simply $m_M$, so it descends to the identity morphism $M=Q\otimes_Q M\to M$.
\ep

This allows us to prove two versions of Frobenius reciprocity:
$$
\CC(L \otimes_Q M, N) \cong \Mor_{\modQ}(L, N \otimes \bar{M})\ \ \text{and}\ \ \CC(\bar{M} \otimes_Q L, N) \cong \Mor_{\Qmod}(L, M \otimes N),
$$
where $L$ is a right $Q$-module in the first case and a left $Q$-module in the second. Namely, the first isomorphism is given by the map
$T \mapsto (T \otimes \iota) (\iota \otimes_Q \bar{S}_M)$, and its inverse is given by
$T' \mapsto (\iota \otimes R_M^*) (T' \otimes_Q \iota)$.
The second isomorphism is defined similarly.

\smallskip

Now, suppose that $M$ is a $Q$-bimodule in $\CC$, so that $\bar{M}$ is also a $Q$-bimodule.

\begin{lemma} \label{lem:Sdescent}
The morphism $\bar S_M^*\colon M\otimes\bar M\to Q$ descends to a morphism $[\bar S^*_M]\colon M\otimes_Q\bar M\to Q$ of $Q$-bimodules.
\end{lemma}

\bp We have to check that $\bar S^*_M(m^r_M\otimes\iota_{\bar M})=\bar S^*_M(\iota_M\otimes m^l_{\bar M})$. By the characterization
$\bar R^*_M(m^r_M\otimes\iota)=\bar R^*_M(\iota\otimes m^l_{\bar M})$, we compute:
\begin{multline*}
\bar S^*_M(\iota_M\otimes m^l_{\bar M})= (\iota_Q \otimes \bar{R}_M^*)(m^{l*}_M \otimes \iota_{\bar M})(\iota_M\otimes m^l_{\bar M})\\
=(\iota_Q \otimes \bar{R}_M^*)(\iota_Q\otimes\iota_M\otimes m^l_{\bar M})(m^{l*}_M\otimes\iota_Q \otimes \iota_{\bar M})
=(\iota_Q \otimes \bar{R}_M^*)(\iota_Q\otimes m^r_M\otimes\iota_{\bar M})(m^{l*}_M\otimes\iota_Q \otimes \iota_{\bar M}).
\end{multline*}
Therefore, as $\bar S^*_M(m^r_M\otimes\iota_{\bar M})=(\iota_Q \otimes \bar{R}_M^*)(m^{l*}_M\otimes\iota_{\bar M})(m^r_M\otimes\iota_{\bar M})$, we have to show that
\begin{equation*}\label{eq:mrml}
m^{l*}_Mm^r_M=(\iota\otimes m^r_M)(m^{l*}_M\otimes\iota).
\end{equation*}
But this identity means simply that $m_M^{l*}$ is a morphism of right $Q$-modules, so the lemma is proved.
\ep

We can now define duality morphisms for $Q$-bimodules. First, similarly to $\bar S_M$, define
$$
S_M=(m^l_{\bar M}\otimes\iota)(\iota\otimes R_{M})=(\iota\otimes m^r_M)(R_M\otimes\iota) \colon Q\to \bar M\otimes M.
$$
Then put
$$
R^Q_M=d(Q)^{-1}P_{\bar M,M}S_M\ \ \text{and}\ \ \bar R^Q_M=d(Q)^{-1}P_{M,\bar M}\bar S_M.
$$
By Lemma~\ref{lem:Sdescent} we can write $\bar S_M^*=[\bar S^*_M]P_{M,\bar M}$, and $(P_{M,\bar M}\bar S_M)^*=[\bar S^*_M]P_{M,\bar M}P^*_{M,\bar M}=d(Q)[\bar S^*_M]$. Similar identities hold for $S_M$, so we get
\begin{equation}\label{eq:RQstar}
\bar R^{Q*}_MP_{M,\bar M}=\bar S^*_M\ \ \text{and}\ \ R^{Q*}_MP_{\bar M,M}=S^*_M.
\end{equation}

\begin{lemma}
\label{lem:S-M-conjug-sol}
The pair $(R^Q_M,\bar R^Q_M)$ solves the conjugate equations for $M$ in $\QmodQ$.
\end{lemma}

\bp
We will only show that $(\bar R^{Q*}_M\otimes_Q\iota)(\iota\otimes_Q R^Q_M)=\iota$. We have
\begin{align*}
(\bar R^{Q*}_M\otimes_Q\iota)(\iota\otimes_Q R^Q_M)P_{M,Q} &=d(Q)^{-1}(\bar R^{Q*}_M\otimes_Q\iota)
P_{M,\bar M\otimes_Q M}(\iota\otimes P_{\bar M,M}S_M)\\
&=d(Q)^{-1}P_{Q,M}(\bar R^{Q*}_MP_{M,\bar M}\otimes\iota)(\iota\otimes S_M)\\
&=d(Q)^{-1}P_{Q,M}(\bar S_M^*\otimes\iota)(\iota\otimes S_M),
\end{align*}
where we used the associativity $P_{M, \bar{M}\otimes_Q M} (\iota_M \otimes P_{\bar{M},M}) = P_{M \otimes_Q \bar{M}, M} (P_{M,\bar{M}} \otimes \iota)$ of the tensor product and then that $\bar R^{Q*}_MP_{M,\bar M}=\bar S_M^*$
by~\eqref{eq:RQstar}. We have
\begin{align*}
(\bar{S}_M^* \otimes \iota) (\iota \otimes S_M)&=
(\iota_Q\otimes \bar R_{M}^*\otimes\iota_M)(m^{l*}_M\otimes\iota_{\bar M}\otimes\iota_M)
(\iota_M\otimes\iota_{\bar M}\otimes m^r_M)(\iota_M\otimes R_M\otimes\iota_Q)\\
&=(\iota_Q\otimes m^r_M)(m^{l*}_M\otimes\iota_Q)=m^{l*}_Mm^r_M.
\end{align*}
Therefore, using that $P_{Q,M}=m^l_M$, $P_{M,Q}=m^r_M$ and $m^l_Mm^{l*}_M=d(Q)\iota$ by Lemma~\ref{lem:comput-m-m-star}, we get
$$
(\bar R^{Q*}_M\otimes_Q\iota)(\iota\otimes_Q R^Q_M)P_{M,Q}
=d(Q)^{-1}P_{Q,M}m^{l*}_Mm^r_M=m^r_M=P_{M,Q}.
$$
Hence $(\bar R^{Q*}_M\otimes_Q\iota)(\iota\otimes_Q R^Q_M)=\iota_M$.
\ep

We thus see that $\QmodQ$ is a rigid C$^*$-tensor category. Let us decorate the constructions related to $\QmodQ$, such as the categorical trace, the dimension function, and so on, with superindex $Q$.

\begin{proposition}
\label{prop:S-M-resc-std-sol}
The solution $(R^Q_M,\bar R^Q_M)$  of  the conjugate equations for $M$ in $\QmodQ$ is standard and $d^Q(M)=d(Q)^{-1}d(M)$. Furthermore,
for any $T\in\End_{\QmodQ}(M)$ we have $\Tr^Q_M(T)=d(Q)^{-1}\Tr_M(T)$.
\end{proposition}

\bp
For any $T\in\End_{\QmodQ}(M)$ we have
$$
\bar R^{Q*}_M(T\otimes_Q\iota)\bar R^{Q}_M=d(Q)^{-1}\bar R^{Q*}_M(T\otimes_Q\iota)P_{M,\bar M}\bar S_M
=d(Q)^{-1}\bar R^{Q*}_MP_{M,\bar M}(T\otimes\iota)\bar S_M=d(Q)^{-1}\bar S^*_M(T\otimes\iota)\bar S_M,
$$
where in the last step we used \eqref{eq:RQstar}. This expression is  a scalar endomorphism of $Q$, so in order to detect this scalar we can compose it with $v^*$ on the right and $v$ on the left. But $\bar S_M v=\bar R_M$, hence
$$
\bar R^{Q*}_M(T\otimes_Q\iota)\bar R^{Q}_M=d(Q)^{-1}\bar R^*_M(T\otimes\iota) R_M=d(Q)^{-1}\Tr_M(T).
$$
By a similar argument $R^{Q*}_M(\iota\otimes_Q T) R^{Q}_M$ gives the same result. Hence the solution $(R^Q_M,\bar R^Q_M)$ is standard and $\Tr^Q_M(T)=d(Q)^{-1}\Tr_M(T)$.
\ep

For $Q$-bimodules, the contravariant functor $T \to T^\vee$ can be a priori defined in two ways, using standard solutions either in $\CC$ or in $\QmodQ$. However, these two definitions give the same result. In fact, the following slightly more general statement is true.

\begin{lemma}
\label{lem:vee-functor-same}
 Let $T \colon M \to N$ be a morphism of left $Q$-modules. If $T^\vee \in \CC(\bar{N}, \bar{M})$ satisfies $(T \otimes \iota) \bar{R}_M = (\iota \otimes T^\vee) \bar{R}_N$, then $T^\vee$ is a morphism of right $Q$-modules and $(T \otimes_Q \iota) P_{M,\bar{M}} \bar{S}_M = (\iota \otimes_Q T^\vee) P_{N,\bar{N}} \bar{S}_N$.
\end{lemma}

\bp
It is immediate that $T^\vee$ is a morphism of right $Q$-modules.
In order to show the second claim it is enough to establish the identity $(T \otimes \iota) \bar{S}_M = (\iota \otimes T^\vee) \bar{S}_N$. We have
\begin{multline*}
(\iota_N\otimes T^\vee)\bar S_N=(\iota_N\otimes T^\vee)(m_N\otimes\iota_{\bar N})(\iota_Q\otimes\bar R_N)
=(m_N\otimes\iota_{\bar M})(\iota_Q\otimes T\otimes\iota_{\bar M})(\iota_Q\otimes\bar R_M)\\
=(T\otimes\iota_{\bar M})(m_M\otimes\iota_{\bar M})(\iota_Q\otimes\bar R_M)
=(T\otimes\iota_{\bar M})\bar S_M,
\end{multline*}
and the lemma is proved.
\ep

\subsection{Schauenburg's induction}
\label{sec:schauen-induc}

If $\Irr(\CC)$ is finite and $Q$ is generating, in which case $Q$ corresponds to a finite depth subfactor, Ocneanu observed that $\Dcen(\CC)$ and $\Dcen(\QmodQ)$ have the same fusion rules through a characterization of the fusion rules in terms of the associated TQFT~\cite{MR1642584}*{p.~641 and Theorem~12.29}. In fact, by a result of Schauenburg~\cite{MR1822847} this can be strengthened to an equivalence of tensor categories $\Dcen(\CC) \cong \Dcen(\QmodQ)$ without the finiteness assumption on $\CC$ and the generating assumption on $Q$. Although this result is stated in the algebraic context without involution, it can be further extended to our framework of ind-C$^*$-tensor categories in a straightforward way.

Before we explain this, let us remark that we can identify the categories $\indcat{\Qmod_\CC}$ and $\Qmod_{\indC}$. Indeed, for any object $M$ in $\Qmod_{\indC}$ the morphism $d(Q)^{-1/2}m_M^*$ gives an isometric $Q$-modular embedding of $M$ into $Q\otimes M$. Since $Q\otimes M$ lies in $\indcat{\Qmod_\CC}$ and this category is closed under subobjects, it follows that~$M$ is isomorphic to an object in $\indcat{\Qmod_\CC}$. The same remark applies to $Q$-bimodules.

Turning to the equivalence between $\ZC$ and $\Dcen(\indcat\QmodQ)$, the starting point is that, as observed by Schauenburg, if $c$ is a unitary half-braiding on an ind-object $X$ of $\CC$, then $X \otimes Q$ has the structure of a $Q$-bimodule in $\indC$, hence $X\otimes Q$ can be considered as an ind-object of $\QmodQ$.

\begin{lemma}
The pair $(m^r, m^l) = (\iota \otimes m_Q, (\iota \otimes m_Q) (c_Q \otimes \iota))$ defines the structure of a $Q$-bimodule on $X \otimes Q$.
\end{lemma}

\bp
Let us concentrate on the $*$-compatibility of $m^l$, since this is the only new property compared to~\cite{MR1822847}. We have $m^{l*}=(c_Q^* \otimes \iota_Q) (\iota_X \otimes w)$ and we want to show that this is equal to $(\iota_Q \otimes m^l)(wv\otimes\iota_X\otimes\iota_Q)$. Since $\iota\otimes c_Q=(c_Q^*\otimes\iota)c_{Q\otimes Q}$,  we have
$$
(\iota_Q \otimes m^l)(wv\otimes\iota_X\otimes\iota_Q) = (c_Q^* \otimes m_Q) (c_{Q \otimes Q} \otimes \iota_Q) (w v \otimes \iota_X\otimes\iota_Q).
$$
Using the naturality of the half-braiding, this is equal to $(c_Q^* \otimes m_Q) (\iota_X \otimes w v \otimes \iota_Q)$, which is indeed equal to $(c_Q^* \otimes \iota_Q) (\iota_X \otimes w)$ by the $*$-compatibility of $m_Q$.
\ep

Next, we define a unitary half-braiding $\tilde c$ on $X\otimes Q\in\indcat{\QmodQ}$. Let $Y$ be a $Q$-bimodule. Then $Y \otimes X$ and $X \otimes Y$ are models of $Y \otimes_Q (Q \otimes X)$ and $(X \otimes Q) \otimes_Q Y$, with the structure morphisms of the tensor product given by $m_Y^r \otimes \iota_X$ and $\iota_X \otimes m_Y^l$. Since $c_Q$ is by definition a morphism of left $Q$-modules, it induces a unitary morphism $Y \otimes_Q (Q \otimes X) \to Y \otimes_Q (X \otimes Q)$, so we get an isomorphism $Y \otimes_Q (X \otimes Q)\cong Y\otimes X$. Then, up to these isomorphisms, we define $\tilde{c}_Y\colon Y \otimes_Q (X \otimes Q) \to (X \otimes Q) \otimes_Q Y$ simply as the morphism $c_Y\colon Y\otimes X\to X\otimes Y$.

\begin{lemma} \label{lem:schaun-braiding}
The unitaries $\tilde c_Y$ form a half-braiding on $X\otimes Q\in\indcat{\QmodQ}$.
\end{lemma}

\bp
Since the proof of the corresponding statement in~\cite{MR1822847} is omitted, let us briefly indicate the argument. The $Q$-bimodule structure on $Y\otimes X\cong Y\otimes_Q(X\otimes Q)$ is given by
$$
m^l_{Y\otimes X}=m^l_Y\otimes\iota_X, \ \ m^r_{Y\otimes X}=(m^r_Y\otimes\iota_X)(\iota_Y\otimes c_Q^*),
$$
and similarly the  $Q$-bimodule structure on $X\otimes Y\cong (X\otimes Q)\otimes_Q Y$ is given by
$$
m^l_{X\otimes Y}=(\iota_X\otimes m^l_Y)(c_Q\otimes\iota_Y),\ \ m^r_{X\otimes Y}=\iota_X\otimes m^r_Y.
$$
Using this it is easy to check that the morphism $c_Y\colon Y\otimes X\to X\otimes Y$ is $Q$-bimodular.

Next, the morphism $\tilde c_Y\colon Y\otimes_Q(X\otimes Q)\to (X\otimes Q)\otimes_Q Y$ is defined by the morphism
$$
\sigma_Y=(\iota_X\otimes v\otimes\iota_Y)c_Y(m^r_Y\otimes\iota_X)(\iota_Y\otimes c^*_Q)\colon Y\otimes X\otimes Q\to X\otimes Q\otimes Y.
$$
As $c_Y(m^r_Y\otimes\iota)=(\iota\otimes m^r_Y)(c_Y\otimes\iota)(\iota\otimes c_Q)$, we have
$$
\sigma_Y=(\iota_X\otimes v\otimes\iota_Y)(\iota_X\otimes m^r_Y)(c_Y\otimes\iota_Q).
$$
From this we deduce that $(\sigma_Z\otimes\iota_Y)(\iota_Z\otimes\sigma_Y)=\sigma_{Z\otimes Y}$. By the naturality of $c$, we have $(\iota_{X\otimes Q}\otimes P_{Z,Y})\sigma_{Z\otimes Y}=\sigma_{Z\otimes_QY}(P_{Z,Y}\otimes\iota_{X\otimes Q})$. Hence $\tilde c$ satisfies $(\tilde c_Z\otimes_Q\iota_Y)(\iota_Z\otimes_Q\tilde c_Y)=\tilde c_{Z\otimes_Q Y}$.
\ep

Thus, putting $F(X, c) = (X \otimes Q, \tilde{c})$, we obtain a C$^*$-tensor functor $F\colon\ZC \to \Dcen(\indcat\QmodQ)$.

\smallskip

We will show that an inverse functor can be obtained by exactly the same construction using the dual $Q$-system $\hat Q$, modulo an equivalence of the categories $\hatQbimod$ and $\CC$ that we are now going to explain.

Let $L \colon \CC \to \Qmod$ be the free module functor $U \mapsto Q \otimes U$, and $O \colon \Qmod \to \CC$ be the forgetful functor. As was already observed in Section~\ref{sec:prelim-q-system}, we have the adjunction
$$
\CC(O (M), U) \cong \Mor_{\Qmod}(M, L (U)),
$$
induced by the natural transformations
\begin{align*}
\eta_M &= d(Q)^{-\hlf} m^{l*}_M\colon M \to L O (M) = Q \otimes M,&
\epsilon_U &= d(Q)^{\hlf} v^* \otimes \iota_U\colon OL(U)\to U.
\end{align*}
The composition $T = L O$ has the structure of a \emph{monad}~\cite{MR1712872}*{Chapter~VI}. With the above normalization of~$\eta$ and~$\epsilon$, the multiplication $\mu\colon T^2\to T$ is given by $\hat w^*\otimes_Q\iota_M\colon \hat Q\otimes_Q\hat Q\otimes_Q M=Q\otimes Q\otimes M\to \hat Q\otimes_Q M=Q\otimes M$. Then a $T$-algebra structure $\hat{m}_M \colon T( M)=\hat Q\otimes_Q M  = Q \otimes M\to M$ on $M \in \Qmod$ is precisely an algebraic left $\hat{Q}$-module structure on $M$. We consider only $T$-algebras satisfying the $*$-compatibility condition, which can be written as $\hat{m} = \eta^* \mu T(\hat{m}^*)$. Then the category $(\Qmod)^T$ of $T$-algebras is the category $\hat{Q}\mhyph\mathrm{mod}_{\Qmod}$ of left $\hat Q$-modules in $\Qmod$.

As observed before, the categories $\CC$ and $\Qmod$ both have coequalizers, and the Frobenius reciprocity implies that a coequalizer in $\Qmod$ is also a one in $\CC$. This means that the functor $L$ creates coequalizers, and Beck's theorem~\cite{MR1712872}*{Section~VI.7} implies that $(\Qmod)^T$ is equivalent
to $\CC$. More precisely, the comparison functor $P\colon (\Qmod)^T \to \CC$ is characterized as the coequalizer diagram $Q \otimes W \rightrightarrows W \to P (W)$ for the morphisms $d(Q)^{\hlf} v^* \otimes \iota_W$ and $\hat{m}_W$ for $W \in (\Qmod)^T$ and the inverse functor $\CC\to (\Qmod)^T=\hat{Q}\mhyph\mathrm{mod}_{\Qmod}$ is given by $U\mapsto Q\otimes U$.

The above considerations can be carried out for the right $\hat{Q}$-modules in $\modQ$ and the right and/or left $\hat{Q}$-modules in $\QmodQ$, which leads to the equivalences of C$^*$-categories
\begin{align*}
\mathrm{mod}_{\modQ}\mhyph\hat{Q} &\cong \CC,&
{\hatQmod_{\QmodQ}} &\cong \modQ,&
\hat Q\mhyph\mathrm{mod}_{\QmodQ}\mhyph\hat Q &\cong \CC.
\end{align*}
We of course also have similar equivalences for the ind-categories.

\begin{proposition}\label{prop:hat-Q-bimod-eqv-C}
The equivalence $\hatQbimod \cong \CC$ of C$^*$-categories can be extended to an equivalence of C$^*$-tensor categories.
\end{proposition}

\bp
The equivalence $\CC \to \hatQbimod$ is given by $X \mapsto Q \otimes X \otimes Q$ at the level of objects. We have to show that $Q \otimes X \otimes Y \otimes Q$ becomes a model of $(Q \otimes X \otimes Q) \otimes_{\hat{Q}} (Q \otimes Y \otimes Q)$ in a natural way.

The right $\hat{Q}$-module structure on $Q \otimes X \otimes Q$ is given by $d(Q)^\hlf\iota_{Q \otimes X} \otimes v^* \otimes \iota_Q$, where we as usual use $Q \otimes X \otimes Q \otimes Q$ as a model of $(Q \otimes X \otimes Q) \otimes_Q \hat{Q}$ with the structure morphism $P_{Q \otimes X \otimes Q, \hat{Q}} = \iota_{Q \otimes X} \otimes m_Q \otimes \iota_Q$. The left $\hat{Q}$-module structure on $Q \otimes Y \otimes Q$ can be described in a similar way, and we also have
$$
(Q \otimes X \otimes Q) \otimes_Q \hat{Q} \otimes_Q (Q \otimes Y \otimes Q) = Q \otimes X \otimes Q \otimes Q \otimes Y \otimes Q.
$$
Thus, $(Q \otimes X \otimes Q) \otimes_{\hat{Q}} (Q \otimes Y \otimes Q)$ is a coequalizer of $\iota_{Q \otimes X} \otimes v^* \otimes \iota_{Q \otimes Y \otimes Q}$ and $\iota_{Q \otimes X \otimes Q} \otimes v^* \otimes \iota_{Y \otimes Q}$. Since the endofunctors on $\CC$ of the form $Z \mapsto V \otimes Z \otimes W$ for $V, W \in \CC$ are exact, we thus see that it is enough to show that $v^*\colon Q\to\un$ is a coequalizer of the morphisms $v^* \otimes \iota_Q$ and $\iota_Q \otimes v^*$. But this is obvious as $v^*v=\iota$.
\ep

Now, the same construction as for $F$ using the dual $Q$-system $\hat{Q}$ provides a C$^*$-tensor functor $$\tilde{F}\colon \Dcen(\indcat{\QmodQ}) \to \Dcen(\indcat{\hatQbimod}).$$
Take $(X,c)\in\ZC$. Then $F(X,c)=(X\otimes Q,\tilde c)$ and $\tilde F F(X,c)=((X\otimes Q)\otimes_Q\hat Q,\tilde{\tilde c})$. We can identify $(X\otimes Q)\otimes_Q\hat Q$ with $X\otimes\hat Q=X\otimes Q\otimes Q$. Under this identification the right $Q$- and $\hat Q$-module structures are obvious, but this is less so for the left module structures.

\begin{lemma}
The morphism $c_Q\otimes\iota_Q\colon Q\otimes X\otimes Q\to X\otimes Q\otimes Q=(X\otimes Q)\otimes_Q\hat Q$ is an isomorphism of $Q$- and $\hat Q$-bimodules.
\end{lemma}

\bp We only have to consider the left module structures. As was already stated in the proof of Lemma~\ref{lem:schaun-braiding}, the morphism $c_{\hat Q}\colon \hat Q\otimes X\to X\otimes\hat Q= (X\otimes Q)\otimes_Q\hat Q$ is a morphism of $Q$-bimodules. Furthermore, since the isomorphism
$\hat Q\otimes_Q(X\otimes Q)\cong \hat Q\otimes X$ defined by the isomorphisms $c^*_Q\colon X\otimes Q\to Q\otimes X$ and $\hat Q\otimes_Q Q\cong \hat Q$ does not affect the first factor of $\hat Q=Q\otimes Q$, by definition of the left $\hat Q$-module structure on $(X\otimes Q)\otimes_Q\hat Q$ we also see that $c_{\hat Q}\colon \hat Q\otimes X\to X\otimes\hat Q= (X\otimes Q)\otimes_Q\hat Q$ is a morphism of left $\hat Q$-modules. Therefore in order to prove the lemma it suffices to check that $c^*_{\hat Q}(c_Q\otimes\iota_Q)\colon Q\otimes X\otimes Q\to\hat Q\otimes X$ is a morphism of left $Q$- and $\hat Q$-modules. But this is obvious as $c^*_{\hat Q}(c_Q\otimes\iota_Q)=\iota_Q\otimes c_Q^*$.
\ep

The equivalence $\indC\cong \indcat{\hatQbimod}$, $X\mapsto Q\otimes X\otimes Q$, defines an equivalence of the corresponding Drinfeld centers $\ZC\cong\Dcen(\indcat{\hatQbimod})$.

\begin{lemma}
The functor $\tilde F F$ is naturally unitarily isomorphic to the equivalence functor $\ZC\to\Dcen(\indcat{\hatQbimod})$.
\end{lemma}

\bp We have to find out what happens with the half-braiding $\tilde{\tilde c}$ under the isomorphism $(X\otimes Q)\otimes_Q\hat Q=X\otimes\hat Q\cong Q\otimes X\otimes Q$ from the previous lemma.

Take an object $Y\in\CC$. Consider the morphism
$$
\tilde {\tilde c}_{Q\otimes Y\otimes Q}\colon (Q\otimes Y\otimes Q)\otimes_{\hat Q}((X\otimes Q)\otimes_Q\hat Q)\to
((X\otimes Q)\otimes_Q\hat Q)\otimes_{\hat Q}(Q\otimes Y\otimes Q).
$$
If we identify the right hand side with $(X\otimes Q)\otimes_Q(Q\otimes Y\otimes Q)$, then by the proof of Lemma~\ref{lem:schaun-braiding} this morphism is implemented by the morphism
$$
\tilde c_{Q\otimes Y\otimes Q}\colon (Q\otimes Y\otimes Q)\otimes_{Q}(X\otimes Q)\to
(X\otimes Q)\otimes_Q(Q\otimes Y\otimes Q)
$$
together with the multiplication map $(Q\otimes Y\otimes Q)\otimes_Q\hat Q\to Q\otimes Y\otimes Q$ defining the right $\hat Q$-module structure on $Q\otimes Y\otimes Q$.  Similarly, if we identify $(X\otimes Q)\otimes_Q(Q\otimes Y\otimes Q)$ with $X\otimes Q\otimes Y\otimes Q$, then $\tilde c_{Q\otimes Y\otimes Q}$ is implemented by the morphism $c_{Q\otimes Y\otimes Q}$ together with the multiplication map $Q\otimes Y\otimes Q\otimes Q\to Q\otimes Y\otimes Q$. To summarize, if we identify $((X\otimes Q)\otimes_Q\hat Q)\otimes_{\hat Q}(Q\otimes Y\otimes Q)$ with $X\otimes Q\otimes Y\otimes Q$, then $\tilde {\tilde c}_{Q\otimes Y\otimes Q}$ is induced by the morphism
$$
d(Q)^{1/2}(\iota_{X\otimes Q\otimes Y}\otimes v^*w^{(2)*}\otimes\iota_Q)(c_{Q\otimes Y\otimes Q}\otimes\iota_{Q\otimes\hat Q})\colon Q\otimes Y\otimes Q\otimes X\otimes Q\otimes\hat Q\to X\otimes Q\otimes Y\otimes Q,
$$
where $w^{(2)}=(w\otimes\iota)w=(\iota\otimes w)w$. Identifying $(X\otimes Q)\otimes_Q\hat Q$ with $X\otimes\hat Q$, we can equivalently write that $\tilde {\tilde c}_{Q\otimes Y\otimes Q}$ is induced by
$$
d(Q)^{1/2}(\iota_{X\otimes Q\otimes Y}\otimes v^*w^*\otimes\iota_Q)(c_{Q\otimes Y\otimes Q}\otimes\iota_{\hat Q})\colon Q\otimes Y\otimes Q\otimes X\otimes\hat Q\to X\otimes Q\otimes Y\otimes Q,
$$
The half-braiding $\tilde{\tilde c}$ on $(X\otimes Q)\otimes_Q\hat Q=X\otimes\hat Q$ defines a half-braiding $c'$ on $Q\otimes X\otimes Q\cong X\otimes \hat Q$. Then we conclude that $c'_{Q\otimes Y\otimes Q}$ is implemented by the morphism $Q\otimes Y\otimes Q\otimes Q\otimes X\otimes Q\to Q\otimes X\otimes Y\otimes Q$ given by
\begin{multline*}
d(Q)^{1/2}(c_Q^*\otimes\iota_{Y\otimes Q})(\iota_{X\otimes Q\otimes Y}\otimes v^*w^*\otimes\iota_Q)(c_{Q\otimes Y\otimes Q}\otimes\iota_{\hat Q})(\iota_{Q\otimes Y\otimes Q}\otimes c_Q\otimes\iota_Q)\\
=d(Q)^{1/2}(\iota_Q\otimes c_Y\otimes\iota_Q)(\iota_{Q\otimes Y}\otimes v^*w^*\otimes\iota_{X\otimes Q}).
\end{multline*}
Since $d(Q)^{1/2}(\iota_{Q\otimes Y}\otimes v^*w^*\otimes\iota_{X\otimes Q})$ is precisely the structure morphism
$Q\otimes Y\otimes Q\otimes Q\otimes X\otimes Q\to (Q\otimes Y\otimes Q)\otimes_{\hat Q}(Q\otimes X\otimes Q)=Q\otimes Y\otimes X\otimes Q$, we thus see that $c'$ coincides with the half-braiding on $Q\otimes X\otimes Q$ defined by the half-braiding $c$ on $X$ using the equivalence functor $\indC\to \indcat{\hatQbimod}$.
\ep

It follows that $F$ defines an equivalence between the category $\ZC$ and a subcategory of the category $\Dcen(\indQmodQ)$. For the same reason $\tilde F$ defines an equivalence between $\Dcen(\indQmodQ)$ and a subcategory of $\Dcen(\indcat{\hatQbimod})$. We then conclude that $F$ is an equivalence of the categories $\ZC$ and $\Dcen(\indQmodQ)$ and, modulo the equivalence $\ZC \cong \Dcen(\indcat{\hatQbimod})$, $\tilde F$ is an inverse functor. We thus have the following version of the result of Schauenburg~\cite{MR1822847}.

\begin{theorem}
The functor $F\colon\ZC\to\Dcen(\indQmodQ)$, $(X,c)\mapsto (X\otimes Q,\tilde c)$, is a unitary monoidal equivalence of categories.
\end{theorem}

We remark that Schauenburg's result does not give an equivalence of the categories $\ZsC$ and $\Dcen_s(\indQmodQ)$ of spherical objects, or in other words, an equivalence of the representation categories of~$C^*(\CC)$ and $C^*(\QmodQ)$. These categories are not equivalent in general. But the result does give rise to functors $\tilde F\colon\Rep C^*(\CC)\to\Rep C^*(\QmodQ)$ and $\tilde G\colon\Rep C^*(\QmodQ)\to \Rep C^*(\CC)$ such that $\tilde F\tilde G$ and $\tilde G\tilde F$ map every representation to a subrepresentation.

\subsection{Comparison of almost invariant vectors}

If $(X,c) \in \ZC$, the invariant vectors for the fusion algebra of $\CC$ are the vectors in $\Mor_\ZC(\un, X)$. Since Schauenburg's induction is a C$^*$-tensor functor, they correspond to the vectors in $\Mor_{\Dcen(\indQmodQ)}(Q, X \otimes Q)$, or the invariant vectors for the fusion algebra of $\QmodQ$. It is less obvious what happens with almost invariant vectors, since this is a notion that does not make sense within the Drinfeld center itself.

\begin{theorem}
\label{thm:almost-inv-vec-mor-eqv}
For any $(X,c) \in \ZC$, the representation of $C^*(\CC)$ defined by $(X,c)$ weakly contains the trivial representation if and only if the representation of $C^*(\QmodQ)$ defined by the image $F(X,c)=(X\otimes Q,\tilde c)$ of $(X,c)$ under Schauenburg's induction weakly contains the trivial representation.
\end{theorem}

For the proof we need to understand the morphisms $Q\to X\otimes Q$ in $\indQmodQ$.

\begin{lemma} \label{lem:char-vec-sp-F-X}
For any $(X, c) \in \ZC$, the morphisms in $\Mor_{\indQmodQ}(Q, X \otimes Q)$ are of the form
$$
T^S=(\iota_X\otimes v^*w^*\otimes\iota_Q)(c_Q\otimes\iota_Q\otimes\iota_Q)(\iota_Q\otimes S\otimes\iota_Q\otimes\iota_Q)w^{(3)},
$$
where $w^{(n)}=(w\otimes\iota)w^{(n-1)}$ ($w^{(1)}=w$) and $S\in\Mor_\indC(Q,X)$.
\end{lemma}

\bp
The right $Q$-modular morphisms $Q\to X\otimes Q$ have the form $(T\otimes\iota)w$. By definition, such a morphism respects the left actions of $Q$ if and only if
\begin{equation}\label{eq:bimod-hom-char-in-orig}
(T\otimes\iota_Q)w w^*=(\iota_X \otimes w^*)(c_Q \otimes \iota_Q) (\iota_Q \otimes T \otimes \iota_Q) (\iota_Q \otimes w).
\end{equation}
Therefore in order to prove the lemma it suffices to show that a morphism $T$ satisfies \eqref{eq:bimod-hom-char-in-orig} if and only if it has the form
$$
T=(\iota_X\otimes v^*w^*)(c_Q\otimes\iota_Q)(\iota_Q\otimes S\otimes\iota_Q)w^{(2)}
$$
for some $S$. If $T$ satisfies \eqref{eq:bimod-hom-char-in-orig}, then multiplying \eqref{eq:bimod-hom-char-in-orig} by $\iota\otimes v^*$ on the left and by $w$ on the right we see that~$T$ indeed has the above form, with $S=d(Q)^{-1}T$. Conversely, let $T$ be as above. Then
\begin{align*}
(T\otimes\iota_Q)w w^*&=(\iota_X\otimes v^*w^*\otimes\iota_Q)(c_Q\otimes\iota_Q\otimes\iota_Q)(\iota_Q\otimes S\otimes\iota_Q\otimes\iota_Q)w^{(3)}w^*\\
&=(\iota_X\otimes v^*w^*\otimes\iota_Q)(\iota_X\otimes\iota_Q\otimes w)(c_Q\otimes\iota_Q)(\iota_Q\otimes S\otimes\iota_Q)w^{(2)}w^*\\
&=(\iota_X\otimes w^*)(c_Q\otimes\iota_Q)(\iota_Q\otimes S\otimes\iota_Q)w^{(2)}w^*.
\end{align*}
On the other hand,

\smallskip
$\displaystyle
(\iota_X \otimes w^*)(c_Q \otimes \iota_Q) (\iota_Q \otimes T \otimes \iota_Q) (\iota_Q \otimes w)$
\begin{align*}
&=(\iota_X \otimes w^*)(c_Q \otimes \iota_Q) (\iota_Q\otimes\iota_X\otimes v^*w^*\otimes\iota_Q)(\iota_Q\otimes c_Q\otimes\iota_Q\otimes\iota_Q)(\iota_Q\otimes\iota_Q\otimes S\otimes\iota_Q\otimes\iota_Q)(\iota_Q \otimes w^{(3)})\\
&=(\iota_X \otimes w^*)(c_Q \otimes \iota_Q) (\iota_Q\otimes\iota_X\otimes w^*)(\iota_Q\otimes c_Q\otimes\iota_Q)(\iota_Q\otimes\iota_Q\otimes S\otimes\iota_Q)(\iota_Q \otimes w^{(2)})\\
&=(\iota_X \otimes w^*)(\iota_X\otimes\iota_Q\otimes w^*)(c_{Q\otimes Q}\otimes\iota_Q)(\iota_Q\otimes\iota_Q\otimes S\otimes\iota_Q)(\iota_Q \otimes w^{(2)})\\
&=(\iota_X \otimes w^*)(\iota_X\otimes w^*\otimes\iota_Q)(c_{Q\otimes Q}\otimes\iota_Q)(\iota_Q\otimes\iota_Q\otimes S\otimes\iota_Q)(\iota_Q \otimes w^{(2)})\\
&=(\iota_X \otimes w^*)(c_Q\otimes\iota_Q)(\iota_Q\otimes S\otimes\iota_Q)(w^*\otimes\iota_Q\otimes\iota_Q)(\iota_Q \otimes w^{(2)}).
\end{align*}
Since by applying the Frobenius condition twice we obtain
$$
w^{(2)}w^*=(w\otimes\iota)(w^*\otimes\iota)(\iota\otimes w)=(w^*\otimes\iota\otimes\iota)(\iota\otimes w^{(2)}),
$$
this proves the lemma.
\ep

\begin{proof}[Proof of Theorem~\ref{thm:almost-inv-vec-mor-eqv}]
Let $\{\xi_i\}_i$ be a net of almost invariant unit vectors for $C^*(\CC)$ in $\Mor_\indC(\un,X)$. Recall that by Lemma~\ref{lem:aa} this means that $c_Y(\iota_Y\otimes\xi_i)-\xi_i\otimes \iota_Y\to0$ for all $Y\in\CC$. Consider the vectors
$$
\tilde\xi_i=d(Q)^{-1}T^{\xi_iv^*}\in\Mor_{\indQmodQ}(Q,X\otimes Q).
$$
We claim that they are almost unit and almost invariant for $C^*(\QmodQ)$.

By definition, as $i$ grows the vector $\tilde\xi_i$ becomes close (in $\Mor_\indC(Q,X\otimes Q)$) to
$$
d(Q)^{-1}(\iota_X\otimes v^*w^*\otimes\iota_Q)(\xi_iv^*\otimes\iota_Q\otimes\iota_Q\otimes\iota_Q)w^{(3)}
=d(Q)^{-1}(\xi_i\otimes\iota_Q)w^*w=\xi_i\otimes\iota_Q,
$$
hence $\|\tilde\xi_i\|\to1$.

Now, take a $Q$-bimodule $Y$. We want to check that $\tilde c_Y(\iota_Y\otimes_Q\tilde\xi_i)$ becomes close $\tilde\xi_i\otimes_Q\iota_Y$ for large $i$. Recall that as a model of $Y\otimes_Q(X\otimes Q)$ we take $Y\otimes X$ with the structure morphism $P_{Y,X\otimes Q}=(m^r_Y\otimes\iota_X)(\iota_Y\otimes c_Q^*)$, as a model of $(X\otimes Q)\otimes_Q Y$ we take $X\otimes Y$ with $P_{X\otimes Q,Y}=\iota_X\otimes m^l_Y$, and then $\tilde c_Y=c_Y$. Since $\iota_Y\otimes v$ is a right inverse of $P_{Y,Q}=m^r_Y$ and $v\otimes\iota_Y$ is a right inverse of $P_{Q,Y}=m^l_Y$, it follows that
\begin{align}
\tilde c_Y(\iota_Y\otimes_Q\tilde\xi_i)&=c_Y(m^r_Y\otimes\iota_X)(\iota_Y\otimes c_Q^*)(\iota_Y\otimes\tilde\xi_i)(\iota_Y\otimes v)=(\iota_X\otimes m^r_Y)(c_Y\otimes\iota_Q)(\iota_Y\otimes\tilde\xi_i)(\iota_Y\otimes v),\label{eq:aa-comparison1}\\
\tilde\xi_i\otimes_Q\iota_Y&=(\iota_X\otimes m^l_Y)(\tilde\xi_i\otimes\iota_Y)(v\otimes\iota_Y).\label{eq:aa-comparison2}
\end{align}
Since, as we have already shown, $\tilde\xi_i\approx\xi_i\otimes\iota_Q$, and $c_Y(\iota_Y\otimes\xi_i)\approx\xi_i\otimes\iota_Y$, both \eqref{eq:aa-comparison1} and \eqref{eq:aa-comparison2} become close to $\xi_i\otimes\iota_Y$ for large $i$. Thus, rescaling $\tilde{\xi}_i$ to a unit vector we obtain almost invariant unit vectors for $C^*(\QmodQ)$.

Finally, the opposite implication also follows from the above, since an inverse of $F$ is defined in the same way as $F$ using the dual $Q$-system $\hat{Q}$.
\end{proof}

\begin{corollary}
Property (T) is invariant under weak monoidal Morita equivalence of C$^*$-tensor categories.
\end{corollary}

\begin{remark}
 If $Q$ generates $\CC$, so that $\CC$, $\Qmod$, $\modQ$ and $\QmodQ$ can be regarded as bimodule categories defined by an extremal finite index subfactor, the above corollary has been already obtained by Popa~\cite{MR1729488}*{Proposition~9.8} using the universality of the symmetric enveloping algebra and permanence of relative property (T) for finite index intermediate algebras.
\end{remark}

\subsection{Comparison of regular half-braidings}
\label{sec:compar-reg-hbr}

In this last section we assume that $Q$ is a standard \emph{irreducible} $Q$-system.
Let $(M_k)_k$ be representatives of the isomorphism classes of simple objects in $\modQ$. Since $Q\in\modQ$ is simple by assumption, we may assume that $M_e = Q$ for some index $e$. We fix standard solutions $(R_k, \bar{R}_k)$ for $M_k$ once and for all, and denote by $(S_k, \bar{S}_k)$ the corresponding morphisms defined as in the Section~\ref{sec:mod-cat-Frob-recip}. Our goal is to show that the ind-object $\Zreg(\modQ) = \oplus_k M_k \otimes_Q \bar{M}_k$ admits a half-braiding which corresponds to $\Zreg(\QmodQ)$ under Schauenburg's induction.

First, let us construct a half-braiding on $\Zreg(\modQ)$. This goes completely analogously to the construction of $\Zreg(\CC)$. We fix $X \in \CC$ and a standard solution $(R_X, \bar{R}_X)$. For each $k$ and $l$, choose an orthonormal basis $(u^\alpha_{k l})_\alpha$ in $\Mor_{\modQ}(M_l,X \otimes M_k)$, so that $\sum_\alpha u_{k l}^\alpha u_{k l}^{\alpha *}$ is the projection of $X \otimes M_k$ onto the isotypic component corresponding to $M_l$. We then define $c_{X, l k}\colon X \otimes M_k \otimes_Q \bar{M}_k \to M_l \otimes_Q \bar{M}_l \otimes X$ by
$$
c_{X, l k} = \left(\frac{d_k}{d_l}\right)^{\hlf} \sum_\alpha (u^{\alpha *}_{k l} \otimes_Q u^{\alpha \vee}_{k l} \otimes \iota_X) (\iota_{ X \otimes M_k \otimes_Q \bar{M}_k } \otimes R_X),
$$
where $d_k=d(M_k)$ is the dimension of $M_k$ in $\CC$ and $u^{\alpha \vee}_{k l}$ is defined using $\bar{R}_l$ and $\bar{R}_{X \otimes M_k} = (\iota \otimes \bar{R}_k \otimes \iota) \bar{R}_X$.

\begin{lemma}
The morphism $c_X \colon X \otimes \Zreg(\modQ) \to \Zreg(\modQ) \otimes X$ defined by the morphisms $(c_{X, l k})_{l, k}$ is unitary.
\end{lemma}

\bp
The proof follows that of Lemma~\ref{lem:reg-half-br-candid-unitary}. First, in order to see that it is isometry, we need to show
$$
\frac{d_k}{d_l} u^{\beta}_{k l} u^{\alpha *}_{k l} \otimes_Q (\iota \otimes \Tr_{\bar X}) (u^{\beta * \vee}_{k l} u^{\alpha \vee}_{k l} ) = \delta_{\alpha, \beta} \iota.
$$
Since $\bar M_k$ is a simple left $Q$-module, the endomorphism $(\iota \otimes \Tr_{\bar X}) (u^{\beta * \vee}_{k l} u^{\alpha \vee}_{k l} ) $ of $\bar M_k$ is scalar, and we find this scalar as in the proof of Lemma~\ref{lem:reg-half-br-candid-unitary} by computing its trace.

Next, to verify that $c_X$ is unitary, we need to show that $\sum_k c_{X, l k} c_{X, l k}^* = \iota_{M_l \otimes_Q \bar{M}_l \otimes X}$, or
$$
\sum_{k, \alpha} \frac{d_k}{d_l} (\iota_l \otimes_Q u_{k l}^{\alpha \vee} \otimes \iota_X) (\iota_{X\otimes M_k\otimes_Q\bar M_k} \otimes R_X R_X^*) (\iota_l \otimes_Q u_{k l}^{\alpha \vee*} \otimes \iota_X) = \iota_{M_l \otimes_Q \bar{M}_l \otimes X}.
$$
This would follow if for each $k$ the morphism
$$
\sum_{\alpha} \frac{d_k}{d_l} (u_{k l}^{\alpha \vee} \otimes \iota_X) (\iota_{\bar k} \otimes R_X R_X^*) (u_{k l}^{\alpha \vee*} \otimes \iota_X)
$$
was the projection of $\bar{M}_l \otimes X$ onto the isotypic component corresponding to $\bar{M}_k$. This is again proved as in Lemma~\ref{lem:reg-half-br-candid-unitary} by observing that the Frobenius reciprocity isomorphism
$\CC(\bar{M}_l, \bar{M}_k \otimes \bar{X})\cong\CC(\bar{M}_l \otimes X, \bar{M}_k)$ defines an isomorphism
$\Mor_{\Qmod}(\bar{M}_l, \bar{M}_k \otimes \bar{X}) \to \Mor_{\Qmod}(\bar{M}_l \otimes X, \bar{M}_k)$.
\ep

\begin{proposition}
The unitaries $c_X \colon X \otimes \Zreg(\modQ) \to \Zreg(\modQ) \otimes X$ form a half-braiding on $\CC$.
\end{proposition}

\bp
This is proved in the same way as Theorem~\ref{them:reg-half-braiding}, by considering the decomposition of $X\otimes Y\otimes M_k$ defined by decompositions of $Y\otimes M_k$ and $X\otimes M_l$.
\ep

When we need to be specific, we will write $c^{\modQ}$ for the above half-braiding $(c_X)_X$.  But if there is no danger of confusion, we will omit the half-braiding altogether and simply write $\Zreg(\modQ)$ for the object $(\Zreg(\modQ),c^{\modQ})\in\ZC$.

\begin{theorem}
\label{thm:compare-Zreg-Qmod-and-Zreg-QmodQ}
The image of $\Zreg(\modQ)$ under Schauenburg's induction $\ZC\to\Dcen(\indQmodQ)$ is isomorphic to
$\Zreg(\QmodQ)$.
\end{theorem}

We need some preparation to prove this theorem. Consider the canonical morphism $\zeta_e\colon Q=M_e\otimes_Q\bar M_e\to\Zreg(\modQ)$ in $\indC$.

\begin{lemma}
\label{lem:characterization-mor-from-ZregmodQ}
The morphism $\zeta_e$ satisfies identity~\eqref{eq:bimod-hom-char-in-orig}.
\end{lemma}

\bp
Consider the right hand side of~\eqref{eq:bimod-hom-char-in-orig} for $T=\zeta_e$, that is, the expression
$$
(\iota \otimes w^*) (c_Q^{\modQ} \otimes \iota)(\iota\otimes \zeta_e\otimes\iota) (\iota\otimes w).
$$
Composing it on the left with the projection onto $M_l\otimes_Q\bar M_l\otimes Q$ we get
$$
\left(\frac{d_e}{d_l}\right)^{\hlf}\sum_\alpha(\iota_{M_l\otimes_Q\bar M_l}\otimes w^*)
(u_l^{\alpha*} \otimes_Q u_l^{\alpha\vee}\otimes\iota_Q\otimes\iota_Q)(\iota_Q\otimes\iota_Q\otimes wv\otimes\iota_Q)(\iota_Q\otimes w),
$$
where $(u^\alpha_l)_\alpha$ is an orthonormal basis in $\Mor_{\modQ}(M_l,Q \otimes Q)$ and we identify $(Q\otimes Q)\otimes_Q(Q\otimes Q)$ with~$Q^{\otimes 3}$. Using that $(\iota\otimes w^*)(wv\otimes\iota)=w$, we see that the above expression equals
$$
\left(\frac{d_e}{d_l}\right)^{\hlf}\sum_\alpha(\iota_{M_l\otimes_Q\bar M_l}\otimes \iota_Q)
(u_l^{\alpha*} \otimes_Q u_l^{\alpha\vee}\otimes\iota_Q)(\iota_Q\otimes w^{(2)}).
$$
If $l\ne e$, then this expression is zero, since $u_l^{\alpha\vee}w\colon Q\to \bar M_l$ is a morphism of left $Q$-modules. For $l=e$ we may assume that $u_e^{\alpha_0} = d(Q)^{-\hlf} w$ for some $\alpha_0$. Then the above formula picks up only the term corresponding to $\alpha_0$, and using that $u_e^{\alpha_0} w=d(Q)^{-1/2}w^*w=d(Q)^{1/2}\iota$ we see that it gives
$$
(w^*\otimes_Q\iota_Q\otimes\iota_Q)(\iota_Q\otimes w)=ww^*\colon Q\otimes Q\to Q\otimes Q=M_e\otimes_Q\bar M_e\otimes Q.
$$
Thus the right hand side of~\eqref{eq:bimod-hom-char-in-orig} equals  $(\zeta_e\otimes\iota)ww^*$, so $\zeta_e$ satisfies~\eqref{eq:bimod-hom-char-in-orig}.
\ep

\begin{lemma}
\label{lem:restr-from-ZregmodQ-inj}
 For any $(X,c) \in \ZC$, the map $\Mor_\ZC(\Zreg(\modQ), X) \to \Mor_\indC(Q, X)$, $T\mapsto T\zeta_e$, is injective.
\end{lemma}

\bp Fix an index $k$. We claim that the canonical morphism $M_k\otimes_Q \bar M_k\to \Zreg(\modQ)$ is given by
$$
\zeta_k=d_e^{-3/2}d_k^{1/2}(\iota \otimes \bar{R}_k^*) (c^{\modQ}_{k} \otimes \iota_{\bar k}) (\iota_{k} \otimes \zeta_e \otimes \iota_{\bar k}) (m_{k}^{r *} \otimes_Q m_{\bar k}^{l *}),
$$
where we identify $(M_k\otimes Q)\otimes_Q(Q\otimes\bar M_k)$ with $M_k\otimes Q\otimes\bar M_k$. Indeed, as in the previous lemma, composing this with the projection onto $M_l\otimes_Q\bar M_l$ we see that the above expression gives
$$
d_e^{-1}d_k^{1/2}d_l^{-1/2}\sum_\alpha
(u_l^{\alpha*} \otimes_Q u_l^{\alpha\vee})(m_{k}^{r *} \otimes_Q m_{\bar k}^{l *}),
$$
where $(u^\alpha_l)_\alpha$ is an orthonormal basis in $\Mor_{\modQ}(M_l,M_k \otimes Q)$. If $l\ne k$, this expression is zero. If $l=k$, then it picks up the term corresponding to the isometry $u_k^{\alpha_0}=d(Q)^{-1/2}m_{k}^{r*}$. Since $m_{\bar k}^{l *}=m^{r\vee}_{k}$ by definition, we then see that the above expression is the identity morphism. Thus our claim is proved.

Now, if $T\in\Mor_\ZC(\Zreg(\modQ), X)$, we get
$$
T\zeta_k=d_e^{-3/2}d_k^{1/2}(\iota_X \otimes \bar{R}_k^*) (c_{k} \otimes \iota_{\bar k}) (\iota_{k} \otimes T \zeta_e \otimes \iota_{\bar k}) (m_{k}^{r *} \otimes_Q m_{\bar k}^{l *}).
$$
Since $T$ is determined by the morphisms $T\zeta_k$, we conclude that it is determined by $T\zeta_e$.
\ep

\bp[Proof of Theorem~\ref{thm:compare-Zreg-Qmod-and-Zreg-QmodQ}]
Consider the image $(\Zreg(\modQ)\otimes Q),\tilde c)$ of $(\Zreg(\modQ),c^{\modQ})$ under Schauenburg's induction.
Since by Lemma~\ref{lem:characterization-mor-from-ZregmodQ} the morphism $\zeta_e$ satisfies identity~\eqref{eq:bimod-hom-char-in-orig}, we have an isometric morphism $\xi=d(Q)^{-1/2}(\zeta_e\otimes\iota)w\colon Q\to \Zreg(\modQ)\otimes Q$ in $\indQmodQ$. Let $\phi$ be the positive definite function on $\Irr(\QmodQ)$ defined by $\xi$. Then by Proposition~\ref{prop:Zphi-universal} there exists a unique isometric morphism $T\colon Z_\phi\to \Zreg(\modQ)\otimes Q$ in $\Dcen(\indQmodQ)$ such that $T\xi_\phi=\xi$.

The morphism $T$ is unitary. Indeed, since Schauenburg's induction is an equivalence of categories, the endomorphism $p=\iota-TT^*$ of $\Zreg(\modQ)\otimes Q$ must be of the form $S\otimes\iota$ for some $S\in\End_\ZC(\Zreg(\modQ))$. Since $pT=0$, we have $p\xi=0$, hence $S\zeta_e=0$. By Lemma~\ref{lem:restr-from-ZregmodQ-inj} it follows that $S=0$, so $TT^*=\iota$.

To finish the proof it remains to show that $\phi=\delta_e$, as then $Z_\phi=\Zreg(\QmodQ)$. Let us fix representatives $(X_a)_a$ of the isomorphism classes of simple objects in $\QmodQ$, so that $Z_\phi=A^\phi\text{-}\oplus_a X_a\otimes_Q\bar X_a$. We continue to denote by $e$ the index corresponding to the unit object and assume that $X_e=Q$. Then by \eqref{eq:intertbr} the composition $T_a$ of $T$ with the canonical morphism $X_a\otimes_Q\bar X_a\to Z_\phi$ is given by
$$
T_a = (d^Q_a)^{\hlf}(\iota \otimes_Q \bar R^{Q*}_a) (\tilde{c}_a \otimes_Q \iota_{\bar{a}}) (\iota_a \otimes_Q \xi \otimes_Q \iota_{\bar{a}})\colon X_a\otimes_Q \bar X_a\to \Zreg(\modQ)\otimes Q.
$$
(Recall that the upper index $Q$ indicates that we are dealing with the category $\QmodQ$.)
Since $\pi_\phi([X_a])\xi=(d^Q_a)^{-1/2}T_a\bar R^Q_a$, in order to check that $\phi=\delta_e$ it suffices to show that $\xi^*T_a=T_e^*T_a=0$ for $a\ne e$.

Let us compute the morphisms $T_a$ more explicitly. As we already used in the proof of Lemma~\ref{lem:schaun-braiding}, the morphism $\tilde c_a$ is implemented by the morphism $X_a\otimes \Zreg(\modQ)\otimes Q\to \Zreg(\modQ)\otimes Q\otimes X_a$ given by $$(\iota\otimes v\otimes\iota_a)(\iota\otimes m^r_a)(c^{\modQ}_a\otimes\iota_Q).$$
Since we also have $\bar R^{Q*}_a P_{\bar a,a}=\bar S^*_a$ by \eqref{eq:RQstar}, and $P_{Q,Q}=w^*$, we conclude that $T_a$ is implemented by the morphism $X_a\otimes Q\otimes \bar X_a\to\Zreg(\modQ)\otimes Q$ given by
$$
(d^Q_a)^{\hlf}(\iota\otimes w^*)(\iota\otimes\iota_Q\otimes \bar S^*_a)(\iota\otimes v\otimes\iota_a\otimes\iota_{\bar a})(\iota\otimes m^r_a\otimes\iota_{\bar a})(c^{\modQ}_a\otimes\iota_Q\otimes\iota_{\bar a}) (\iota_a \otimes \xi \otimes \iota_{\bar{a}}).
$$
Since $w^*(v\otimes\iota)=\iota$ and $\bar S^*_a(m^r_a\otimes\iota)=\bar S^*_a(\iota\otimes m^l_{\bar a})$ by Lemma~\ref{lem:Sdescent}, the above expression equals
$$
(d^Q_a)^{\hlf}(\iota\otimes \bar S^*_a)(\iota\otimes \iota_a\otimes m^l_{\bar a})(c^{\modQ}_a\otimes\iota_Q\otimes\iota_{\bar a}) (\iota_a \otimes \xi \otimes \iota_{\bar{a}}).
$$
Using that $\xi=d_e^{-1/2}(\zeta_e\otimes\iota)w$ and $(\iota\otimes m^l_{\bar a})(w\otimes\iota)=m^{l*}_{\bar a}m^l_{\bar a}$ by Lemma~\ref{lem:comput-m-m-star}, and then that $P_{Q,X_{\bar a}}=m^l_{\bar a}$, we finally get that the morphism $T_a$ is implemented by the morphism
$$
\tilde T_a= (d^Q_a)^{\hlf}d_e^{-1/2}(\iota\otimes \bar S^*_a)(c^{\modQ}_a\otimes\iota_{\bar a}) (\iota_a \otimes \zeta_e\otimes \iota_{\bar{a}})(\iota_a\otimes m^{l*}_{\bar a})\colon X_a\otimes X_{\bar a}\to\Zreg(\modQ)\otimes Q.
$$
Consider the component $\tilde T_{ka}$ of this morphism mapping $X_a\otimes X_{\bar a}$ into $M_k\otimes_Q\bar M_k\otimes Q$. By definition,
$$
\tilde T_{ka}= (d^Q_a)^{1/2}d_k^{-1/2}\sum_\alpha(\iota\otimes \bar S^*_a)
(u^{\alpha *}_{ak} \otimes_Q u^{\alpha \vee}_{ak} \otimes \iota_a\otimes\iota_{\bar a}) (\iota_a\otimes\iota_Q\otimes R_a\otimes\iota_{\bar a})(\iota_a\otimes m^{l*}_{\bar a}),
$$
where $(u^\alpha_{ak})_\alpha$ is an orthonormal basis in $\Mor_{\modQ}(M_k,X_a\otimes Q)$. Since by the definition of the maps $\bar S$ we have $(\iota \otimes \bar{S}_a^*) (R_a \otimes \iota) = m_{\bar{a}}^{r *}$, we get
$$
\tilde T_{ka}= (d^Q_a)^{1/2}d_k^{-1/2}\sum_\alpha
(u^{\alpha *}_{ak} \otimes_Q u^{\alpha \vee}_{ak} \otimes \iota_Q) (\iota_a\otimes\iota_Q\otimes m_{\bar{a}}^{r *})(\iota_a\otimes m^{l*}_{\bar a}).
$$
From this we see that $\tilde T^*_{kb}\tilde T_{ka}=0$ for $b\ne a$, since
$$
m^{l}_{\bar b}(\iota_Q\otimes m_{\bar{b}}^{r})(u^{\beta \vee*}_{bk} \otimes \iota_Q)(u^{\alpha \vee}_{ak} \otimes \iota_Q)(\iota_Q\otimes m_{\bar{a}}^{r *})m^{l*}_{\bar a}
$$
is a morphism $X_{\bar a}\to X_{\bar b}$ of $Q$-bimodules. Hence $T^*_bT_a=0$ for $b\ne a$, and the theorem is proved.
\ep

\bigskip

\raggedright

\bigskip


\begin{bibdiv}
\begin{biblist}

\bib{MR914742}{article}{
      author={Anantharaman-Delaroche, Claire},
       title={On {C}onnes' property {$T$} for von {N}eumann algebras},
        date={1987},
        ISSN={0025-5513},
     journal={Math. Japon.},
      volume={32},
      number={3},
       pages={337\ndash 355},
      review={\MR{914742 (89h:46086)}},
}

\bib{arXiv:1410.6238}{misc}{
      author={Arano, Yuki},
       title={Unitary spherical representations of {D}rinfeld doubles},
         how={preprint},
        date={2014},
      eprint={\href{http://arxiv.org/abs/1410.6238}{{\tt arXiv:1410.6238
  [math.QA]}}},
        note={to appear in J. Reine Angew. Math.},
         doi={10.1515/crelle-2015-0079},
}

\bib{MR3308880}{book}{
      author={Bischoff, Marcel},
      author={Kawahigashi, Yasuyuki},
      author={Longo, Roberto},
      author={Rehren, Karl-Henning},
       title={Tensor categories and endomorphisms of von {N}eumann
  algebras---with applications to quantum field theory},
      series={Springer Briefs in Mathematical Physics},
   publisher={Springer, Cham},
        date={2015},
      volume={3},
        ISBN={978-3-319-14300-2; 978-3-319-14301-9},
         url={http://dx.doi.org/10.1007/978-3-319-14301-9},
         doi={10.1007/978-3-319-14301-9},
      review={\MR{3308880}},
}

\bib{MR2391387}{book}{
      author={Brown, Nathanial~P.},
      author={Ozawa, Narutaka},
       title={{$C^*$}-algebras and finite-dimensional approximations},
      series={Graduate Studies in Mathematics},
   publisher={American Mathematical Society},
     address={Providence, RI},
        date={2008},
      volume={88},
        ISBN={978-0-8218-4381-9; 0-8218-4381-8},
      review={\MR{MR2391387 (2009h:46101)}},
}

\bib{MR2863377}{article}{
      author={Brugui{{\`e}}res, Alain},
      author={Natale, Sonia},
       title={Exact sequences of tensor categories},
        date={2011},
        ISSN={1073-7928},
     journal={Int. Math. Res. Not. IMRN},
      number={24},
       pages={5644\ndash 5705},
      eprint={\href{http://arxiv.org/abs/1006.0569}{{\tt arXiv:1006.0569
  [math.QA]}}},
         url={http://dx.doi.org/10.1093/imrn/rnq294},
         doi={10.1093/imrn/rnq294},
      review={\MR{2863377}},
}

\bib{MR2355605}{article}{
      author={Brugui{{\`e}}res, Alain},
      author={Virelizier, Alexis},
       title={Hopf monads},
        date={2007},
        ISSN={0001-8708},
     journal={Adv. Math.},
      volume={215},
      number={2},
       pages={679\ndash 733},
      eprint={\href{http://arxiv.org/abs/math/0604180}{{\tt arXiv:math/0604180
  [math.QA]}}},
         url={http://dx.doi.org/10.1016/j.aim.2007.04.011},
         doi={10.1016/j.aim.2007.04.011},
      review={\MR{2355605 (2009b:18006)}},
}

\bib{MR3079759}{article}{
      author={Brugui{{\`e}}res, Alain},
      author={Virelizier, Alexis},
       title={On the center of fusion categories},
        date={2013},
        ISSN={0030-8730},
     journal={Pacific J. Math.},
      volume={264},
      number={1},
       pages={1\ndash 30},
      eprint={\href{http://arxiv.org/abs/1203.4180}{{\tt arXiv:1203.4180
  [math.QA]}}},
         url={http://dx.doi.org/10.2140/pjm.2013.264.1},
         doi={10.2140/pjm.2013.264.1},
      review={\MR{3079759}},
}

\bib{MR3238527}{article}{
      author={De~Commer, Kenny},
      author={Freslon, Amaury},
      author={Yamashita, Makoto},
       title={C{CAP} for {U}niversal {D}iscrete {Q}uantum {G}roups},
        date={2014},
        ISSN={0010-3616},
     journal={Comm. Math. Phys.},
      volume={331},
      number={2},
       pages={677\ndash 701},
      eprint={\href{http://arxiv.org/abs/1306.6064}{{\tt arXiv:1306.6064
  [math.OA]}}},
         url={http://dx.doi.org/10.1007/s00220-014-2052-7},
         doi={10.1007/s00220-014-2052-7},
      review={\MR{3238527}},
}

\bib{MR3121622}{article}{
      author={De~Commer, Kenny},
      author={Yamashita, Makoto},
       title={Tannaka-{K}re\u\i n duality for compact quantum homogeneous
  spaces. {I}. {G}eneral theory},
        date={2013},
        ISSN={1201-561X},
     journal={Theory Appl. Categ.},
      volume={28},
       pages={No. 31, 1099\ndash 1138},
      eprint={\href{http://arxiv.org/abs/1211.6552}{{\tt arXiv:1211.6552
  [math.OA]}}},
      review={\MR{3121622}},
}

\bib{MR654325}{article}{
   author={Deligne, Pierre},
   author={Milne, James S.},
   title={Tannakian categories},
   conference={
      title={Hodge cycles, motives, and Shimura varieties},},
   book={
      series={Lecture Notes in Mathematics},
      volume={900},
      publisher={Springer-Verlag, Berlin-New York},},
   date={1982},
   pages={101\ndash 228},
   isbn={3-540-11174-3},
   url={http://www.jmilne.org/math/xnotes/tc.html},
   review={\MR{654325 (84m:14046)}},
}

\bib{MR1316301}{article}{
      author={Evans, David~E.},
      author={Kawahigashi, Yasuyuki},
       title={On {O}cneanu's theory of asymptotic inclusions for subfactors,
  topological quantum field theories and quantum doubles},
        date={1995},
        ISSN={0129-167X},
     journal={Internat. J. Math.},
      volume={6},
      number={2},
       pages={205\ndash 228},
         url={http://dx.doi.org/10.1142/S0129167X95000468},
         doi={10.1142/S0129167X95000468},
      review={\MR{1316301 (96d:46080)}},
}

\bib{MR1642584}{book}{
      author={Evans, David~E.},
      author={Kawahigashi, Yasuyuki},
       title={Quantum symmetries on operator algebras},
      series={Oxford Mathematical Monographs},
   publisher={The Clarendon Press Oxford University Press},
     address={New York},
        date={1998},
        ISBN={0-19-851175-2},
        note={Oxford Science Publications},
      review={\MR{1642584 (99m:46148)}},
}

\bib{MR3447719}{article}{
      author={Ghosh, Shamindra~Kumar},
      author={Jones, Corey},
       title={Annular representation theory for rigid {$C^*$}-tensor
  categories},
        date={2016},
        ISSN={0022-1236},
     journal={J. Funct. Anal.},
      volume={270},
      number={4},
       pages={1537\ndash 1584},
      eprint={\href{http://arxiv.org/abs/1502.06543}{{\tt arXiv:1502.06543
  [math.OA]}}},
         url={http://dx.doi.org/10.1016/j.jfa.2015.08.017},
         doi={10.1016/j.jfa.2015.08.017},
      review={\MR{3447719}},
}

\bib{MR1749868}{article}{
      author={Hayashi, Tomohiro},
      author={Yamagami, Shigeru},
       title={Amenable tensor categories and their realizations as {AFD}
  bimodules},
        date={2000},
        ISSN={0022-1236},
     journal={J. Funct. Anal.},
      volume={172},
      number={1},
       pages={19\ndash 75},
         url={http://dx.doi.org/10.1006/jfan.1999.3521},
         doi={10.1006/jfan.1999.3521},
      review={\MR{1749868 (2001d:46092)}},
}

\bib{MR1644299}{article}{
   author={Hiai, Fumio},
   author={Izumi, Masaki},
   title={Amenability and strong amenability for fusion algebras with
   applications to subfactor theory},
   journal={Internat. J. Math.},
   volume={9},
   date={1998},
   number={6},
   pages={669--722},
   issn={0129-167X},
   review={\MR{1644299 (99h:46116)}},
   doi={10.1142/S0129167X98000300},
}

\bib{MR1782145}{article}{
      author={Izumi, Masaki},
       title={The structure of sectors associated with {L}ongo-{R}ehren
  inclusions. {I}. {G}eneral theory},
        date={2000},
        ISSN={0010-3616},
     journal={Comm. Math. Phys.},
      volume={213},
      number={1},
       pages={127\ndash 179},
         url={http://dx.doi.org/10.1007/s002200000234},
         doi={10.1007/s002200000234},
      review={\MR{1782145 (2002k:46160)}},
}

\bib{MR1321145}{book}{
      author={Kassel, Christian},
       title={Quantum groups},
      series={Graduate Texts in Mathematics},
   publisher={Springer-Verlag},
     address={New York},
        date={1995},
      volume={155},
        ISBN={0-387-94370-6},
      review={\MR{1321145 (96e:17041)}},
}

\bib{MR1257245}{article}{
      author={Longo, Roberto},
       title={A duality for {H}opf algebras and for subfactors. {I}},
        date={1994},
        ISSN={0010-3616},
     journal={Comm. Math. Phys.},
      volume={159},
      number={1},
       pages={133\ndash 150},
         url={http://projecteuclid.org/getRecord?id=euclid.cmp/1104254494},
      review={\MR{1257245 (95h:46097)}},
}

\bib{MR1332979}{article}{
      author={Longo, R.},
      author={Rehren, K.-H.},
       title={Nets of subfactors},
        date={1995},
        ISSN={0129-055X},
     journal={Rev. Math. Phys.},
      volume={7},
      number={4},
       pages={567\ndash 597},
      eprint={\href{http://arxiv.org/abs/hep-th/9411077}{{\tt
  arXiv:hep-th/9411077 [hep-th]}}},
         url={http://dx.doi.org/10.1142/S0129055X95000232},
         doi={10.1142/S0129055X95000232},
        note={Workshop on Algebraic Quantum Field Theory and Jones Theory
  (Berlin, 1994)},
      review={\MR{1332979 (96g:81151)}},
}

\bib{MR1444286}{article}{
      author={Longo, R.},
      author={Roberts, J.~E.},
       title={A theory of dimension},
        date={1997},
        ISSN={0920-3036},
     journal={$K$-Theory},
      volume={11},
      number={2},
       pages={103\ndash 159},
      eprint={\href{http://arxiv.org/abs/funct-an/9604008}{{\tt
  arXiv:funct-an/9604008 [math.FA]}}},
         url={http://dx.doi.org/10.1023/A:1007714415067},
         doi={10.1023/A:1007714415067},
      review={\MR{1444286 (98i:46065)}},
}

\bib{MR1712872}{book}{
      author={Mac~Lane, Saunders},
       title={Categories for the working mathematician},
     edition={Second},
      series={Graduate Texts in Mathematics},
   publisher={Springer-Verlag, New York},
        date={1998},
      volume={5},
        ISBN={0-387-98403-8},
      review={\MR{1712872 (2001j:18001)}},
}

\bib{MR1742858}{article}{
      author={Masuda, Toshihiko},
       title={Generalization of {L}ongo-{R}ehren construction to subfactors of
  infinite depth and amenability of fusion algebras},
        date={2000},
        ISSN={0022-1236},
     journal={J. Funct. Anal.},
      volume={171},
      number={1},
       pages={53\ndash 77},
         url={http://dx.doi.org/10.1006/jfan.1999.3523},
         doi={10.1006/jfan.1999.3523},
      review={\MR{1742858 (2001f:46093)}},
}

\bib{MR1966524}{article}{
      author={M{\"u}ger, Michael},
       title={From subfactors to categories and topology. {I}. {F}robenius
  algebras in and {M}orita equivalence of tensor categories},
        date={2003},
        ISSN={0022-4049},
     journal={J. Pure Appl. Algebra},
      volume={180},
      number={1-2},
       pages={81\ndash 157},
      eprint={\href{http://arxiv.org/abs/math/0111204}{{\tt
  arXiv:math/0111204}}},
         url={http://dx.doi.org/10.1016/S0022-4049(02)00247-5},
         doi={10.1016/S0022-4049(02)00247-5},
      review={\MR{1966524 (2004f:18013)}},
}

\bib{MR1966525}{article}{
      author={M{\"u}ger, Michael},
       title={From subfactors to categories and topology. {II}. {T}he quantum
  double of tensor categories and subfactors},
        date={2003},
        ISSN={0022-4049},
     journal={J. Pure Appl. Algebra},
      volume={180},
      number={1-2},
       pages={159\ndash 219},
      eprint={\href{http://arxiv.org/abs/math/0111205}{{\tt
  arXiv:math/0111205}}},
         url={http://dx.doi.org/10.1016/S0022-4049(02)00248-7},
         doi={10.1016/S0022-4049(02)00248-7},
      review={\MR{1966525 (2004f:18014)}},
}

\bib{MR3204665}{book}{
      author={Neshveyev, Sergey},
      author={Tuset, Lars},
       title={Compact quantum groups and their representation categories},
      series={Cours Sp{\'e}cialis{\'e}s [Specialized Courses]},
   publisher={Soci{\'e}t{\'e} Math{\'e}matique de France, Paris},
        date={2013},
      volume={20},
        ISBN={978-2-85629-777-3},
      review={\MR{3204665}},
}

\bib{arXiv:1405.6572}{misc}{
      author={Neshveyev, Sergey},
      author={Yamashita, Makoto},
       title={Poisson boundaries of monoidal categories},
         how={preprint},
        date={2014},
      eprint={\href{http://arxiv.org/abs/1405.6572}{{\tt arXiv:1405.6572
  [math.OA]}}},
}

\bib{MR996454}{incollection}{
      author={Ocneanu, Adrian},
       title={Quantized groups, string algebras and {G}alois theory for
  algebras},
        date={1988},
   booktitle={Operator algebras and applications, {V}ol.\ 2},
      series={London Math. Soc. Lecture Note Ser.},
      volume={136},
   publisher={Cambridge Univ. Press},
     address={Cambridge},
       pages={119\ndash 172},
      review={\MR{MR996454 (91k:46068)}},
}

\bib{popa-corr-preprint}{misc}{
      author={Popa, Sorin},
       title={Correspondences},
         how={preprint},
        date={1986},
        note={INCREST preprint},
        url={http://www.math.ucla.edu/~popa/popa-correspondences.pdf}
}

\bib{MR1302385}{article}{
      author={Popa, Sorin},
       title={Symmetric enveloping algebras, amenability and {AFD} properties
  for subfactors},
        date={1994},
        ISSN={1073-2780},
     journal={Math. Res. Lett.},
      volume={1},
      number={4},
       pages={409\ndash 425},
      review={\MR{MR1302385 (95i:46095)}},
}

\bib{MR1334479}{article}{
      author={Popa, Sorin},
       title={An axiomatization of the lattice of higher relative commutants of
  a subfactor},
        date={1995},
        ISSN={0020-9910},
     journal={Invent. Math.},
      volume={120},
      number={3},
       pages={427\ndash 445},
         url={http://dx.doi.org/10.1007/BF01241137},
         doi={10.1007/BF01241137},
      review={\MR{1334479 (96g:46051)}},
}

\bib{MR1729488}{article}{
      author={Popa, Sorin},
       title={Some properties of the symmetric enveloping algebra of a
  subfactor, with applications to amenability and property {T}},
        date={1999},
        ISSN={1431-0635},
     journal={Doc. Math.},
      volume={4},
       pages={665\ndash 744 (electronic)},
      review={\MR{MR1729488 (2001c:46116)}},
}

\bib{MR2215135}{article}{
      author={Popa, Sorin},
       title={On a class of type {${\rm II}_1$} factors with {B}etti numbers
  invariants},
        date={2006},
        ISSN={0003-486X},
     journal={Ann. of Math. (2)},
      volume={163},
      number={3},
       pages={809\ndash 899},
      eprint={\href{http://arxiv.org/abs/math/0209130}{{\tt
  arXiv:math/0209130}}},
         url={http://dx.doi.org/10.4007/annals.2006.163.809},
         doi={10.4007/annals.2006.163.809},
      review={\MR{2215135 (2006k:46097)}},
}

\bib{MR3406647}{article}{
      author={Popa, Sorin},
      author={Vaes, Stefaan},
       title={Representation theory for subfactors, {$\lambda$}-lattices and
  {$\rm C^*$}-tensor categories},
        date={2015},
        ISSN={0010-3616},
     journal={Comm. Math. Phys.},
      volume={340},
      number={3},
       pages={1239\ndash 1280},
      eprint={\href{http://arxiv.org/abs/1412.2732}{{\tt arXiv:1412.2732
  [math.OA]}}},
         url={http://dx.doi.org/10.1007/s00220-015-2442-5},
         doi={10.1007/s00220-015-2442-5},
      review={\MR{3406647}},
}

\bib{MR1822847}{article}{
      author={Schauenburg, Peter},
       title={The monoidal center construction and bimodules},
        date={2001},
        ISSN={0022-4049},
     journal={J. Pure Appl. Algebra},
      volume={158},
      number={2-3},
       pages={325\ndash 346},
         url={http://dx.doi.org/10.1016/S0022-4049(00)00040-2},
         doi={10.1016/S0022-4049(00)00040-2},
      review={\MR{1822847 (2002f:18013)}},
}

\bib{MR1245354}{article}{
      author={Yamagami, Shigeru},
       title={A note on {O}cneanu's approach to {J}ones' index theory},
        date={1993},
        ISSN={0129-167X},
     journal={Internat. J. Math.},
      volume={4},
      number={5},
       pages={859\ndash 871},
         url={http://dx.doi.org/10.1142/S0129167X9300039X},
         doi={10.1142/S0129167X9300039X},
      review={\MR{1245354 (95f:46114)}},
}

\end{biblist}
\end{bibdiv}
\end{document}